\documentclass[noinfoline]{imsart}
\usepackage[utf8]{inputenc}
\usepackage{eurosym}
\usepackage{amstext,amsthm,amsmath}
\usepackage{color}
\usepackage{amssymb,mathtools}
\usepackage[nointegrals]{wasysym} % enthaelt iint-Def die aber auch in amsmath definiert sind
\usepackage{mathrsfs}
\usepackage{pdfsync}
\usepackage{graphicx}
\usepackage{enumerate}
\usepackage[normalem]{ulem}
\usepackage{dsfont,bbm,mathtools,mathrsfs,color,wasysym,caption,pdfsync, pgfplots, tkz-euclide}
\usepackage{tikz, pictexwd}
\usetkzobj{all}
\usepackage[sort]{natbib}

\newtheorem{mytheorem}{Theorem}
\DeclareMathAlphabet{\mathpzc}{OT1}{pzc}{m}{it}
\definecolor{darkgrey}{rgb}{0.75,0.75,0.75}

\makeatletter
\renewcommand{\@fnsymbol}[1]{\@arabic{#1 }}
\makeatother

\newcommand{\dickm}[1]{\text{\boldmath ${#1}$}}

\newtheorem{proposition}{Proposition}[section]
\newtheorem{corollary}[proposition]{Corollary}

\newtheorem{lemma}[proposition]{Lemma}
\theoremstyle{definition}
\newtheorem{definition}[proposition]{Definition}
\newtheorem{remark}[proposition]{Remark}

\numberwithin{equation}{section}
\DeclareMathAlphabet{\mathpzc}{OT1}{pzc}{m}{it}
\definecolor{darkgrey}{rgb}{0.75,0.75,0.75}
\renewcommand{\beta}{r}

\makeatletter
\renewcommand{\@fnsymbol}[1]{\@arabic{#1 }}
\makeatother

\arxiv{math.PR/1402.1769}

\begin{document}
\begin{frontmatter}
  \title{The fixation time of a strongly beneficial allele in a
    structured population} 
  \runtitle{The fixation time of a strongly beneficial allele}
  
  \begin{aug}
    \author{\fnms{Andreas}
      \snm{Greven}\ead[label=e1]{greven@mi.uni-erlangen.de}
      \ead[label=u1,url]{http://www.mi.uni-erlangen.de/$\sim$greven}}$\!\!$,
    \author{\fnms{Peter}
      \snm{Pfaffelhuber}\ead[label=e2]{p.p@stochastik.uni-freiburg.de}
      \ead[label=u2,url]{http://www.stochastik.uni-freiburg.de/homepages/pfaffelh/}}$\!\!$,\\
    \author{\fnms{Cornelia}
      \snm{Pokalyuk}\ead[label=e3]{cornelia.pokalyuk@gmx.de}}$\!\!$
    \, and \author{\fnms{Anton}
      \snm{Wakolbinger}\ead[label=e4]{wakolbin@math.uni-frankfurt.de}
      \ead[label=u4,url]{http://www.math.uni-frankfurt.de/$\sim$ismi/wakolbinger/}}
  
  \runauthor{Greven, Pfaffelhuber, Pokalyuk, Wakolbinger}

  \affiliation{University of Erlangen\thanksmark{m1}$\!\!$, University
    of Freiburg\thanksmark{m2}$\!\!$, \'Ecole Polytechnique
    F\'ed\'erale de Lausanne\thanksmark{m3}$\!\!$ and University of
    Frankfurt\thanksmark{m4}}

  \address{Department Mathematik\\ Friedrich-Alexander University of Erlangen\\
    Cauerstr. 11\\91058 Erlangen\\ Germany\\
    \printead{e1}\\
    %\printead{u1}
  } 
    \address{Abteilung f{\"u}r mathematische Stochastik\\ Albert-Ludwigs
    University of Freiburg\\ Eckerstr. 1\\ 79104 Freiburg\\ Germany\\
    \printead{e2}\\ 
    % \printead{u2}  
    % \bigskip
  }
  \address{School of Life Sciences\\
    \'Ecole polytechnique f\'ed\'erale de Lausanne\\ (EPFL)\\
    and Swiss Institute of Bioinformatics\\
    1015 Lausanne\\
    Switzerland\\
    \printead{e3}}
  \address{Institut für Mathematik\\
    Johann-Wolfgang Goethe-Universität\\
    60054 Frankfurt am Main \\
    Germany\\
    \printead{e4}\\
    %\printead{u4}
  }
\end{aug}

\begin{abstract}
For a beneficial allele which enters a large unstructured
  population and eventually goes to fixation, it is known that the
  time to fixation is approximately $2\log(\alpha)/\alpha$ for a large
  selection coefficient $\alpha$. For a population that is distributed
  over finitely many colonies, with migration between these colonies,
  we detect various regimes of the migration rate $\mu$ for which the
  fixation times have different asymptotics as $\alpha \to \infty$.

 If $\mu$ is of order $\alpha$, the allele fixes (as in the spatially
  unstructured case) in time $\sim 2\log(\alpha)/\alpha$. If $\mu$ is
  of order $\alpha^\gamma, 0\leq \gamma \leq 1$, the fixation time is $\sim (2 +
  (1-\gamma)\Delta) \log(\alpha)/\alpha$, where $\Delta$ is the number of migration steps that are needed to reach all other colonies
  starting from the colony where the beneficial allele appeared. If $\mu =
  1/\log(\alpha)$, the fixation time is $\sim
  (2+S)\log(\alpha)/\alpha$, where $S$ is a random time in a simple
  epidemic model. 
  
  The main idea for our analysis is to combine a new
  moment dual for the process conditioned to fixation with the time
  reversal in equilibrium of a spatial version of Neuhauser and
  Krone's ancestral selection graph.  
\end{abstract}

\begin{keyword}[class=AMS]
\kwd[Primary ]{92D15}
\kwd[; secondary ]{60J80}\kwd{60J85}\kwd{60K37}\kwd{92D10}
\end{keyword}

\begin{keyword}
\kwd{Interacting Wright-Fisher diffusions}
\kwd{ancestral selection graph}
\kwd{branching process approximation}
\end{keyword}

\end{frontmatter}

\section{Introduction}
\sloppy The goal of this paper is the asymptotic analysis of the time
which it takes for a single strongly beneficial mutant to eventually
go to fixation in a spatially structured population. The beneficial
allele and the wildtype will be denoted by $\mathpzc B$ and $\mathpzc
b$, respectively. The evolution of type frequencies is modelled by a
$[0,1]^d$-valued diffusion process $\mathcal X = (\underline
X(t))_{t\ge 0}, \underline X(t) = (X_i(t))_{i=1,\dots,d}$, where $d
\in \{2,3,\ldots\}$ denotes the number of colonies and $X_i(t)$ stands
for the frequency of the beneficial allele $\mathpzc B$ in colony $i$
at time~$t$. The dynamics accounts for resampling, selection and
migration.  The process $\mathcal X$ is started at time $0$ by an
entrance law from $\underline 0 := (0,\ldots,0)$ and is conditioned to
eventually hit $\underline 1 := (1,\ldots, 1)$.

Models of this kind are building blocks for more complex ones that are
used to obtain predictions for genetic diversity patterns under
various forms of selection.  Indeed, together with the strongly
beneficial allele,  neutral alleles at physically linked
genetic loci { also} have the tendency to go to fixation,
provided these loci are not too far from the selective locus under
consideration. This so-called genetic hitchhiking was first modelled
by  \cite{MaynardSmithHaigh1974}.  A synonymous notion is that of a
selective sweep, which alludes to the fact that, after fixation of the
beneficial allele $\mathpzc B$, neutral variation has been swept from
the population.  Important tools were developed from these patterns to
locate targets of selection in a genome and quantify the role of
selection in evolution, see e.g.\ reviews in \cite{Sabeti2006},
   \cite{Nielsen2005}, \cite{Thornton2007}.

   The process of fixation of a strongly beneficial mutant in the {\em
     panmictic} (i.e. unstructured) case has been studied using a
   combination of techniques from diffusion processes and coalescent
   processes in a random background; see
   e.g. \cite{KaplanHudsonLangley1989}, \cite{StephanWieheLenz1992},
   \cite{SchweinsbergDurrett2005},
   \cite{EtheridgePfaffelhuberWakolbinger2006}. However, since the
   analytical tools applied in these papers rely on the theory of {\em
     one-}dimensional diffusion processes, the extension of these
   results to a spatially structured situation is far from
   straight-forward.

The starting point for the tools developed in this paper is the ancestral selection
graph (ASG) of \cite{NeuhauserKrone1997}. This process has been
introduced in order to study the genealogy under models including
selection. Although the ASG can in principle be used for an arbitrary
strength of selection, it has been employed mainly for models of weak
selection, since then the resulting genealogy is close to a neutral
one. However, \cite{pmid19371754} have used the ASG for strong
balancing selection and \cite{PfaffelhuberPokalyuk2013} have shown how
to use the ASG in order to re-derive classical results for selective
sweeps in a panmictic population. In our present work a {\it spatial
version of the ASG} is the tool of choice which carries over from the
panmictic to the structured case, thus extending the techniques
developed in \cite{PfaffelhuberPokalyuk2013} and leading to new
results for the spatially structured case. The key idea here is to employ  the {\it equilibrium ASG} in a ``paintbox representation'' of the
(fixed time) distributions of the type frequency process conditioned to eventual fixation,  and then use {\em time reversal} of the equilibrium ASG
to obtain an object accessible to the asymptotic analysis.

The fixation process in a structured population under selection has
been the object of study before.  \cite{Slatkin1981} and
\cite{Whitlock:2003:Genetics:12807795} give heuristic results and
comparisons to the panmictic case. While the former paper only gives
results for strong selection but very weak migration, the latter study
gives a comparison to the panmictic case and studies the question
which parameters should be used in the panmictic setting in order to
approximate fixation probabilities and fixation times for structured
populations.  In  \cite{KimMaruki2011} the above studies are extended
by analysing in addition the expected heterozygosity of linked neutral
loci in the case of frequent migration for populations structured
according to a circular stepping-stone model, see also Remark~\ref{KM}
below.  \cite{Hartfield2012} gives a more thorough analysis of the
fixation times for large selection/migration ratios in general
stepping-stone populations based on the assumption that in each colony
the beneficial mutation spreads before migrating.

Our
investigation will provide rigorous results on fixation times for
structured populations, and will detect the corresponding {\em regimes of relative migration/selection speed}.
~

\noindent
\paragraph {Outline of the paper.} After introducing the model in
Section~\ref{S:model} we formulate our main results. These concern the
existence of solutions and the structure of the set of solutions of
the system of SDEs specified in our model (Theorem~\ref{t.1}) and the
asymptotics of the fixation times for a strongly beneficial allele
$\mathpzc B$ in a structured population (Theorem~\ref{T2}). For the
panmictic case (i.e. $d=1$), it is well-known that the fixation time,
for a large selection coefficient $\alpha$, is approximately
$2\log(\alpha)/\alpha$. As it turns out, the time-scale of
$\log(\alpha)/\alpha$ applies in our spatial setting as well. However,
population structure may slow down the fixation process. We study this
deceleration for various regimes of the migration rate $\mu$.  A
spatial version of the ancestral selection graph is introduced in
Section~\ref{S:ASG}, and its role in the analysis of the fixation
probability and the fixation time by the method of duality is
clarified. This leads to a proof of Theorem~\ref{t.1} in
Sec.~\ref{proofthm1}, and to the key Proposition~\ref{P:main} which
relates the asymptotic distribution of the fixation time of the
Wright-Fisher system to that of a marked particle system. Based on the
latter, the proof of Theorem~\ref{t.2} is completed in
Sec.~\ref{S:proofs2}.

\section{Model and main results}
\label{S:model}
We consider solutions $\mathcal X = (\underline X(t))_{t\geq 0}$,
$\underline X(t) = (X_1(t),\dots,X_d(t)) \in [0,1]^d$, of the system
of {\em interacting Wright-Fisher diffusions} 
\begin{align}
  \label{eq:SDE0}
  \notag dX_i & = \Big(\alpha X_i(1-X_i) + \mu \sum_{j=1}^d
  b(i,j)(X_{j} - X_i)\Big) dt + \sqrt{\frac{1}{\rho_i}X_i(1-X_i)}
  dW_i,\\
  & \qquad \qquad \qquad \qquad \qquad \qquad \qquad \qquad \qquad
  \qquad \qquad \qquad i=1,\dots,d
\end{align}
for independent Brownian motions $W_1,\dots,W_d$. Here, $\alpha$ and
$\mu$ are positive constants (the {\em selection} and {\em migration
  coefficient}), and $b(i,j)$, $i, j = 1,\ldots, d$, $i\neq j$, are
non-negative numbers (the {\em backward migration rates}) that
constitute an irreducible rate matrix $\underline{\underline b}$ whose
unique equilibrium distribution has the weights
$\rho_1,\dots,\rho_d$ (which stand for
the relative population sizes of the colonies).  It is well-known (see
e.g.\ \cite{D93}) that the system~\eqref{eq:SDE0} has a unique weak
solution.

Equation \eqref{eq:SDE0}  models the evolution of the relative
frequencies of the beneficial allele at the various colonies, assuming
a {\em migration equilibrium} between the colonies. The ``gene flow'' from colony
$i$ to colony~$j$ is $\rho_i \mu a(i,j)= \rho_j \mu b(j,i) $; here,
$\underline{\underline a} = (a(i,j))$ {with
\begin{align}\label{eq:aij}
  a(i,j)= \frac{\rho_j}{\rho_i} b(j,i)
\end{align}} is the matrix of {\em forward migration rates}.
\begin{remark}[Limit of Moran models\label{rem:Moran}]
  We note in passing that the process $\mathcal X$ arises as the weak
  limit (as $N\to \infty$) of a sequence of structured two-type Moran
  models with $N$ individuals. The dynamics of this Moran model is
  local pairwise resampling with rates $1/\rho_i$, selection with
  coefficient $\alpha$ (i.e.\ offspring from every beneficial line in
  colony $i$ replaces some line in the same colony at rate $\alpha$;
  note that this is the same as selection events which occur at rate
  $s:=\frac{\alpha}{N}$ for each (ordered) pair of particles) and
  migration with rates $\mu a(i,j)$ per line.  Considering now the
  relative frequencies of the beneficial type at the various colonies
  and letting $N\to \infty$ gives \eqref{eq:SDE0}. Here, our
  assumption that $(\rho_i)$ constitutes an equilibrium for the
  migration ensures that we are in a {\em demographic equilibrium}
  with asymptotic colony sizes $\rho_i N$ (otherwise the
  $\rho_i, \, \rho_j$ in the formulas would have to be replaced by
  time-dependent intensities).
\end{remark}
\noindent
We define the fixation time of $\mathcal X$ as
\begin{align}\label{eq:fixationTime}
  T_{\text{fix}}:= \inf\{t>0: \underline X(t) = \underline 1\}.
\end{align}
The fixation probability of the system \eqref{eq:SDE0}, started in
$\underline X(0) = \underline x$, is well-known (see
\cite{Nagylaki:1982:J-Theor-Biol:7169795}). In
Corollary~\ref{c.3.1} we will provide a new proof for the formula
\begin{align} \label{fixprob0} \mathbf P_{\underline
    x}(T_{\text{fix}}<\infty) = \frac{1-e^{-2\alpha (x_1\rho_1 +
      \cdots + x_d\rho_d)}}{1-e^{-2\alpha}}.
\end{align}
Since fixation of the beneficial allele, $\{T_{\text{fix}}<\infty\}$,
is an event in the terminal $\sigma$-algebra of $\mathcal X$,
conditioning on this event leads to an $h$-transform of
\eqref{eq:SDE0} which turns out to be given by the system of SDEs
\begin{align}
  \label{eq:SDE1}
  \notag dX^\ast_i & = \Big(\alpha
  X^\ast_i(1-X^\ast_i)\coth\Big(\alpha \sum_{j=1}^d X^\ast_{j}
  \rho_j\Big) + \mu \sum_{j=1}^d b(i,j)(X^\ast_{j} - X^\ast_i)\Big)
  dt \\ & \qquad\qquad\qquad\qquad \qquad\qquad\qquad\qquad
  \qquad\quad+ \sqrt{\frac{1}{\rho_i}X^\ast_i(1-X^\ast_i)} dW_i
\end{align}
for $i=1,\dots,d$, with $\coth(x) = \frac{e^{2x}+1}{e^{2x}-1}$. The
uniqueness of the solution of ~\eqref{eq:SDE0} carries over to that
of~\eqref{eq:SDE1} as long as $\underline x\neq \underline 0$.  For
$\underline x = \underline 0$, the right hand side of \eqref{eq:SDE1}
is not defined, and we have to talk about \emph{entrance laws from
  $\underline 0$} for solutions of \eqref{eq:SDE1} in this case.

  \begin{definition}[Entrance law from $\underline 0$]
    Let $((\underline X^\ast(t))_{t>0}, \mathbf P)$ with $\underline X^{\ast}(t)  = (X^{\ast}_1(t), ..., X^{\ast}_d(t))$ be a solution of
    \eqref{eq:SDE1} such that $\underline X^\ast(t) \neq \underline 0$
    for $t>0$ and $\underline X^\ast(t) \xrightarrow{t\to 0}
    \underline 0$ in probability. Then, the law of $\underline X^\ast$
    under $\mathbf P$ is called an {\it entrance law from $\underline
      0$} for the dynamics \eqref{eq:SDE1}.
  \end{definition}

\noindent
The following is shown in Section~\ref{proofthm1}.

\begin{mytheorem}
%[Entrance laws from $\underline 0$ of the
%  system~\eqref{eq:SDE1}] 
  \label{t.1} 
  a) For $\underline x \in [0,1]^d\setminus \{\underline 0\}$, the system  \eqref{eq:SDE1}
  has a unique weak solution.\\
  b)  Every entrance law from
  $\underline 0$ is a convex combination of $d$ extremal entrance
  laws from $\underline 0$, which we denote by $ \mathbf
  P^i_{\underline 0}(\mathcal X^\ast \in (.))$, with $(\mathcal
  X^\ast, \mathbf P^i_{\underline 0})$ arising as the limit in
  distribution of $(\mathcal X^\ast, \mathbf P_{\varepsilon \underline
    e_i})$ as $\varepsilon \to 0$, where $\underline e_i$ is the
  vector whose $i$-th component is 1 and whose other components are
  $0$.
%  
%  the
%  set of { entrance laws from
%    $\underline 0$} of~\eqref{eq:SDE1} is convex and has $d$
%  extremal points. The extremal
%  
%  
%  These can be numbered such that the extremal
%  solution with number $\iota$, which we denote by { the
%    distribution of $\mathcal X^\ast$ under some $\mathbf
%    P^\iota_{\underline 0}$}, arises as a limit in distribution of
%  $(\mathcal X^\ast, \mathbf P_{\varepsilon \underline e_\iota})$ as
%  $\varepsilon \to 0$, where $\underline e_i$ is the vector whose
%  component $i$ is $1$ and whose other components are $0$.
\end{mytheorem}

\begin{remark}[Interpretation of the extremal solutions]\label{founder}
  We call $(\mathcal X^\ast, \mathbf P^\iota_{\underline 0})$ the {\em
    solution with the founder in colony $\iota$}. In intuitive terms
  the \label{rem:inter} case $\underline x = \underline 0$ corresponds
  to the beneficial allele $\mathpzc B$ being present in a copy number
  which is too low to be seen in a very large population, i.e.\ on a
  macroscopic level. In this case, since the process is conditioned on
  fixation, there is exactly one individual -- called founder -- which
  will be the ancestor of all individuals at the time of
  fixation. This intuition is made precise in a picture involving
  duality, see Section \ref{secpaint}. The $d$ different entrance
  laws from $\underline 0$ belonging to~\eqref{eq:SDE1} correspond to
  the $d$ different possible geographic locations of the founder. 
  % If
  % the founder is in colony $\iota$,  { due to
  %   exchangeablitity we can assume that is wins} all neutral
  % resampling events within its colony due to the conditioning.
  % %This results in a linear increase of the
  % %$\mathpzc B$-allele in this colony, and thus the set of zeros of
  % %$(X_{\iota}(t))_{t\geq 0}$ has no accumulation point at $t=0$. 
  % In the sequel, we will address this solution of~\eqref{eq:SDE1} as
  % the one with the founder being in colony~$\iota$.
\end{remark}

\noindent
Before stating  our main result on the fixation time of the
system~\eqref{eq:SDE1} we fix some notation and formulate one more
definition.

\begin{remark}[Notation]
  To facilitate notation we will use Landau symbols. For functions $f,
  g: \mathbb R\to\mathbb R$, we write\ (i) $f = \mathcal O(g)$ as
  $x\to x_0\in \overline{\mathbb R}$ if $\limsup_{x\to
    x_0}|f(x)/g(x)|<\infty$, (ii) $f\in\Theta(g)$ if { and only if}
  $f\in\mathcal O(g)$ and $g\in\mathcal O(f)$ and (iii) $f\sim g$ as
  $x\to x_0$ if and only if $f(x)/g(x)\xrightarrow{x\to x_0}1$, { (iv)
    $f= o(g)$ as $x\to x_0\in \overline{\mathbb R}$ if $\limsup_{x\to
      x_0}|f(x)/g(x)|=0$.}
  \noindent
  We write $\xRightarrow{}$ for convergence in distribution and
  $\xrightarrow{}_p$ for convergence in probability.
\end{remark}
In the case of a single colony ($d = 1$) we have $T_{\rm{fix}}\sim 2\log
\alpha/\alpha$ as $\alpha \to \infty$.  Indeed, it is well known that in this case
the conditioned diffusion~\eqref{eq:SDE1} can be separated into three
phases \citep{EtheridgePfaffelhuberWakolbinger2006}:  the beneficial allele
$\mathpzc B$ first has to increase up to a (fixed) small
$\varepsilon>0$.  This phase lasts a time
$\sim \log(\alpha)/\alpha$. In the second phase, the frequency increases to
$1-\varepsilon$ in time of order $1/\alpha$ which is short as compared
to the first and third phase. In the third phase, it takes still about
time $\log(\alpha)/\alpha$ until the allele finally fixes in the
population.

\begin{definition}[Two auxiliary epidemic processes]
  Let \label{def:auxProc} $\underline{\underline a}$ be the matrix of
  forward migration rates and let $G=(V,E)$ be the (connected) graph
  with vertex set $1,\dots,d$ and edge set $E:=\{(i,j): a(i,j)>0\}.$
  We need two auxiliary processes in order to formulate our theorem.
  \begin{enumerate}
  \item For $\gamma\in [0,1]$ and $\iota \in \{1,\dots,d\}$, consider
    the (deterministic) process $\mathcal I^{\iota,\gamma} := \mathcal
    I^\iota = (\underline I^\iota(t))_{t\geq 0}$, $\underline
    I^\iota(t) = (I_1^\iota(t),\dots,I_d^\iota(t))$, with state space
    $\{0,1\}^d$ defined as follows: The process starts in
    $I^\iota_j(0)=\delta_{\iota j}, j=1,...,d$. As soon as one
    component ($I_k^\iota$, say) reaches $1$, then after the
    additional time $1-\gamma$ all those components $I^\iota_j$ for
    which $a(k,j)>0$ are set to 1. The fixation time of this process
    will be denoted by
      $$S_{\mathcal I^{\iota,\gamma}}:=\inf\{t\geq 0: \underline
      I^\iota(t)=\underline 1\}.$$
In other words,  $S_{\mathcal I^{\iota,\gamma}} = (1-\gamma)\Delta_\iota$, where $\Delta_\iota$ is the number of steps that are needed to reach all other
vertices of the graph $G$ in a stepwise percolation  starting from~$\iota$.
    An intuitive interpretation is as follows:
    State~1 of a component means that the colony is infected (by the
    beneficial type $\mathpzc B$) and state $0$ means that it is not
    infected. If a colony gets infected (at time $t$, say), then all the neighbouring (not yet infected) colonies get infected precisely at time $t+1-\gamma$.
    % \sloppy For $p\in [0,1]$, consider the (deterministic) process
    % $\mathcal Y^{\iota,p} := \mathcal Y^\iota = (\underline
    % Y^\iota(t))_{t\geq 0}$, $\underline Y^\iota(t) =
    % (Y_1^\iota(t),\dots,Y_d^\iota(t))$, with state space
    % $\{0,1\}^d$
    % for any $\iota \in \{1,\dots,d\}$. Here, state~1 means that
    % the
    % colony is infected (by the beneficial type $\mathpzc B$) and
    % state
    % 0 means that it is not infected. The process starts in
    % $Y^\iota_\iota(0)=1$ and $Y^\iota_j=0$ for $j \neq \iota$,
    % where
    % $\iota$ is the founder colony. As soon as for some $k$ the
    % process
    % $Y^\iota_k$ reaches the value 1, then {a fixed time
    % interval of length $1-p$ after one of the processes $Y_k$
    % reaches 1, all $Y_k$'s} for which $a(k,j)>0$ are set
    % to 1. The fixation time of this process will be denoted by
    % $$S_{\mathcal Y^{\iota,p}}:=\inf\{t\geq 0: \underline
    % Y^\iota(t)=\underline 1\}.$$
  \item \sloppy For any $\iota \in \{1,\dots,d\}$, consider the
    (random) process $\mathcal J^{\iota} =
    (\underline{J}^{\iota}(t))_{t\geq 0}$, $\underline J^\iota(t) =
    (J_1^\iota(t),\dots,J_d^\iota(t))$, with state space
    $\{0,1,2\}^d$. In state 0, the colony is not infected, in state 1
    it is infected but still not infectious, and in 2, it is
    infectious. The initial state is $J^{\iota}_{\iota}(0) =1$ and
    $J^{\iota}_j(0) =0$ for $j \neq \iota$, where $\iota$ is the
    founder colony. Transitions from state 1 to state 2 occur exactly
    one unit of time after entering state 1.
    % In particular, the founder colony
    % $\iota$ changes to state 2 at time 1.
    For $j \neq \iota$, transitions from 0 to 1 occur at rate $2\sum_k
    \rho_k a(k,j)\mathbf 1_{\{J^{\iota}_k=2\}}$.  The fixation time of
    this process will be denoted by
$$
S_{\mathcal J^{\iota}}:=\inf\{t\geq 0: \underline {J}^{\iota}(t)=\underline 2\};
$$
in particular, this time is larger than $1$.
  \end{enumerate}
\end{definition}

Infection in these epidemic processes indicates presence of the
beneficial type. Our second main result quantifies in terms of these processes how various migration rates affect the spread and the fixation time of the beneficial type. 

\begin{mytheorem}[Fixation times of $\mathcal X^\ast$]\label{t.2}
  For $\iota \in \{1,\ldots,,d\}$, let \label{T2} $\mathcal X^\ast =
  (\underline X^\ast(t))_{t\geq 0}$ be the solution of~\eqref{eq:SDE1}
  with $\underline X^\ast(0) = \underline 0$ and with the founder in
  colony $\iota${, see Remark \ref{founder}}. Then, depending
  on the scaling ratio between $\mu$
  and $\alpha$ as $\alpha\to\infty$, we have the following asymptotics
  for the fixation time $T_{\rm fix}$ defined in
  \eqref{eq:fixationTime} (now for $\underline{\mathcal X}^\ast$ in
  place of $\underline{\mathcal X}$):
  \begin{enumerate}
  \item If $\mu \in \Theta(\alpha)$, then
    \begin{align*}%\label{eq:T11}
      \frac{\alpha}{\log\alpha} T_{\rm{fix}}
      \xrightarrow{\alpha\to\infty}_p 2.
    \end{align*}
  \item More generally, if $\mu \in \Theta(\alpha^\gamma)$ for some
    $\gamma\in[0,1]$, then
    \begin{align*}%\label{eq:T12}
      \frac{\alpha}{\log\alpha} T_{\rm{fix}}
      \xrightarrow{\alpha\to\infty}_p S_{\mathcal I^{\iota,\gamma}} +2.
    \end{align*}
  \item If $\mu = \frac{1}{\log\alpha}$, then \vspace{-0.2cm}
    \begin{align*}%\label{eq:T13}
      \frac{\alpha}{\log\alpha}T_{\rm{fix}}
      \xRightarrow{\alpha\to\infty} S_{\mathcal J^\iota} +1.
    \end{align*}
  \end{enumerate}
%  where $S_{\mathcal Y^{\iota,p}}$ and $S_{\mathcal Z^\iota}$ are
%  the fixation times of the epidemic models from Definition
%  \ref{epidemicmodels}
\end{mytheorem}

\begin{remark}\label{interpret}[Interpretation]
  % First, note that the case $\mu = \Theta(\alpha)$ is covered by both,
  % 1.~and 2.~of the Theorem and give the same result in both cases. In
  % particular, the fixation time (on the time-scale
  % $\log(\alpha)/\alpha$) is continuous at $p=1$.
  % ~
  Let us briefly give some heuristics for the three cases of the
  Theorem. The bottomline of our argument is this: Given a colony~$i$
  is already ``infected'' by the beneficial mutant, the most probable
  scenario (as $\alpha \to \infty$) is that the beneficial type in
  colony~$i$ grows until migration exports the beneficial type to
  other colonies which can be reached from colony~$i$. We argue with
  \emph{successful lines}, which are -- in a population undergoing
  Moran dynamics as in Remark~\ref{rem:Moran} -- individuals whose
  offspring are still present at the time of fixation.
  
  For notational simplicity, we discuss here the situation $d=2$
  with the founder of the sweep being in colony~$\iota=1$. The three
  cases allow us to distinguish when the first successful migrant
  (carrying allele $\mathpzc B$ and still having offspring at the time
  of fixation) moves to colony~2.
  \begin{figure}
  \hspace{0cm} (A) $\mu \in \Theta(\alpha^\gamma)$
  \begin{center}
    {\includegraphics[width=12.2cm,clip]{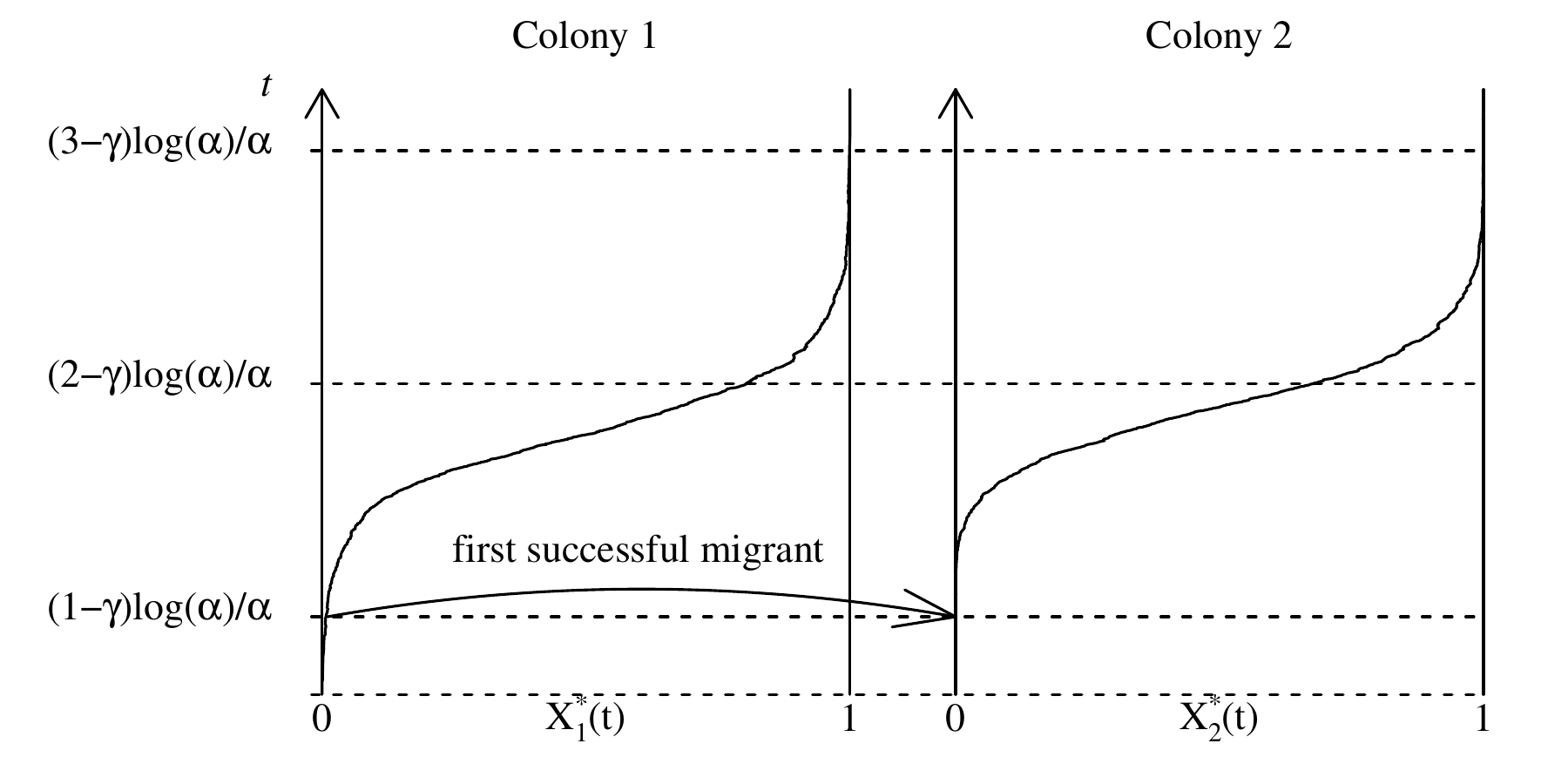}}
  \end{center}
\hspace{0cm} (B) $\mu = 1/(\log(\alpha))$
  \begin{center}
    {\includegraphics[width=12.2cm,clip]{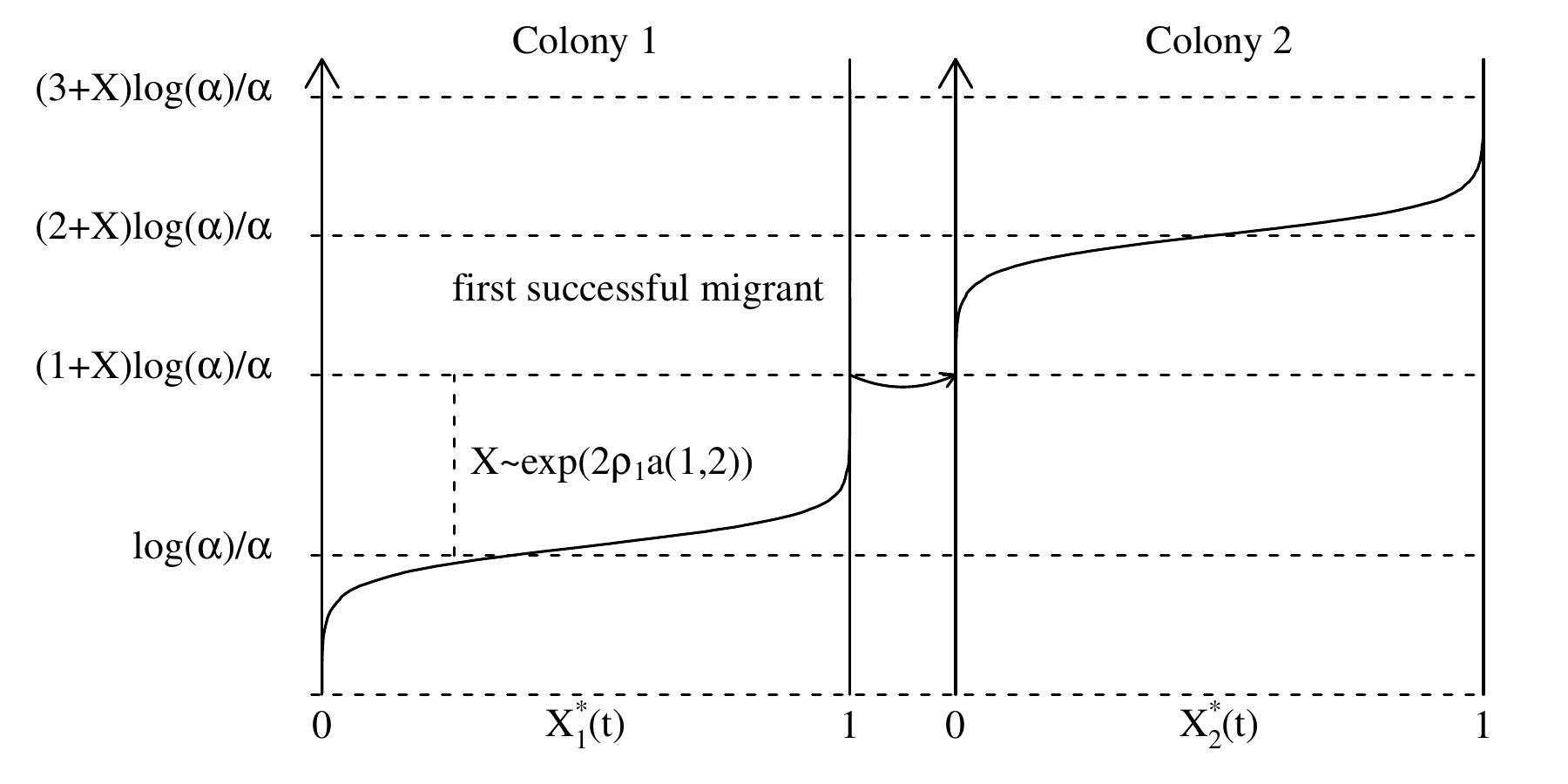}}
  \end{center}
  \caption{\label{fig:thm2} Two examples of a sweep in a structured
    population of $d=2$ colonies. (A) For
    $\mu\in\Theta(\alpha^\gamma)$, the epidemic model $\mathcal{
      I}^{1,\gamma}$ from Theorem~\ref{T2} starts with ${\underline
      I}^1(0)=(1,0)$. The first successful migrant transports the
    beneficial allele to colony~2 at time~$1-\gamma$ (on the
    time-scale $\log(\alpha)/\alpha$). Hence, fixation occurs
    approximately at time $(3-\gamma)\log(\alpha)/\alpha$. (B) For
    $\mu=1/(\log\alpha)$, the epidemic model ${\mathcal J}^1$ from
    Theorem~\ref{T2} starts with ${\underline J}^1(0)=(1,0)$. The
    first successful migrant transports the beneficial allele to
    colony~2 at time $1+X$, where $X$ is an $\exp(2 \rho_1 a(1,2)))$
    distributed waiting time. Then, ${\underline J}^1(1+X)=(2,1)$ and
    thus $S_{{\mathcal J}^1} = 2+X$. From here on, fixation in
    colony~2 takes {one more unit} of time. In total, fixation occurs
    approximately at time $(1+S_{{\mathcal J}^1})\log(\alpha)/\alpha =
    (3+X)\log(\alpha)/\alpha$.
    \\
    For both figures we simulated a Wright-Fisher model, distributed
    on two colonies of equal size, i.e.\
    $a(1,2)=a(2,1)=b(1,2)=b(2,1)=1$ and $\rho_1 = \rho_2 = 1/2$. In
    (A), we used the following parameters: Each colony has size
    $N=10^4$, $m= 0.001$ is the chance that an individual chooses its
    ancestor from the other colony, and $s = 0.01$ is the (relative)
    fitness advantage of beneficials, per generation. This amounts to
    $\gamma = \log(N \cdot s) / \log(N\cdot m) = 2/3$. In (B), we used
    $N=10^5$, $s=0.1$ and $N\cdot m = 1/(\log N\cdot s)$.  }
\end{figure}
  \begin{enumerate}
  \item $\mu\in\Theta(\alpha)$: Since in colony~1 the number of
    successful lines grows like a Yule process with branching
    rate~$\alpha$, migration of the first successful line will occur
    already while the Yule process has $\mathcal O(1)$ lines, i.e.\ at
    a time of order $1/\alpha$ if $\mu\in \Theta(\alpha)$.  From here
    on, the beneficial allele has to fix in both colonies, which
    happens in time $2\log(\alpha)/\alpha$ on each of the colonies.

    We conjecture that this assertion is valid also for the case
    $\mu/\alpha \to \infty$, since intuitively a still higher
    migration rate should result in a panmictic situation due to an
    averaging effect. However, so far our techniques, and in
    particular our fundamental Lemma~\ref{l:numberASGnew}, do not
    cover this case.
  \item $\mu\in\Theta(\alpha^\gamma), 0\ \leq \gamma<1$: Again, the
    question is when the first successful migrant goes to
    colony~2. (In the epidemic model from
    Definition~\ref{def:auxProc}.1, this refers to infection of colony
    2.)  We will argue that this is the case after a time
    $(1-\gamma)\log(\alpha)/\alpha$. Indeed, by this time, the Yule
    process approximating the number of successful lines in colony~1
    has about $\exp(\alpha(1-\gamma)\log(\alpha)/\alpha) =
    \alpha^{1-\gamma}$ lines, each of which travels to colony~2 at
    rate $\alpha^\gamma$, so by that time the overall rate of
    migration to colony~2 is~$\alpha$. More generally, at time $x
    \log(\alpha)/\alpha$, the rate of successful migrants is
    $\alpha^{\gamma+x}$. So, if $\gamma+x<1$, the probability that a
    successful migration happens up to time $x\log(\alpha)/\alpha$ is
    negligible, whereas if $\gamma+x>1$, the probability that a
    successful migration happens up to time $x\log(\alpha)/\alpha$ is
    close to 1. By these arguments, the first successful migration
    must occur around time $(1-\gamma)\log(\alpha)/\alpha$ and the
    time it then takes to fix in colony~2 is again
    $2\log(\alpha)/\alpha$.
  \item $\mu = 1/(\log\alpha)$: Here, migration is so rare that we
    have to wait until almost fixation in colony~1 before a successful
    migrant comes along. Consider the new timescale whose time unit is $\log \alpha / \alpha$, so that migration
    happens at rate $a(1,2)/\alpha$ per individual on this timescale. Roughly, after
    time~1 (in the new timescale), the beneficial allele is almost
    fixed in colony~1. 

    For $N\gg \alpha$, a migrant is successful approximately with
    probability $2\alpha/N$, given by the survival probability of a
    supercritical branching process. So, if one of $N \rho_1$ lines on
    colony~1 migrates, each at rate $a(1,2)/\alpha$, and with the
    success probability being $2\alpha/N$, the rate of successful
    migrants is
    $N \rho_1 \tfrac {a(1,2)}{\alpha} \tfrac{2\alpha}N = 2\rho_1
    a(1,2)$.
    At this rate, the second colony obtains a successful copy of the
    beneficial allele. Thus, in terms of the epidemic model from 2.\
    in Definition~\ref{def:auxProc}, the first colony is infectious if
    allele $\mathpzc B$ is almost fixed there. From the time of the
    first successful migrant on, it takes again time 1 (in the new
    timescale) until the beneficial allele almost fixes in colony
    2. This is when the state of colony 2 in the epidemic model
    changes from 1 (infected) to 2 (infectious).
  \end{enumerate}
\end{remark}

The proof of Theorem~\ref{T2} is given in Section \ref{S:proofs2}. 

\begin{remark}\label{KM}
 In \cite{KimMaruki2011} (see also \cite{Slatkin1976}), it is derived in a heuristic manner that for $s \ll 1 $ and  $sN= \alpha > \mu = mN \gg 1$  the
 time to the first successful migrant is $ \sim \frac 1 \alpha \log (1+\frac \alpha \mu)$. At least for $\mu \in \Theta(\alpha^\gamma)$, $0\le \gamma\le 1$, this is confirmed by our Theorem 2.
\end{remark}

\begin{remark}[Different strengths of migration\label{rem:different}]
  The key argument mentioned at the beginning of Remark
  \ref{interpret} continues to hold if the migration intensity between
  colonies is not of the same order of magnitude. More precisely,
  assume that the asymptotics of the gene flows as $\alpha \to \infty$
  is of the form $\mu \rho_i a(i,j)=\mu \rho_j b(j,i) \in
  \Theta(\alpha^{\gamma_{ij}})$, where the exponents
  $(\gamma_{ij})_{i,j=1,\dots,d} \in [0,1]^{d\times d}$ may vary with
  $i,j$ (possibly also due to a strongly varying colony size).

  Then colony~$j$ can become infected from neighbouring colonies only
  if one of the neighbouring colonies (i) is infected and (ii) carries
  enough beneficial mutants in order to infect colony~$j$. So again
  the fixation time of the beneficial allele can be computed from
  taking the minimal time it takes to infect all colonies across the
  graph $G$, plus the final phase of fixation of the beneficial
  allele.  Consequently, the epidemic process $\mathcal I^\iota :=
  \mathcal I^{\iota, \gamma}$ from Definition~\ref{def:auxProc} can be
  changed to $\mathcal I^{\iota, \underline{\underline \gamma}}$ as
  follows: As soon as for some $i$ the process $I_i^\iota$ reaches the
  value 1, then after an additional fixed time of length $1-\gamma_{ij}$
  all of the $I_j^\iota$ for which $a(i,j)>0$ are set to~1. 
  
  In the sequel we focus on the case $\gamma_{ij} \equiv \gamma$ of a spatially
  homogeneous asymptotics in order to keep the presentation
  transparent. We emphasise however, that our proofs are designed in a
  way which makes the described generalization feasible.
\end{remark}

\section[The ASG]{The ancestral selection graph}
\label{S:ASG}
A principal tool for the analysis of interacting Wright--Fisher
diffusions with selection is their duality with the ancestral
selection graph (ASG) of Krone and Neuhauser, which we recall in
detail below. The main idea for the proof of Theorems \ref{t.1} and
\ref{t.2} is
\begin{itemize}
\item to obtain via the ASG a duality relationship and a Kingman
  paintbox representation also for the diffusion process
  ${\mathcal X^\ast}$ (i.e. the process conditioned to get absorbed at
  $\underline 1$), and to represent $T_{\text{fix}}$ via duality,
\item to show how the equilibrium ASG and its time-reversal can be
  employed for asymptotic calculations as $\alpha \to \infty$.
\end{itemize}
This structure allows us to use the techniques of (multidimensional)
birth-death processes in order to perform the asymptotic analysis
using bounds based on sub- and supercritical branching processes.

In the present section we will focus on the two bullet points, while
the asymptotic analysis of the birth-death processes is in Section
\ref{S:proofs2}, with the basic heuristics in Section \ref{sec4.1}. To
carry out this program we proceed as follows:

In Section~\ref{S30} we will give an informal description of the ASG
and present some of the central ideas of the subsequent proofs. We
will also state a key proposition (Proposition~\ref{P:main}) which
gives a connection between the fixation time and a two-dimensional
birth-and-death process that describes the percolation of the
beneficial type within the equilibrium ASG.  We give a formal
definition of the structured ASG {via a particle
  re\-presentation} in Section \ref{ASGdef} and derive a time-reversal
property in Section \ref{timerev}, which will be important in the
proof of Proposition~\ref{P:main}. In the subsequent sections we will
derive paintbox representations for the solutions of \eqref{eq:SDE0}
and \eqref{eq:SDE1} {using the duality relationships
  from above}, and complete the proofs of Proposition~\ref{P:main} and
Theorem \ref{t.1}.

\subsection{Outline of proof strategy and a key proposition}
\label{S30}
The basic tool for proving Theorems \ref{t.1} and \ref{t.2} will be a representation of $\underline X^\ast(\tau)$ (the solution of \eqref{eq:SDE1} at a fixed time $\tau$) 
in terms of an exchangeable particle system. This representation is first achieved for initial
conditions $\underline x \in [0,1]^d \setminus \{\underline 0\}$, and then also for the entrance laws from $\underline 0$. At the heart
of the construction is a {\em conditional duality} which extends the classical duality between the (unconditioned) $\mathcal X$ (the solution
of \eqref{eq:SDE0}) and the {\em structured ancestral selection graph}.
%Before we enter into the formal definitions, we will give an informal
%introduction of the main concepts, and an explanation of the central
%role of (and the heuristics behind) Proposition~\ref{P:main}.
%
%The {\em (structured) ASG back from time $\tau$}
%% (see also Figure~\ref{fig2} for an illustration)

The latter is constructed in terms of a branching-coalescing-migrating
system {$\mathcal A = (\mathcal A_r)_{r\geq 0}$} of particles, where each
pair of particles in colony~$i$

- {\em coalesces} at rate
$1/\rho_i$, $i=1,\dots,d$, 
\\
and each particle in colony~$i$ 

-  {\em branches} (i.e.\ splits into two) at rate $\alpha$,

-  {\em migrates} (i.e.\ jumps) to colony $j$ at rate $\mu
b(i,j)$. 

When the starting configuration of ${\mathcal A}$ consists of $k_i$ particles in colony~$i$, $i=1,\ldots,
d$, we will speak  of a $\underline
k$-ASG, where
%the trajectories of the particles can be viewed as the {\em
%  potential ancestral lineages} of a sample of {particles} drawn from the
%population, with $k_i$ {particles}
%sampled from colony~$i$. 
%{Whether or not a {\em potential
%    ancestor} is a {\em real ancestor} depends on the constellation of
%  types in the graph.} 
  for brevity we write $\underline k :=
(k_i)_{i=1,\dots,d}$. 
%Informally, we say that ${\mathcal A}_r$ comprises all \emph{potential ancestors} of the $\underline k$-sample by time $r$. 
A~more refined definition of ${\mathcal A}$, which will also allow to speak of a {\em connectedness relation} between particles at
different times, will be given in Sections \ref{ASGdef} and \ref{genrel}. With this refined definition, each particle in ${\mathcal A_r}$ is
represented as a point in $\{1,\ldots,d\}\times [0,1]$, the first component referring to the colony in which the particle is located, and the
second component being a {\em label} which is assigned independently and uniformly at each branching, coalescence and migration event. The ASG
then records the trajectories of all the particles in ${\mathcal A}$, see 
Figure~\ref{figZ}(a) for an illustration.

\begin{figure}
%\begin{subfigure}{\textwidth}
\begin{tikzpicture}
  \node at (0,6){
  \pgftext{ { \selectfont (a)
  } }};

%Achse
\draw[->] (0,5)--(0,0);
%zeichne (durchgezogene) Linie (mit Pfeil) vom Punkt (x_1,y_1)=(0,5) zu Punkt (x_2,y_2)= (0,0)

%Achsenbeschriftung
  \node at (-.25,4.5){
  \pgftext{ { \selectfont 0
  } }};
  
  \node at (-.25,0.5){
  \pgftext{ { \selectfont $\tau$
  } }};

%colony 1

% vertikale Linien

\draw[] (2,4.5)--(2,2);
\draw[] (4,4.5)--(4,3.5);
\draw[] (3,3.5)--(3,2);
\draw[] (5,3.5)--(5,2.5);
\draw[] (3.5,2)--(3.5,0.5);
\draw[] (4.5,1.5)--(4.5,0.5);

%horizontale Linien
\draw[line width= 2pt] (1,4.5)--(6.5,4.5);
\draw[dashed] (3,3.5)--(5,3.5);
\draw[dashed] (2,2)--(3.5,2);
\draw[dashed] (4.5,1.5)--(9.5,1.5);
\draw[dashed] (5,2.5)--(9.3,2.5);
\draw[dashed] (9.8,2.5)--(11,2.5);
\draw[line width = 2pt] (1,0.5)--(6.5,0.5);

%Beschriftung
  \node at (1,5){
  \pgftext{ { \selectfont 0
  } }};
  
    \node at (6.5,5){
  \pgftext{ { \selectfont 1
  } }};
  
    \node at (3.75,0){
  \pgftext{ { \selectfont colony 1
  } }};

%colony 2

% vertikale Linien
\draw[] (9,4.5)--(9,3);
\draw[] (10.5,4.5)--(10.5,3);
\draw[] (11.5,4.5)--(11.5,2);
\draw[] (12.5,4.5)--(12.5,4);
\draw[] (9.5,3)--(9.5,1.5);
\draw[] (11,2.5)--(11,0.5);
\draw[] (10,0.5)--(10,2);
\draw[] (12,4)--(12,2);
\draw[] (13,4)--(13,0.5);

%horizontale Linien 
\draw[line width= 2pt] (8,4.5)--(13.5,4.5);
\draw[dashed] (9,3)--(10.5,3);
\draw[dashed] (12,4)--(13,4);
\draw[dashed] (10,2)--(10.8,2);
% \tkzDefPoint(11,2){O}
%  \tkzDefPoint(10.8,2){A}
% \tkzDrawArc[rotate,dashed](O,A)(180)
\draw[dashed] (11.2,2)--(12,2);
\draw[line width= 2pt] (8,.5)--(13.5,.5);

% Boegen
 \tkzDefPoint(11,2){O}
  \tkzDefPoint(10.8,2){A}
 \tkzDrawArc[rotate,dashed](O,A)(180)

  \tkzDefPoint(9.5,2.5){O}
  \tkzDefPoint(9.3,2.5){A}
 \tkzDrawArc[rotate,dashed](O,A)(180)

%Beschriftung
  \node at (8,5){
  \pgftext{ { \selectfont 0
  } }};
  
    \node at (13.5,5){
  \pgftext{ { \selectfont 1
  } }};
  
    \node at (10.75,0){
  \pgftext{ { \selectfont colony 2
  } }};
 
\end{tikzpicture}
 %\caption{\label{figZ}A realisation of the $\underline k$-ASG in the time interval $[0, \tau]$ with $2$ colonies, and
% $\underline k = (2,4)$. Initially and at each coalescence, branching and migration event, independent and uniform$[0,1]$-distributed
% labels are assigned to the particles, and the {\em genealogical connections} of particles are recorded (visualised by the horizontal dashed lines).}
%\end{subfigure}
%\begin{subfigure}{\textwidth}
\begin{tikzpicture}
\node at (0,7){
  \pgftext{ { \selectfont \textcolor{white}{(b)}
  } }};

\node at (0,6){
  \pgftext{ { \selectfont (b)
  } }};

%Achse
\draw[->] (0,5)--(0,0);
%zeichne (durchgezogene) Linie (mit Pfeil) vom Punkt (x_1,y_1)=(0,5) zu Punkt (x_2,y_2)= (0,0)

%Achsenbeschriftung
  \node at (-.25,4.5){
  \pgftext{ { \selectfont 0
  } }};
  
  \node at (-.25,0.5){
  \pgftext{ { \selectfont $\tau$
  } }};

%colony 1

% vertikale Linien

\draw[] (2,4.5)--(2,2);
\draw[line width= 3pt] (4,4.5)--(4,3.5);
\draw[] (3,3.5)--(3,2);
\draw[line width= 3pt] (5,3.5)--(5,2.5);
\draw[] (3.5,2)--(3.5,0.5);
\draw[] (4.5,1.5)--(4.5,0.5);

%horizontale Linien
\draw[line width= 2pt] (1,4.5)--(6.5,4.5);
\draw[dashed] (3,3.5)--(5,3.5);
\draw[dashed] (2,2)--(3.5,2);
\draw[dashed] (4.5,1.5)--(9.5,1.5);
\draw[dashed] (5,2.5)--(9.3,2.5);
\draw[dashed] (9.8,2.5)--(11,2.5);
\draw[line width = 2pt] (1,0.5)--(6.5,0.5);

%Beschriftung
     \node at (2,5){
  \pgftext{ { \selectfont $\mathpzc b$ 
  } }};

       \node at (4,5){
  \pgftext{ { \selectfont $\mathpzc B$ 
  } }};
  
     \node at (3.5,0){
  \pgftext{ { \selectfont $\mathpzc b$ 
  } }};

       \node at (4.5,0){
  \pgftext{ { \selectfont $\mathpzc b$ 
  } }};
  
    \node at (3.75,-0.5){
  \pgftext{ { \selectfont colony 1
  } }};

%colony 2

% vertikale Linien
\draw[] (9,4.5)--(9,3);
\draw[] (10.5,4.5)--(10.5,3);
\draw[line width= 3pt] (11.5,4.5)--(11.5,2);
\draw[line width= 3pt] (12.5,4.5)--(12.5,4);
\draw[] (9.5,3)--(9.5,1.5);
\draw[line width= 3pt] (11,2.5)--(11,0.5);
\draw[line width= 3pt] (10,0.5)--(10,2);
\draw[line width= 3pt] (12,4)--(12,2);
\draw[] (13,4)--(13,0.5);

%horizontale Linien 
\draw[line width= 2pt] (8,4.5)--(13.5,4.5);
\draw[dashed] (9,3)--(10.5,3);
\draw[dashed] (12,4)--(13,4);
\draw[dashed] (10,2)--(10.8,2);
% \tkzDefPoint(11,2){O}
%  \tkzDefPoint(10.8,2){A}
% \tkzDrawArc[rotate,dashed](O,A)(180)
\draw[dashed] (11.2,2)--(12,2);
\draw[line width= 2pt] (8,.5)--(13.5,.5);

% Boegen
 \tkzDefPoint(11,2){O}
  \tkzDefPoint(10.8,2){A}
 \tkzDrawArc[rotate,dashed](O,A)(180)

  \tkzDefPoint(9.5,2.5){O}
  \tkzDefPoint(9.3,2.5){A}
 \tkzDrawArc[rotate,dashed](O,A)(180)

%Beschriftung
     \node at (9,5){
  \pgftext{ { \selectfont $\mathpzc b$ 
  } }};

     \node at (10.5,5){
  \pgftext{ { \selectfont $\mathpzc b$ 
  } }};

       \node at (11.5,5){
  \pgftext{ { \selectfont $\mathpzc B$ 
  } }};
  
      \node at (12.5,5){
  \pgftext{ { \selectfont $\mathpzc B$ 
  } }};
  
     \node at (11,0){
  \pgftext{ { \selectfont $\mathpzc B$ 
  } }};
  
   \node at (10,0){
  \pgftext{ { \selectfont $\mathpzc B$ 
  } }};
  
    \node at (13,0){
  \pgftext{ { \selectfont $\mathpzc b$ 
  } }};
  
    \node at (10.75,-0.5){
  \pgftext{ { \selectfont colony 2
  } }};
 
\end{tikzpicture}
\caption{\label{figZ} (a) A realisation of the $\underline k$-ASG in
  the time interval $[0, \tau]$ with $2$ colonies, and $\underline k =
  (2,4)$. Initially and at each coalescence, branching and migration
  event, independent and uniform$[0,1]$-distributed
  labels are assigned to the particles, and the {\em genealogical connections} of particles are recorded (visualised by the horizontal dashed lines).\\
  (b) The same realisation of the ASG as in Figure \ref{figZ}(a), now
  showing the particle's types. Two of the five particles in
  ${\mathcal A_\tau}$ are marked with
  $\mathpzc B$. Percolation of type $\mathpzc B$ happens ``upwards''
  along the ASG: all those particles in the $(2,4)$-sample
  ${\mathcal A_0}$ are assigned type
  $\mathpzc B$ which are connected to a type $\mathpzc B$-particle in
  ${\mathcal A_\tau}$.}
\end{figure}
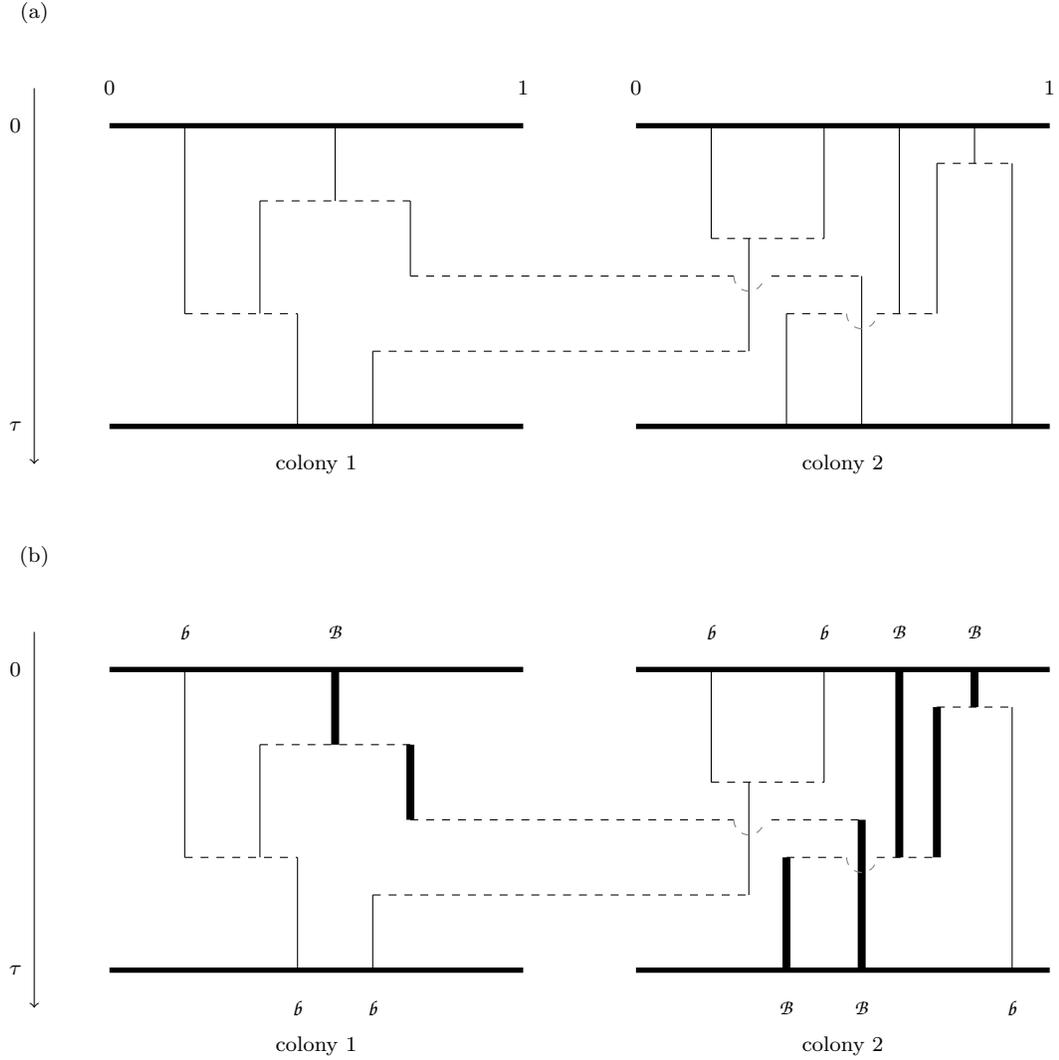

%\begin{figure}
%\includegraphics[width=12cm]{sASG.eps}
%  \caption{\label{figZ}A realisation of the $\underline k$-ASG in the time interval $[0, \tau]$ with $2$ colonies, and $\underline k = (2,4)$. Initially and at each coalescence, branching and migration event, independent and uniform$[0,1]$-distributed labels are assigned to the particles, and the {\em genealogical connections} of particles are recorded (visualised by the horizontal dashed lines).}
%\end{figure}

Writing $K_r^{\underline k}(i)$ for the number of particles in the $\underline k$-ASG  in colony $i$ at time $r$ 
%(cf. the definition in formula  \eqref {eq:Kr}) 
and using the notation
\begin{equation}\label{prodnot}
(\underline 1-\underline y)^{\underline \ell} := \prod_{i=1}^d
  (1-y_i)^{\ell_i}, \quad \underline y =(y_1,\ldots, y_d) \in [0,1]^d, \, \underline \ell = (\ell_1, \ldots, \, \ell_d) \in \mathbb N_0^d,
  \end{equation}
we have a moment duality between $\underline K= (K(i))_{i=1,\ldots,d}$ and the solution $\mathcal X$ of~\eqref{eq:SDE0}:
\begin{equation}\label{basedual}
\mathbf E_{\underline x} [(\underline
    1-\underline X(\tau))^{\underline k}] = \mathbb E [(\underline
    1-\underline x)^{\underline K^{\underline k}_\tau}], \qquad \underline x\in [0,1]^d, \, \underline k\in \mathbb N_0^d, \, \tau \ge 0.
    \end{equation}
Here and in the following, we
denote the probability measure that underlies the particle process ${\mathcal A}$ (and processes related to it) by $\mathbb
P$ (and thus distinguish it from the probability measure $\mathbf
P_{\underline x}$ that underlies the diffusion process $\mathcal X$ 
appearing in~\eqref{eq:SDE0} as well as the corresponding processes, like $\mathcal X^\ast$). Analogously, we use these notation types for the corresponding expectations and variances.
The proof of the {\em basic duality relationship} \eqref{basedual} will be recalled in Lemma \ref{duality}.

Eq.\eqref{basedual} has a conceptual interpretation in population genetics terms: We know that $\underline X(\tau)$ is the vector whose $i$-th coordinate
is the frequency of the beneficial type $\mathpzc B$ in colony $i$ at time $\tau$ when $\underline X(0) = \underline x$. Thus, the left hand side of \eqref{basedual} is
the probability that nobody in a $\underline k$-sample drawn from the population (with $k_i$ individuals drawn from colony $i$, $i=1,\ldots, d$) is of type $\mathpzc B$, given
that $\tau$ time units ago the type frequencies were $\underline x$. 
In the light of a Moran model with selection (whose diffusion limit yields the process $\mathcal X$), the particles' trajectories in the ASG can be interpreted as {\em potential
ancestral  lineages} of the $\underline k$-sample. The type of {a particle} in the
sample can be recovered by a simple rule: it is the beneficial type
$\mathpzc B$ if and only if at least one of its potential ancestors
carries type $\mathpzc B$. In other words, the beneficial type {\em
  percolates upwards} along the lineages of the ASG; see 
Fig.~\ref{figZ}(b) for an illustration.
%\begin{center}
%\includegraphics[width=12cm]{sASGBb.eps}
%\captionof{figure}{The same realisation of the ASG as in Figure \ref{figZ}, now showing the particle's types. Two of the five particles in $\mathcal Z(\tau)$ are marked with $\mathpzc B$. Percolation of type $\mathpzc B$ happens ``upwards'' along the ASG:  all those particles in the $(2,4)$-sample $\mathcal Z_0$ are assigned type $\mathpzc B$ which are connected to a type $\mathpzc B$-particle in $\mathcal Z_\tau$.}\label{figZBb}
%\end{center}
%\begin{figure}[h]
%\includegraphics[width=12cm]{sASGBb.eps}
%  \caption{\label{figZBb}{A realisation of the $\underline k$-ASG in the time interval $[0, \tau]$ with $2$ colonies, and $\underline k = (2,3)$.} Initially and at each coalescence, branching and migration event, independent and uniform$[0,1$-distributed marks are assigned to the particles, and the genealogical {\em connections} of particles are recorded.}
%\end{figure}

Consequently, the event that nobody in the $\underline k$-sample is of type $\mathpzc B$ equals the event that nobody of the sample's potential
ancestors is of type $\mathpzc B$. The probability of this event, however, is just the right
hand side of \eqref{basedual}. Thus, Eq.~\eqref{basedual} expresses the probability of one and the same event in two different ways.   

We will argue in Sec.~\ref{secpaintbox} that the process
${\mathcal A}$ can be started with infinitely many particles in each
colony, with the number of particles immediately coming down from
infinity. This process will be denoted by
${\mathcal A^{\underline \infty}}$. If one marks the particles in
${\mathcal A^{\underline \infty}_\tau}$ independently with
probabilities given by $\underline x$ and lets the types percolate
upwards along the ASG, then one obtains for each
$i\in \{1,\ldots, d\}$ an exchangeable marking of the particles in
${\mathcal A}^{\underline \infty}_0$ that are located in colony
$i$. Let us denote by $F^{\underline x,\tau}_i$ the relative frequency
of the marked particles within all particles of
${\mathcal A}^{\underline \infty}_0$ that are located in colony $i$;
due to de Finetti's theorem, for each $i$, the quantity
$ F^{\underline x,\tau}_i$ exists a.s.  Based on the duality
relationship \eqref{basedual} we will show in Lemma \ref{paintboxrep}
that
$$\mathbf P_{\underline x}(\underline X(\tau)\in (\cdot)) =
\mathbb P(\underline F^{\underline x,\tau}\in (\cdot)), \qquad
\underline x\in [0,1]^d \setminus \{\underline 0\}, \, \tau \ge 0.$$
%i.e. for each $\tau$, the vector $\underline F^{\underline x,\tau}$ has the same distribution as $\underline X(\tau)$.
Following Aldous' terminology (see e.g.\ p.\ 88 in \cite{Aldous1985})
we will call this a ``Kingman paintbox'' representation of $\underline
X(\tau)$.

In order to find a similar representation for
$\underline X^\ast(\tau)$, we will use a coupling of { two processes,
  denoted $\mathcal Z := \mathcal Z^{\underline\infty}$ and
  $\mathcal Y$, which both follow the same dynamics as $\mathcal
  A$.
  Here, $\mathcal Z^{\underline \infty}$ starts with
  $\mathcal Z^{\underline \infty}_0 = \underline\infty$ and
  $\mathcal Y_0$ is an equilibrium configuration of the
  coalescence-branching-migration dynamics described above.}
%   the notion of the {\em equilibrium ASG} turns
%out to be crucial. 
(As we will prove in Proposition \ref{lemeq}, the particle numbers in
equilibrium constitute a Poisson configuration with intensity measure
$(2\alpha\rho_1,\dots,2\alpha\rho_d)$, conditioned to be non-zero.)
{ Since $\mathcal Z$ and $\mathcal Y$ follow the same
  exchangeable dynamics, we can embed both in a single}
  particle system~$\mathcal A$ which starts in
the a.s.  disjoint union $\mathcal A_0:= \mathcal Y_0 \cup \mathcal Z
_0$ and follows the coalescence-branching-migration dynamics.
{Then, $\mathcal Y$ arises by following particles within
  $\mathcal Y_0$ along $\mathcal A$ and $\mathcal Z$ arises by
  following particles within $\mathcal Z_0$ along $\mathcal A$.}

% We will denote the vector of particle numbers in $\mathcal Y_r$ by
% $\underline N_r$, and (as before) the vector of particle numbers in
% $\mathcal Z _r$ by $\underline K_r$.
Let $\mathcal A^{(\underline x)}_\tau$
denote the subsystem of marked particles of $\mathcal A_\tau= \mathcal Y_\tau\cup \mathcal Z_\tau$ which arises
by an independent marking with probabilities $\underline x$. We will prove in  Lemma \ref {dualcondfix} that
\begin{equation*} %\label{basedual}
  \mathbf E_{\underline x} [(\underline
  1-\underline X^\ast(\tau))^{\underline k}] = \mathbb P(\mathcal Z^{\underline k}_\tau\cap \mathcal A^{(\underline x)}_\tau =
  \varnothing | \mathcal Y_\tau \cap  \mathcal A^{(\underline x)}_\tau\neq \varnothing), \qquad \underline x\in [0,1]^d \setminus \{\underline 0\}, \, \underline k\in \mathbb N_0^d, \, \tau \ge 0,
\end{equation*}
with $\mathcal Z^{\underline k}$ started in $\underline k$ particles.
This {\em conditional duality relationship} will be crucial for
deriving the paintbox representation for $\underline
X^\ast(\tau)$.
With the notation $\underline F^{\underline x,\tau}$ introduced above
for the vector of frequencies of the marked particles we will prove in
Lemma \ref{fixfreq} that
$$\mathbf P_{\underline x}(\underline X^\ast(\tau)\in (\cdot)) =
\mathbb P(\underline F^{\underline x,\tau}\in (\cdot)\mid \mathcal
Y_\tau\cap \mathcal A^{(\underline x)}_{\tau}\neq \varnothing), \qquad
\underline x\in [0,1]^d \setminus \{\underline 0\}, \, \tau \ge 0.$$

\begin{figure}
  \hspace{1cm}
  \parbox{5cm}{ \hspace{0.5cm}$\mathcal Y$}  \hspace{1cm}
  \parbox{5cm}{$\mathcal Z^{\underline \infty}$}

  \vspace{-.75cm}

  \begin{center}
    {\includegraphics[width=13cm,trim=0 45 0 0,clip]{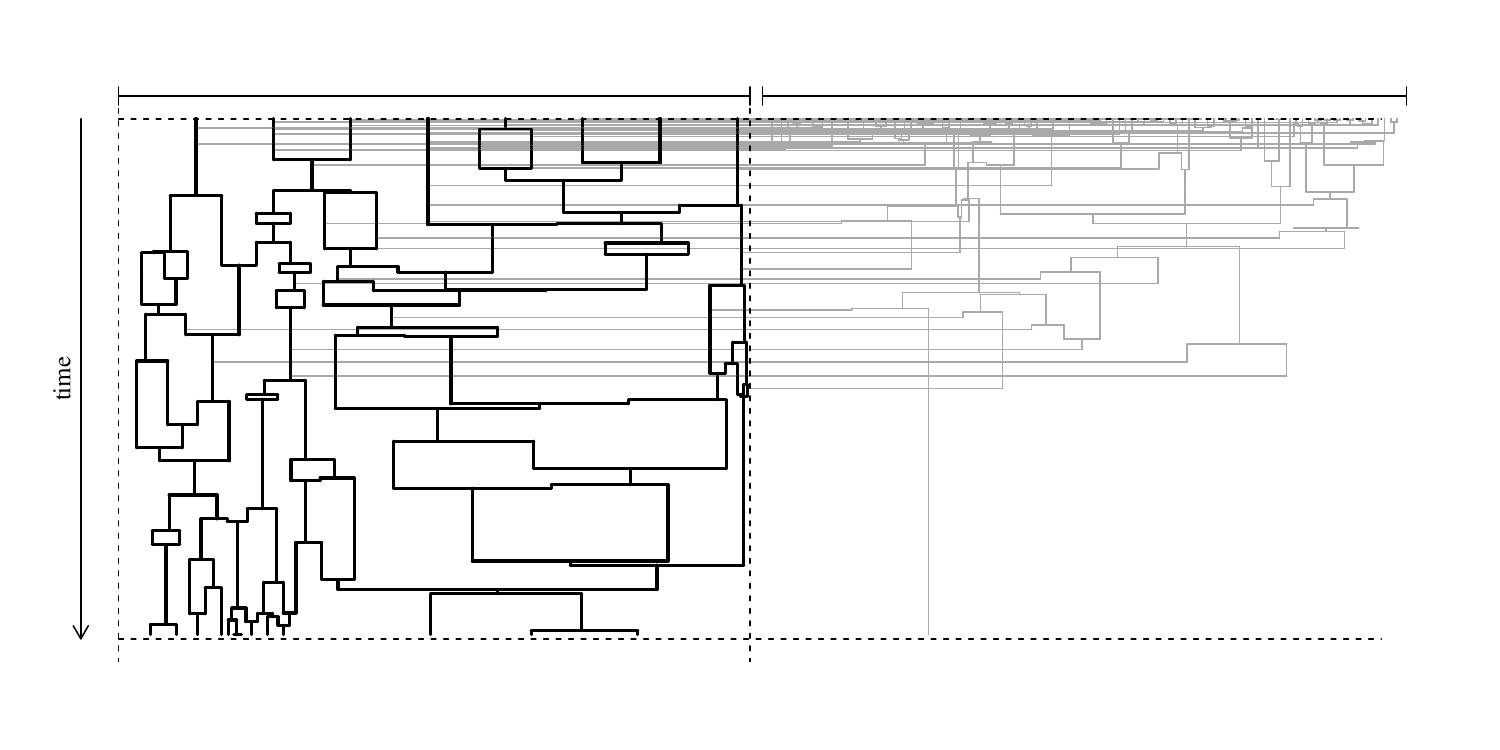}}
  \end{center}
  \caption{\label{fig:asg2} The paintbox representations constructed
    in Section \ref{secpaint} uses two particle systems that are
    coupled to each other.  Initially, these two systems are disjoint,
    and the coupling consists in a (local) coalescence between the two
    ASG's as illustrated in the figure. The potential ancestors of the
    sample on top of the figure are found at the bottom of the
    figure. The bold lines (in the left part of the figure) belong to
    $\mathcal Y$, the gray lines belong to
    $\mathcal Z\setminus\mathcal Y$.}
\end{figure}
  
Let us emphasize that the conditioning under the event $\{\mathcal
Y_\tau \cap \mathcal{A}^{(\underline x)}_\tau\neq \varnothing\}$
affects the distribution of $\mathcal Y$, i.e.\ takes it out of
equilibrium { and changes its dynamics between times $0$
  and $\tau$}. We will denote the vector of particle numbers in
$\mathcal Y_r$ by $\underline N_r$, $r\ge 0$.
%By a slight abuse of terminology
%(and in particular also because this conditioning will not affect the
%dynamics of the time-reversed $\underline\Psi$-ASG of which we will
%make use in Proposition \ref {P:main}), we will continue to speak of
%the $\underline\Psi$-ASG also under this conditioning.
  
Now consider, for some $\iota \in \{1,\ldots, d\}$ and $0< \varepsilon
<1$, the vector $\underline x = \varepsilon \underline e_\iota$,
meaning that initially a fraction $\varepsilon$ of the {particles} in
colony $\iota$ is of beneficial type while all the other colonies
carry only the inferior type $\mathpzc b$. In the limit $\varepsilon
\to 0$ the conditioning under the event $\{{\mathcal Y_\tau \cap
  \mathcal A_\tau^{(\varepsilon \underline e_\iota)}}\neq
\varnothing\}$ amounts to changing the distribution of $\underline
N_\tau$ from its equilibrium distribution to the distribution of
$\underline \Pi +\underline e_\iota$, where $\underline
\Pi$ is Poi($2\alpha\underline \rho$)-distributed, see
Remark~\ref{Poisrem}.  This will result in a paintbox representation for
the distribution of $\mathcal X^{\ast}(\tau)$ under the measure
$\mathbf P^\iota_{\underline 0}$ which appears in Theorem~\ref{t.1},
see Corollary~\ref{corlimit}~a).  The event that, in the system \eqref
{eq:SDE1}, fixation of the beneficial type has occurred by time~$\tau$
can then be reexpressed as the event that the (one) marked particle in
$\mathcal Y_\tau$ is among the potential ancestors of {\em all} the
infinitely many {particles} in $\mathcal Z^{\underline \infty}_0$, see
Corollary~\ref{corlimit} c).

We will show in Lemma \ref{cdinf} and in Corollary~\ref{corfix} that
frequencies within $\mathcal Y$ and $\mathcal Z$ are very close, such
that for
%the time which it takes for $\mathcal Z^{\underline \infty}$ to coalesce with $\mathcal Y$ is
%negligible compared to $\log(\alpha)/\alpha$ as $\alpha \to
%\infty$. Thus, to obtain, for a single mutant entering at colony
%$\iota$ at time $0$ and conditioned to go to fixation, 
the distribution of the fixation time on the $\log
(\alpha)/\alpha$-timescale it will suffice to study the probability
that the marking of a single particle in colony $\iota$ at time $\tau$
percolates ``upwards'' through $\mathcal Y$ in the time interval
$[0,\tau]$.  This analysis is most conveniently carried through in the
{\em time reversal} $\hat {\mathcal Y}$ of $\mathcal Y$, whose
migration rates are reversed as given by Equation \ref{eq:aij}.  The
event $\{{\mathcal Y_\tau \cap \mathcal A_\tau^{(\varepsilon
    \underline e_\iota)} } \neq \varnothing\}$ is the same as
$\{{\hat{\mathcal Y}_0 \cap \mathcal A_0^{(\underline x)} }\neq
\varnothing\}$; thus the conditioning changes the initial condition of
${\hat{ \mathcal{Y}}}$ but not its dynamics { (whereas, as
  mentioned above, the dynamics of $\mathcal Y$, \emph{is} changed by
  the conditioning)}.
  
\noindent
We will write
%  distinguish between those particles in $\hat {\mathcal Y}$ which
%carry the beneficial allele (since they percolate starting from the
%founder of the sweep), whose counting process will be denoted by
$(\underline M_t)_{t\geq 0}$ for the counting process of the marked particles in $(\hat {\mathcal Y}_t)_{t\ge0}$, and $(\underline L_t)_{t\geq 0}$  for the
counting process of {\em all} particles in $(\hat {\mathcal Y}_t)_{t\ge0}$. 
 The dynamics of the bivariate
process $ (\underline L_t)_{t\geq 0}, \underline M_t)_{t\geq 0}) $ is described next, together with the key
result how to use the ASG for approximating the fixation time under
strong selection. Its proof is given in Section~\ref{proofpropo} and
an illustration is given in Figure~\ref{fig:asg1}.
\begin{figure}
  \hspace*{1cm} (A) \hspace{5cm} (B)
  \begin{center}
    {\includegraphics[width=5cm,trim=0 45 0 0,clip]{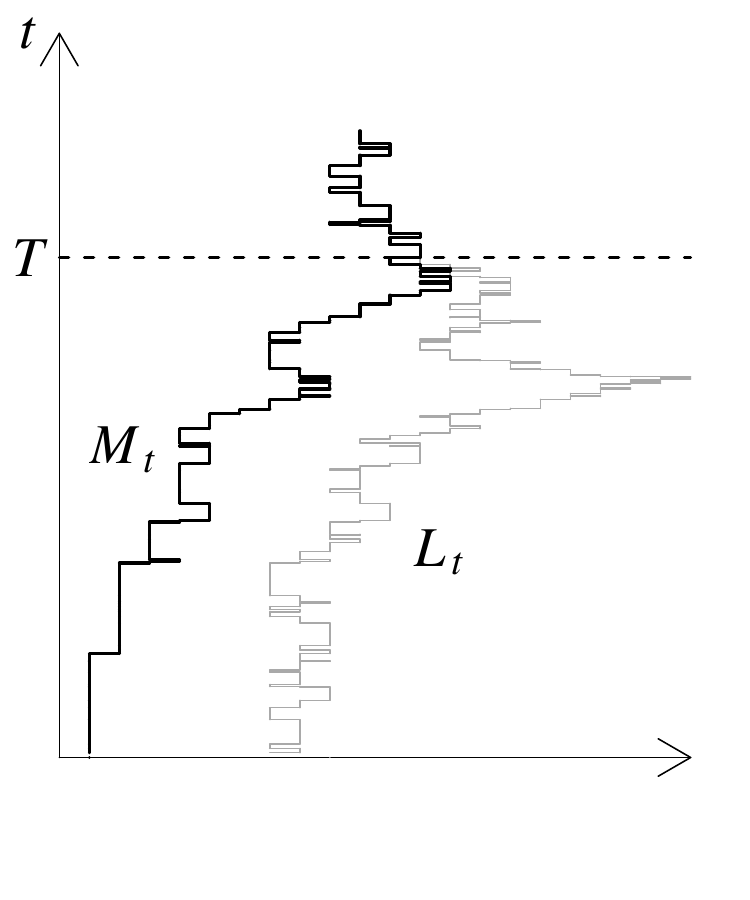}} \qquad
    {\includegraphics[width=5cm,trim=0 45 0 0,clip]{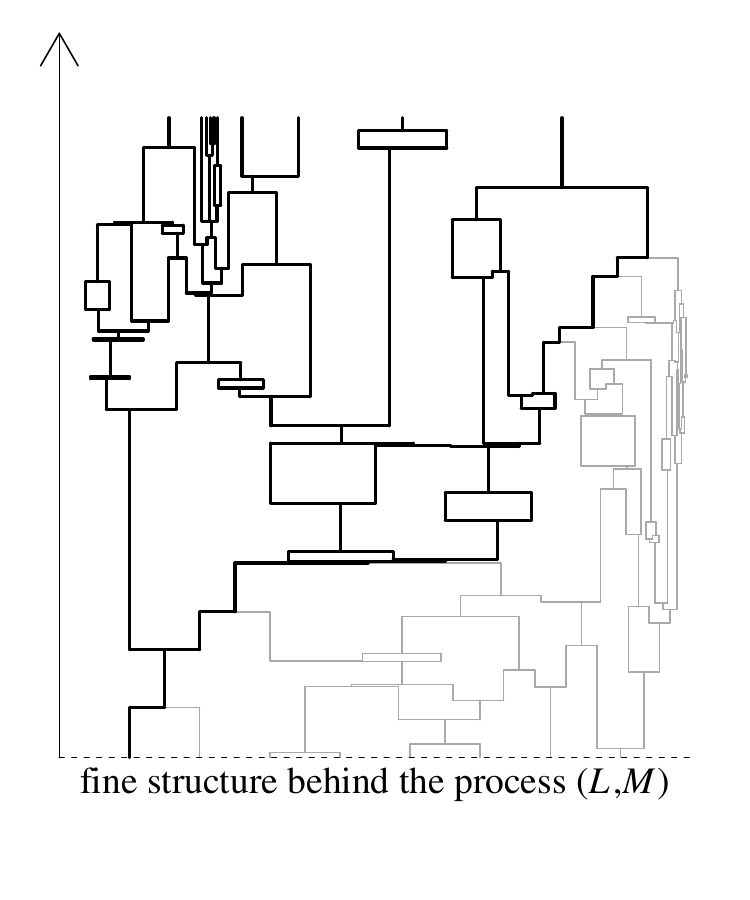}}
  \end{center}
  \caption{\label{fig:asg1}} (A) A realisation of the processes
  $({\underline M}_t)_{t\geq 0}$ and $(\underline L_t)_{t\geq 0}$ for
  the case of one colony. The joint distribution of these two
  processes is given in Proposition~\ref{P:main}. $T$ is the first
  time $t$ when ${\underline M}_t = {\underline L}_t$. (B) The pair
  $(\underline L, \underline M)$ has an underlying structure in terms
  of the particle system $\hat{\mathcal Y}$, where $\underline L$
  arises as the counting process of {\em all} particles in
  $\hat{\mathcal Y}$, and $({\underline M}_t)_{t\geq 0}$ is the
  counting process of the {\em marked particles} in
  $\hat{\mathcal Y}$.
   \end{figure}

\begin{proposition}[An approximation of $T_{\text{fix}}$]
  Let \label{P:main} $(\underline L_t, \underline M_t)$, $\underline L_t=
  (L_t^1,\dots,L_t^d)$,  $\underline M_t=
  (M_t^1,\dots,M_t^d)$, be defined as follows: For fixed $\iota \in
  \{1,\ldots, d\}$, let $\Pi_1,\dots,\Pi_d$ be independent and
  $\text{Poi}(2\alpha \rho_i)$-distributed, and put ${\underline L}_0=
  \underline \Pi + \underline e_\iota$, $\underline M_0 = \underline
  e_{{\iota}}$. The process $(\underline L, \underline M)$ jumps
 from $(\underline \ell, \underline m)$ to
  \begin{align*}
    (\underline \ell + \underline e_i, \underline m + \underline e_i)
    & \text{ at rate } \alpha m_i,\\
    (\underline \ell + \underline e_i, \underline m)
    & \text{ at rate } \alpha (\ell_i-m_i),\\
    (\underline \ell - \underline e_i, \underline m - \underline e_i)
    & \text{ at rate } \frac{1}{\rho_i}
    \binom{m_i}{2},\\
    (\underline \ell - \underline e_i, \underline m) & \text{ at rate
    } \frac{1}{\rho_i}{(\ell_i-m_i)} m_i
    + \frac{1}{\rho_i}\binom{\ell_i - m_i}{2},\\
    (\underline \ell - \underline e_i + \underline e_j, \underline m -
    \underline e_i + \underline e_j)& \text{ at rate }
    \mu a(i,j) m_i,\\
    (\underline \ell - \underline e_i + \underline e_j, \underline m)&
    \text{ at rate } \mu a(i,j) (\ell_i - m_i).
  \end{align*}
  Moreover, let 
  \begin{equation}\label{defT}
  T := \inf\{t\ge 0: \underline M_t =
  \underline L_t\},
  \end{equation}
  and let $T_{\rm{fix}}$ be the fixation time of $\mathcal X^\ast$,
  where $\mathcal X^\ast$ is a solution of the SDE \eqref{eq:SDE1} as
  described in Theorem~\ref{t.1}. Assume that the limiting
  distribution of $\frac \alpha{\log \alpha}T$ exists as
  $\alpha \to \infty$. Then
  \begin{align}\label{mainas}
    \lim_{\alpha\to\infty}\mathbf P_{\underline
      0}^\iota\Big(\frac{\alpha}{\log\alpha}T_{\rm{fix}} \leq
    t\Big) = \lim_{\alpha\to\infty} \mathbb
    P\Big(\frac{\alpha}{\log\alpha} T \leq t\Big),
  \end{align}
  in each continuity point of the limiting distribution
  function. Here, $\mu = \mu(\alpha)$ can depend on $\alpha$ in an
  arbitrary way.
\end{proposition}

\begin{remark}[Existence of limiting distribution]
  Our proof of Theorem \ref{t.2} in Sec. \ref{S:proofs2} will reveal
  in particular that the limiting distribution of
  $\frac \alpha{\log \alpha}T$ exists as $\alpha \to \infty$, at least
  if $\mu = \mu(\alpha)$ falls in one of the three cases of
  Theorem~\ref{t.2}.
\end{remark}

\subsection{The structured ancestral selection graph as a particle system}\label{ASGdef}
We will define a Markov process ${\mathcal A}
= ({\mathcal A_r)_{r \ge 0}}$ that takes its values with probability~1 in the set of finite
subsets of $\{1,\ldots, d\} \times [0,1]$. We shall refer to the
elements of ${\mathcal A_r}$ as {\em
  particles}.  For each particle $(i, u) \in {\mathcal
  A_r}$, we call $i$ the particle's {\em location}
and $u$ the particle's {\em label}. Recall that we denote
the probability measure that underlies ${\mathcal
  A}$ by $\mathbb P$.  It will sometimes be
convenient to annotate the configuration of locations of the initial
state as a  superscript of $\mathcal A$ or $\mathcal
Z$. Specifically, for $\underline k = (k_1,\ldots, k_d)\in \mathbb
N_0^d$, we put
\begin{equation}\label{Zk}
  {\mathcal A}^{\underline k}_0= \bigcup_{i=1}^d\{(i,U_{ig}) :  1\le g\le k_i\}, 
\end{equation}
where the $U_{ig}$ are independent and uniformly distributed on
$[0,1]$.

~

We now specify the Markovian dynamics of ${\mathcal A}$ in terms of its
jump kernel $\mathscr D^{b}$ for some migration kernel
$\underline{\underline b}$ on $\{1,\dots,d\}$. Here we distinguish
three kinds of events (see Figure~\ref{fig:split} for an
illustration):
\begin{itemize}
\item[(1)] Coalescence: for all $i=1,\ldots,d$, every pair of
  particles in colony $i$ is replaced at rate $ 1/\rho_i$ by one
  particle in colony $i$ with a label that is uniformly distributed
  on $[0,1]$ and independent of everything else.
\item[(2)] Branching: for all $i=1,\ldots,d$, every particle in
  colony $i$ is replaced at rate $\alpha$ by two particles in colony
  $i$ with labels that are uniformly distributed on $[0,1]$ and
  independent of each other and of everything else.
\item[(3)] Migration: for all $i=1,\ldots,d$, every particle in
  colony $i$ is replaced at rate $\mu\, b(i,j)$, $j\in \{1,\ldots,
  d\}, j \neq i$, by a particle in colony $j$ with a label that is
  uniformly distributed on $[0,1]$ and independent of everything else.
\end{itemize} 
We will refer to ${\mathcal A} =
({\mathcal A_r})_{r\ge 0}$ also as the
{\em structured ancestral selection graph} (or ASG for short).  The
vector of {\em particle numbers} at time $r$ is $\underline K_r
=(K_r(1),\ldots, K_r(d))$ with
\begin{align}
  \label{eq:Kr}
  K_r(i) := \# \left({\mathcal A_r}\cap
      (\{i\}\times [0,1])\right), \,r \ge 0, \, i=1,\ldots, d.
\end{align}
${\underline K:=}(\underline K_r)_{r\geq 0}$ is a Markov
process whose jump rates (based on the migration kernel
$\underline{\underline b}$) are for $\underline k = (k_1,\dots,k_d)\in
\mathbb N_0^d\setminus \{\underline 0\}$ given by
\begin{equation} \label{eq:defqkl} 
  \begin{aligned}
    q^b_{\underline k, \underline k - \underline e_i} & :=
    q_{\underline k, \underline k - \underline e_i} :=
    \frac{1}{\rho_i}\binom {k_i} 2, \\ q^b_{\underline k, \underline
      k + \underline e_i} & := q_{\underline k, \underline k +
      \underline e_i} := \alpha k_i, \\q^b_{\underline k, \underline
      k -
      \underline e_i + \underline e_j}& :=\mu\, b(i,j) k_i, \\
    q^b_{\underline k, \underline \ell} & := q_{\underline k,
      \underline \ell} := 0 \quad \mbox{otherwise}.
  \end{aligned}
\end{equation}
{By analogy wit{h} 
  the notation ${\mathcal A^{\underline k}}$, we write
  $(\underline K^{\underline k}_r)_{r\geq 0}$ for the process with
  initial state $\underline k$. 
%  for $\underline k \in \mathbbm
%  N^{d}_0\setminus \{\underline 0\}$.
}

\begin{figure}
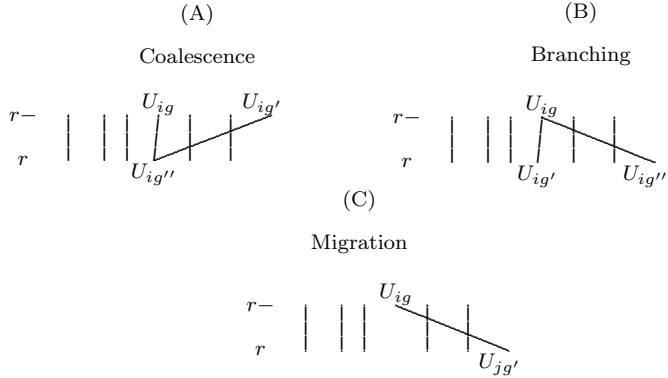

  \parbox{5cm}{\centering{(A)}\\[2ex]Coalescence\\[3ex]
    \beginpicture
    \setcoordinatesystem units <.6cm,.6cm>
    \setplotarea x from 0 to 6, y from -.5 to 1.5
    \put {$r-$} [cC] at 0 1
    \put {$r$} [cC] at 0 0
    \put {$U_{ig}$} [cC] at 3 1.25
    \put {$U_{ig'}$} [cC] at 5.3 1.25
    \put {$U_{ig''}$} [cC] at 2.9 -.3
    \plot 1 0 1 1 /
    \plot 1.8 0 1.8 1 /
    \plot 2.3 0 2.3 1 /
    \plot 3.7 0 3.7 1 /
    \plot 4.6 0 4.6 1 /
    \plot 2.9 0 5.5 1 /
    \plot 2.9 0 3 1 /
    \endpicture}
  \parbox{5cm}{\centering{(B)}\\[2ex]Branching\\[3ex]
    \beginpicture
    \setcoordinatesystem units <.6cm,.6cm>
    \setplotarea x from 0 to 6, y from -.5 to 1.5
    \put {$r-$} [cC] at 0 1
    \put {$r$} [cC] at 0 0
    \put {$U_{ig}$} [cC] at 3 1.25
    \put {$U_{ig''}$} [cC] at 5.3 -.3
    \put {$U_{ig'}$} [cC] at 2.9 -.3
    \plot 1 0 1 1 /
    \plot 1.8 0 1.8 1 /
    \plot 2.3 0 2.3 1 /
    \plot 3.7 0 3.7 1 /
    \plot 4.6 0 4.6 1 /
    \plot 3 1 5.5 0 /
    \plot 2.9 0 3 1 /
    \endpicture}
  \parbox{3cm}{\centering{(C)}\\[2ex]Migration\\[3ex]
    \beginpicture
    \setcoordinatesystem units <.6cm,.6cm>
    \setplotarea x from 0 to 6, y from -.5 to 1.5
    \put {$r-$} [cC] at 0 1
    \put {$r$} [cC] at 0 0
    \put {$U_{ig}$} [cC] at 3 1.25
    \put {$U_{jg'}$} [cC] at 5.3 -.3
    \plot 1 0 1 1 /
    \plot 1.8 0 1.8 1 /
    \plot 2.3 0 2.3 1 /
    \plot 3.7 0 3.7 1 /
    \plot 4.6 0 4.6 1 /
    \plot 3 1 5.5 0 /
    \endpicture}
  \caption{\label{fig:split} If a coalescing event (1), a branching
    event (2) or a migration event (3) occurs by time $r$, we connect
    the lines within the ASG according to the rules as given in
    Section \ref{ASGdef}. In all cases, labels $U_{ig}$ are uniformly
    distributed on $[0,1]$, and are updated upon any event for the
    affected lines.}
\end{figure}

\subsection{Equilibrium and time reversal of the ASG} 
\label{timerev}
\begin{proposition}[Equilibrium for $\mathscr
  D^b$]\mbox{}\label{lemeq}
  \begin{enumerate}
  \item The unique equilibrium distribution $ \pi$ for the dynamics
    $\mathscr D^b$ is the law $\pi$ of a Poisson point process on
    $\{1,\ldots, d\} \times [0,1]$ with intensity measure $2\alpha
    \underline \rho \otimes \lambda$, conditioned to be non-zero
    (where $\underline\rho = (\rho_1,\dots,\rho_d)$ and $\lambda$
    stands for the {Lebesgue
      measure} on $[0,1]$.)
  \item The jump kernel $\hat {\mathscr D}$ of the time reversal of
    ${\mathcal A}$ in its equilibrium $ \pi$ is again of the form
    (1),(2),(3), with the only difference that the migration rates
    $b(i,j)$ are replaced by the migration rates $a(i,j)$ as defined
    in \eqref{eq:aij}, i.e.\ $\hat {\mathscr D} = {\mathscr D^a}$.
  \end{enumerate}
\end{proposition} 

\begin{proof} We will prove the duality relation
  \begin{equation}\label{DDhat}
    \pi (d\underline z) \mathscr D^b(\underline z,d\underline z') =  \pi (d\underline z') 
    {\mathscr D}^{a}(\underline z',d\underline z),
  \end{equation}
  which by well known results about time reversal of Markov chains in
  equilibrium (see e.g.\ \cite{Norris1998}) proves both assertions of
  the Proposition at once.  Since, given the particles' locations,
  their labels are independent and uniformly distributed on $[0,1]$
  and since this is propagated in each of the (coalescence, branching
  and migration) events, it will be sufficient to consider the process
  $\underline K$.  Indeed, defining $q^a_{\underline
    k,\underline\ell}$ as in \eqref{eq:defqkl} and putting
  $$\pi_{(k_1,\dots, k_d)} = \frac{ e^{-2 \alpha}}{1-e^{-2\alpha}}
  \frac{(2\alpha)^{k_1+\cdots+ k_d}}{k_1! \cdots k_d!}  \rho_1^{k_1} \cdots
  \rho_d^{k_d}, \quad \underline k \in \mathbb N_0^d \setminus \{\underline 0\}, $$
  one readily checks for all $\underline k \in \mathbb N_0^d \setminus \{\underline 0\}$
  $$\pi_{\underline k}\cdot q_{\underline k, \underline k - \underline e_i}=\pi_{\underline k - \underline e_i} \cdot q_{\underline
    k - \underline e_i, \underline k}, \quad \pi_{\underline k} \cdot q^b_{\underline k, \underline k -
    \underline e_i + \underline e_j}= \pi_{\underline k - \underline e_i + \underline e_j}
  \cdot q^a_{\underline k - \underline e_i + \underline e_j,
    \underline k}.
  $$
  This can be summarized as
  $$\pi_{\underline k} q^b_{\underline k, \underline\ell} =
  \pi_{\underline\ell} q^a_{\underline\ell, \underline k}, \quad
  \underline k, \underline \ell \in \mathbb N_0^d \setminus
  \{\underline 0\},$$ which by definition of $\mathscr D^{b}$ and
  ${\mathscr D}^{a}$ lifts to \eqref{DDhat}, and thus proves the
  Proposition.
\end{proof}

\subsection{Genealogical relationships in the ASG}\label{genrel}
Thanks to the labelling of the particles it makes sense to speak about genealogical
relationships within ${\mathcal A}$. Doing so will facilitate the
interpretation of the duality relationships in the proofs of
Proposition~\ref{P:main} and Theorem~\ref{t.1}.

\begin{definition}[Connections between particles in ${\mathcal A}$\label{def:33}] Let
  ${\mathcal A}$ follow the dynamics $\mathscr D^b$ described in Section
  \ref{ASGdef}.  We say that a particle $(i',u')$ {\em replaces} a
  particle $(i,u)$ if either of the following relations hold:
  \begin{itemize}
  \item there is a migration event in which $(i,u)$ is replaced by $(i',u')$,
  \item there is a coalescence event for which $(i,u)$ belongs to the
    pair which is replaced by $(i',u')$,
  \item there is a branching event for which $(i',u')$ belongs to the
    pair which replaces $(i,u)$.
  \end{itemize}
  (Note that in the 2nd and 3rd case we have necessarily $i=i'$.) For
  $r,s \ge 0$ we say that two particles $(i,u) \in
  {\mathcal A_{r\wedge s}}$,
  $(i',u')\in {\mathcal A_{r\vee s}}$ are {\em connected} if either $(i,u)=(i',u')$ or there
  exists an $n \in \mathbb N$ and $(i_0,u_0), \ldots, (i_n,u_n)$ such
  that $(i_0,u_0)= (i,u)$, $(i_n,u_n)= (i',u')$, and
  $(i_{\ell},u_{\ell})$ replaces $(i_{\ell-1},u_{\ell-1})$ for
  $\ell=1,\ldots, n$. For any subset $\mathcal S_r$ of
  ${\mathcal A_r}$, let
  $$ \mathscr C_s(\mathcal S_r) := \bigcup_{(i,u) \in \mathcal S_r} \{(i',u') 
  \in {\mathcal A_s}: (i,u) \text{ and
  }(i',u') \text{ are connected}\}$$ be the collection of all those
  particles in ${\mathcal A_s}$ that are
  connected with at least one particle in $\mathcal S_r$.  We briefly
  call $\mathscr C_s(\mathcal S_r)$ {\em the subset of
    ${\mathcal A_s}$ that is connected
    with $\mathcal S_r$}.
\end{definition}

\subsection{Basic duality relationship} We recall a basic duality
result for the ASG for a structured population in Lemma~\ref{duality},
as can e.g.\ be found in \cite[equation (1.5)]{AthreyaSwart2005}. {For
  this purpose we use} a marking procedure of the process
${\mathcal A}$.  

\begin{definition}[A marking of particles\label{def:mark}]
    Let $\mathcal A$ follow the dynamics $\mathscr D^b$ described in
    Section \ref{ASGdef}, and fix a time $\tau >0$.  Take $\underline
    x = (x_1,\ldots, x_d) \in [0,1]^d$, and mark independently all
    particles in colony~$i$ at time $\tau$ with probability
    $x_i$. Denote by
    $$\mathcal A_\tau^{(\underline x)}:=\{(i,u)\in \mathcal A_\tau: (i,u) \mbox{ is marked }\}$$
    the collection of all marked particles in $\mathcal A_\tau$.
    \end{definition}
\begin{remark}[Connectedness and marks\label{rem:coma}]
  In the sequel we will use the following observation: for any subset
  $\mathcal S_0$ of $\mathcal A_0$, 
%  the set of marked particles in
%  $\mathcal S_0$, $\mathcal S_0\cap \mathcal M^{\underline x, \tau}_0$
%  is empty, if and only if the set of marked particles connected to
%  $\mathcal S_0$, given by $\mathscr C_\tau(\mathcal S_0)\cap \mathcal
%  M^{\underline x, \tau}$, is empty. Written as a  formula, 
  $$\mathcal S_0 \cap
  {\mathscr C_0(  \mathcal A_\tau^{(\underline x)}) }=\varnothing \, \mbox{ if and
    only if }\, \mathscr C_\tau(\mathcal S_0) \cap {\mathcal
    A}_\tau^{(\underline x)}=\varnothing.$$ 
    For $\mathcal
  S_0={\mathcal A_0}$, we find that
  $\mathscr C_0(  \mathcal A_\tau^{(\underline x)}) =\varnothing$ if and only if
  $\mathcal A_\tau^{(\underline x)}=\varnothing$.
\\
  In words: no particle in $\mathcal S_0$ is marked (i.e. of
  ``beneficial type''), if and only if no potential ancestral particle
  of $\mathcal S_0$ is marked.
\end{remark}
\begin{lemma}[Basic duality relationship]\label{duality}
  Let $\mathcal X = (\underline X(t))_{t\geq 0}$ be the solution
  of~\eqref{eq:SDE0} with $\underline X(0) = \underline x\in [0,1]^d$,
  and let $\mathcal A$ follow the dynamics $\mathscr D^b$.  Then, for
  all $\underline k = (k_1,\ldots, k_d) \in \mathbb N_0^d$, we have,
  using the notation \eqref{prodnot} and \eqref{eq:Kr} 
  \begin{align} \label{eq:dual} \mathbf E_{\underline x} [(\underline
    1-\underline X(\tau))^{\underline k}] = \mathbb E [(\underline
    1-\underline x)^{\underline K^{\underline k}_\tau}] 
    =
    {\mathbb P({\mathcal A}_\tau^{(\underline x)} =
    \varnothing | \# {\mathcal A}_0 = \underline k)}.
  % = {\mathbb P}_{\underline k}({\mathscr C_0(  \mathcal A_\tau^{(\underline x)}) }= \varnothing) .
  \end{align}
  % where we write $(\underline 1-\underline x)^{\underline k} := \prod_{i=1}^d
  % (1-x_i)^{k_i}$.
\end{lemma}
\begin{proof}
  The generator of the Markov process $\mathcal X$ is given by
  \begin{align*}
    G_{\mathcal X} f(\underline x) & = \tfrac 12 \sum_{j=1}^d
    \frac{1}{\rho_i}x_i(1-x_i) \frac{\partial f^2(\underline
      x)}{\partial^2 x_i} + \alpha \sum_{i=1}^d x_i (1-x_i)
    \frac{\partial f(\underline x)}{\partial x_i} \\ & \qquad \qquad
    \qquad \qquad \qquad \qquad \qquad + \mu \sum_{i,j=1}^d b(i,j)
    (x_j - x_i) \frac{\partial f(\underline x)}{\partial x_i}
  \end{align*}
  for functions $f\in\mathcal C^2([0,1]^d)$. Hence, { for
    $f_{\underline k}(\underline x) := (1-\underline x)^{\underline
      k}$ and $g_{\underline x}(\underline k) := (1-\underline
    x)^{\underline k}$,}
  \begin{align*}
    G_{\mathcal X} {f_{\underline k}(\underline x)}
     & = \sum_{i=1}^d
    \frac{1}{\rho_i} x_i \binom{k_i}{2} (1-\underline x)^{\underline k
      - \underline e_i} + \alpha \sum_{i=1}^d k_i (-x_i)
    (1-\underline x)^{\underline k} \\ & \qquad \qquad \qquad + \mu
    \sum_{i,j=1}^d b(i,j) k_i ((1-x_j) - (1-x_i)) (1-\underline
    x)^{\underline k_i - \underline e_i} \\ & = \sum_{i=1}^d
    \frac{1}{\rho_i}\binom{k_i}{2} \big((1-\underline x)^{\underline k
      - \underline e_i} - (1-\underline x)^{\underline k}\big) \\ &
    \qquad \qquad \qquad + \alpha \sum_{i=1}^d k_i \big(
    (1-\underline x)^{\underline k + \underline e_i} - (1-\underline
    x)^{\underline k}\big) \\ & \qquad \qquad \qquad + \mu
    \sum_{i,j=1}^d b(i,j) k_i \big((1-\underline x)^{\underline k -
      \underline e_i + \underline e_j} - (1-\underline x)^{\underline
      k}\big) \\ & = G_{\underline K} {g_{\underline
        x}(\underline k)},
  \end{align*}
  {where $G_{\underline K}$ is the generator of $\underline
    K$.}
  Now, the first equality in the duality relationship \eqref{eq:dual}
  is straightforward; see \cite[Section 4.4]{EthierKurtz86}.  The
  second equality in \eqref{eq:dual} is immediate from the definition
  of the marking procedure in Definition~\ref{def:mark}.
%   while the
%  third equality is a consequence of Remark~\ref{rem:coma}.
\end{proof}

\subsection {A paintbox representation of 
$\underline X(\tau)$
} 
\label{secpaintbox}
Our next aim is a {\em de Finetti--Kingman paintbox representation} of
the distribution of $\underline X(\tau) $ under $\mathbf P_{\underline
  x}$ in terms of the dual process $\underline K^{\underline \infty}$.
In order to achieve this, we need to be able to start the ASG with
infinitely many lines and define frequencies of marked particles.

\begin{remark}[Asymptotic frequencies]\label{deFinetti}\mbox{}
  \begin{enumerate}
  \item The process ${\mathcal A}$ can be
    started from
    \begin{equation}\label{Zinfty}
      {\mathcal A^{\underline \infty}_0}
     = \bigcup_{i=1}^d\{(i,U_{ig})\} :  1\le g< \infty\},
    \end{equation}
    where $(U_{ig})_{i=1,\dots,d, g=1,2,\dots}$ is an independent family
    of uniformly distributed random variables on $[0,1]$.  Indeed, the
    quadratic death rates {of} the process
    $\underline K$ (recall this process from~\eqref{eq:Kr}) ensure
    that the number of particles comes down from infinity.  In order
    to see this, consider the process $(K_r^1 + \cdots +
    K_r^d)_{r\geq 0}$ and note that given $K_r^1 +
    \cdots + K_r^d=k$ it increases at rate $\alpha k$ and its
    rate of decrease is minimal if colony $i$ carries $\rho_i k$
    lines, $i=1,\dots,d$, hence is bounded from below by
    \begin{align}
      \label{eq:below}
      \sum_{i=1}^{d} \frac{1}{\rho_i} {{k_i}
      \choose {2}} \geq \frac{1}{2} \Bigg( \sum_{i=1}^{d} {k_i}^2 -k
      \Bigg) \geq \frac{1}{2} \Bigg( \frac{1}{d} k^2 -k \Bigg) \geq
      \frac{k(k-d)}{2d},
    \end{align}
    where we have used the Cauchy--Schwarz inequality in the second
    "$\geq$". Using the same bounds as in Proposition~6.9 of
    \cite{DGP12}, we see that
    $\#\mathcal A^{\underline\infty}_\varepsilon =
    O(\varepsilon^{-1})$ as $\varepsilon\to 0$.
  \item For $i=1,\ldots, d$, let $(J_{i1},J_{i2},\ldots) := ((i,
    U_{i1}), (i, U_{i2}),\ldots)$ be the (numbered) collection of
    particles in ${\mathcal A_0^{\underline
        \infty}}$ that are
    located in colony $i$. Then by definition of the dynamics of
    ${\mathcal A^{\underline \infty}}$, the sequence
    \begin{equation}\label{seqex}
      (\mathbbm 1_{\{J_{i1}\in {\mathscr C_0(  \mathcal A_\tau^{(\underline x)})}\}}, 
      \mathbbm 1_{\{J_{i2}\in {\mathscr C_0(  \mathcal A_\tau^{(\underline x)})}\}},
      \ldots)
    \end{equation} 
    is exchangeable. Thus, by de~Finetti's theorem, the asymptotic
    frequency of ones in this sequence exists a.s., which we denote by
    $\underline F^{\underline x,\tau} = (F_i^{\underline
      x,\tau})_{i=1,\dots,d}$ with
    \begin{align}\label{eq:Fix}
      F_i^{\underline x, \tau} := \lim_{n\to\infty} \frac 1n
      \sum_{j=1}^n \mathbbm 1_{\{J_{ij}\in \mathscr C_0(\mathcal
      A_\tau^{(\underline x)})\}}
    \end{align}
  \end{enumerate}
\end{remark}

\begin{lemma}[Asymptotic frequencies and the solution
  of~\eqref{eq:SDE0}]\label{paintboxrep} For $\underline x \in
  [0,1]^d\setminus\{\underline 0\}$, let $\underline F^{\underline
    x,\tau}$ be as in \eqref{eq:Fix}. Then, for the solution
  $\underline X$ of ~\eqref{eq:SDE0} and $\tau\geq 0$,
  \begin{equation}\label{deFin}
    \mathbb P(\underline F^{\underline{x},\tau}\in (.))= 
    \mathbf P_{\underline x}(\underline X(\tau) \in (.)).
  \end{equation}
\end{lemma}

\begin{proof} From \eqref{Zinfty}, for all $\underline k \in \mathbb
  N_0^d \setminus \{\underline 0\}$, the process ${\mathcal
    A^{\underline k}}$ can be seen as
  embedded in ${\mathcal A^{\underline
      \infty}}$, if we write
  \begin{equation}\label{defupsilon}
    {\mathcal A_0^{\underline k}} 
    := \bigcup_{i=1}^d \{(i, U_{ig}): 
    1\le g \le k_i\} \subset {\mathcal A^{\underline \infty}_0}.
  \end{equation}
  By exchangeability of the sequence \eqref{seqex} and de Finetti's
  theorem (cf. Remark \ref{deFinetti}) we obtain
  \begin{equation}\label{freqmoments}
    \mathbb E [(\underline 1-\underline F^{\underline{x},\tau})^{\underline k}] = 
    \mathbb P({\mathcal A_0^{\underline k}}
    \cap {\mathscr C_0(  \mathcal A_\tau^{(\underline x)})}
    =\varnothing).
  \end{equation}
  Since the right-hand sides of \eqref{freqmoments} and \eqref{eq:dual} are equal, 
  we conclude from
%  process $(\mathscr C_r({\mathcal A^{\underline
%      k}_0}))_{r\ge 0}$ (under
%  $\mathbb P_{\underline \infty}$) has the same distribution as the
%  process $({\mathcal A_r^{\underline k})_{r\ge
%      0}}$ (under $\mathbb
%  P_{\underline k}$) we conclude that
%  $$ \mathbb P_{\underline \infty}({\mathcal A_0^{\underline k}}
% \cap {\mathcal A_0^{(\underline
%      x),\tau}} =\varnothing)= \mathbb P_{\underline k}( {\mathcal
%    A_0^{(\underline x),\tau}}=\varnothing).$$ From this and
%  \eqref{freqmoments} together with 
  Lemma \ref{duality}  that
  $$\mathbb E [(\underline 1-\underline F^{\underline{x},\tau})^{\underline k}] 
  = \mathbf E_{\underline x} [(\underline 1-\underline
  X(\tau))^{\underline k}]$$ which shows \eqref{deFin}, since
  $\underline k \in \mathbb N_0^d \setminus \{\underline 0\}$ was
  arbitrary.
\end{proof}

Under $\mathbb P$ we have
$\underline F^{\underline{x},\tau} = \underline 1$ a.s.\ if and only
if for all $i=1,\ldots, d$ the sequences
$(\mathbbm 1_{\{J_{i1}\in {\mathscr C_0( \mathcal A_\tau^{(\underline
      x)})}\}}, \mathbbm 1_{\{J_{i2}\in {\mathscr C_0( \mathcal
    A_\tau^{(\underline x)})}\}},\ldots)$
consist of ones a.s. Hence the events
$\{\underline F^{\underline{x},\tau} = \underline 1\}$ and
$\{{\mathscr C_0( \mathcal A_\tau^{(\underline x)})}= \mathcal
A_0^{\underline \infty}\}$
are a.s. equal under $\mathbb P$. A fortiori we have
\begin{align}
\label{3.6}
  {\mathbf P_{\underline x}(T_{\rm fix} \le \tau)} = \mathbf P_{\underline x}(\underline X(\tau) =1)= \mathbb P
  ({\mathscr C_0( \mathcal A_\tau^{(\underline x)})}= {\mathcal
  A_0^{\underline\infty}}).
\end{align}
This equality allows to compute the probability of eventual fixation.

\begin{corollary}[Eventual fixation]
  \label{c.3.1}
  The probability for eventual fixation of the beneficial type,
\begin{align}\label{defh}
h(\underline x) := \mathbf P_{\underline x}(T_{\rm fix}< \infty)
\end{align}
  can be represented as (using the notation introduced in Lemma
  \ref{duality})
  \begin{align}\label{fixprob}
    h(\underline x) = 1- \mathbb E \left[(\underline 1 - \underline
      x)^{ \underline{\Psi}}\right],
  \end{align}
  where 
%  {
%  \begin{align}
%    \label{eq:Pi}
%    \text{$\underline \Psi\in \mathbb N_0^d \setminus \{\underline 0\}$
%      is Poisson-distributed with parameter 2$\alpha \underline \rho$
%      conditioned to be non-zero.}
%  \end{align}
%  }
    {
  \begin{align}
    \label{eq:Pi}
    \text{$\underline \Psi\in \mathbb N_0^d \setminus \{\underline 0\}$
      is Poisson(2$\alpha \underline \rho$)-distributed,
      conditioned to be non-zero.}
  \end{align}
  }
  In other words, {$\underline \Psi$} {counts} the number of particles {in colonies $1,\ldots, d$}
 {of} the Poisson point process from Proposition~\ref{lemeq}.  In
  particular, $h(\underline x)$ is given by formula \eqref{fixprob0}.  % \begin{align}\label{eq:pfix}
  %   h(\underline x) & = \frac{1-e^{-2\alpha(x_1 + \cdots +
  %   x_d)}}{1-e^{-2\alphad}}.
  % \end{align}
\end{corollary}

\begin{proof}
  Since $\mathbf P_{\underline x}(T_{\text{fix}}< \infty) = \lim_{\tau
    \to \infty} \mathbf P_{\underline x}(T_{\text{fix}}\le \tau)$, we
  can apply the representation \eqref{3.6}. We have that $\underline
  K^{\underline \infty}_\tau \xRightarrow{\tau \to
    \infty}\underline{\Psi}$, and the probability that $(\underline
  K^{\underline \infty}_r)_{r\geq 0}$ between times $r = 0$ and
  $r = \tau$ has a ``bottleneck'' at which the total number of
  lines equals 1 converges to one{; this was called the
    {\em ultimate ancestor} in \cite{KroneNeuhauser1997}}. Thus, as
  $\tau \to \infty$, the r.h.s.\ of \eqref{3.6} converges to the
  probability that at least one
  {particle} in the configuration
  $\underline{\Psi}$ is marked (provided all the
   {particles} at colony $i$ are
  marked independently with probability $x_i$). This latter
  probability equals the r.h.s.\ of \eqref{fixprob}. To evaluate this
  explicitly, we write for independent $L_i\sim \text{Poi}(2\alpha
  \rho_i)$, $i=1,\dots,d$ and $\underline L = (L_1,\dots,L_d)$, $L=L_1
  + \cdots + L_d$ (see Proposition~\ref{lemeq})
  \begin{align*}
    (1-e^{-2\alpha})h(\underline x) & = (1-e^{-2\alpha}) (1-\mathbb
    E[(1-\underline x)^{\underline{\Psi}}]) \\ & = (1-e^{-2\alpha}) -
    \mathbb E[(1-\underline x)^{\underline L}, \underline L \neq
    \underline 0] \\ & = (1-e^{-2\alpha}) - \mathbb E[(1-\underline
    x)^{\underline L}] + \mathbb P(L=0) \\ & = 1 - \prod_{i=1}^d
    \mathbb E[(1-x_i)^{L_i}] \\ & = 1 - \prod_{i=1}^d
    e^{-2\alpha\rho_i}e^{2\alpha\rho_i(1-x_i)} = 1 - e^{-2\alpha
      (x_1\rho_1 + \cdots + x_d\rho_d)},
  \end{align*}
  i.e.\ we have shown~\eqref{fixprob0}.
\end{proof}

\subsection{A duality conditioned on fixation}
\label{sec:YZ}
The next lemma is the analogue of Lemma \ref{duality} for the
conditioned diffusion $\mathcal X^\ast$ in place of $\mathcal X$.
Here, for $\underline k \in \mathbb N_0^d \setminus \{\underline 0\}$,
we will use the process $\mathcal A$, {\em which follows the dynamics
  $\mathscr D^b$ and has the initial state} $\mathcal Y_0 \cup
\mathcal Z^{\underline k}_0$, where $\mathcal Z^{\underline k}_0$ is
as in {the right hand side of} \eqref{Zk} and $\mathcal
Y_0$ is an equilibrium state for the dynamics $\mathscr D^b$ (as
described in Proposition~\ref{lemeq}) which is independent of
$\mathcal Z^{\underline k}_0$.  Note that this independence guarantees
that, with probability one, all labels are distinct, and hence
$\mathcal Y_0$ is a.s. disjoint from $\mathcal Z^{\underline k}_0$.

{ In terms of $\mathcal A$, we define two processes
  $\mathcal Y$ and $\mathcal Z = \mathcal Z^{\underline k}$, which
  follow the dynamics $\mathscr D^b$ with initial states $\mathcal
  Y_0$ and $\mathcal Z_0$, by setting 
  $$\mathcal Z_s = \mathscr
  C_s(\mathcal Z_0^{\underline k}) \text{ and }\mathcal Y_s = \mathscr
  C_s(\mathcal Y_0), s\geq 0.$$ We emphasize that $\mathcal
  Z^{\underline k} \stackrel d = \mathcal A^{\underline k}$ and
  $\mathcal Y \stackrel d = \mathcal A^{\underline \Psi}$ due to
  exchangeability of particles, hence $\mathcal Z^{\underline k}$ and
  $\mathcal Y$ constitute a coupling of $\mathcal A^{\underline k}$
  and $\mathcal A^{\underline \Psi}$ (with disjoint initial states).}

\begin{lemma}[Duality conditioned on fixation] \label{dualcondfix}
  Under $\mathbf P_{\underline x}$ let $\mathcal X^\ast = (\underline
  X^\ast(t))_{t\geq 0}$ \label{l.3.6} be the solution of
  \eqref{eq:SDE1}, started in $\underline X^\ast(0) = \underline x$.
  Under $\mathbb P$ and for $\underline k \in \mathbb N_0^d \setminus
  \{\underline 0\}$, let $\mathcal A${, $\mathcal Y$ and
    $\mathcal Z=\mathcal Z^{\underline k}$} be as described
  above. Then 
  \begin{align} \label{dualfix} \notag \mathbf E_{\underline
      x}[(\underline 1 - \underline X^\ast(\tau))^{\underline k}] &
    =\mathbb P(\mathcal Z_0^{\underline k} \cap \mathscr C_0({\mathcal
      A_\tau^{(\underline x)}}) =\varnothing \mid \mathcal Y_0 \cap
    {\mathscr C_0( \mathcal A_\tau^{(\underline x)})}\neq \varnothing)
    \\ & = \mathbbm P(\mathcal Z_\tau^{\underline k} \cap \mathcal
    A^{(\underline x)}_\tau { = \varnothing}\mid \mathcal Y_\tau \cap
    \mathcal A_\tau^{(\underline x)} \neq \varnothing).
  \end{align}
\end{lemma}

\begin{proof}
  In view of Remark~\ref{rem:coma} the fixation probability
  \eqref{defh} can be expressed as
\begin{align} \label{reprh} h(\underline x)= \mathbb P( 
      \mathcal Y_0 \cap 
{\mathscr C_0(  \mathcal A_\tau^{(\underline x)})}\neq \varnothing) =  \mathbb P(\mathcal
    Y_\tau \cap \mathcal A_\tau^{(\underline x)} \neq \varnothing), \quad \tau \ge 0.
\end{align}
The second equality in \eqref{dualfix} follows right away from Remark
\ref{rem:coma}. To show the first equality, we set out by writing the
Markovian semigroup of $\underline X^\ast$ as the $h$-transform of the
semigroup of $\underline X$,
\begin{align}
  \mathbf E_{\underline x}[(\underline 1 - \underline
  X^\ast({\tau}))^{\underline k}]  =
  \frac{\mathbf E_{\underline x}[(\underline 1 - \underline
  X({\tau}))^{\underline k},
  T_{\text{fix}}<\infty]}{\mathbf P_{\underline x}(T_{\text{fix}}<\infty)} = \frac{\mathbf E_{\underline x}[(\underline 1 -
  \underline X({\tau}))^{\underline k} \,
  h(\underline X({\tau})]}{h(\underline
  x)}. \label{fracdual}
\end{align}
%    \begin{align}\notag
%    \mathbf E_{\underline x}[(\underline 1 - \underline
%    X^\ast({\tau}))^{\underline k}] & =
%    \frac{\mathbf E_{\underline x}[(\underline 1 - \underline
%      X({\tau}))^{\underline k},
%      T_{\text{fix}}<\infty]}{h(\underline x)} \\ \notag & =
%    \frac{\mathbf E_{\underline x}[(\underline 1 - \underline
%      X({\tau}))^{\underline k} \mathbf
%      P_{\underline
%        X({\tau})}(T_{\text{fix}}<\infty)]}{h(\underline
%      x)} \\ & = \frac{\mathbf E_{\underline x}[(\underline 1 -
%      \underline X({\tau}))^{\underline k} \,
%      h(\underline X({\tau})]}{h(\underline
%      x)}. \label{fracdual}
%  \end{align}
The numerator of the right-hand side of \eqref{fracdual} equals
\begin{align}\notag
  \mathbf E_{\underline x}[(\underline 1 - \underline
  X(\tau))^{\underline k} \, &(1-\mathbb E[(\underline 1-\underline
                               X(\tau))^{\underline{\Psi}}])] 
  \\ &= \mathbf E_{\underline x}
       [(\underline 1 - \underline X(\tau))^{\underline k}] - \mathbb E
       \otimes\mathbf E_{\underline x}[(\underline 1-\underline
       X(\tau))^{\underline{\Psi}+\underline k}]. \label{continue}
\end{align}
Writing $(\underline
K^{\underline k}_r)_{r\ge 0}$, $(\underline N_r)_{r \geq 0}$ and
$(\underline G_r)_{r\ge 0}$ for the processes of particle numbers in
$\mathcal Z^{\underline k}$, $\mathcal Y$ and $\mathcal A$,
respectively, we obtain from the duality
relation~\eqref{eq:dual} that 
\begin{align*}
  \mathbb E \otimes\mathbf E_{\underline x}[(\underline 1-\underline
    X(\tau))^{\underline{\Psi}+\underline k}]&= \mathbbm E [ \mathbbm E
    [ \mathbf E_{\underline x} (\underline 1 - \underline
    X(\tau))^{\underline N_0 + \underline k} | \underline N_0]]] \\ &=
    \mathbbm E [ \mathbbm E [
    (\underline 1 - \underline x)^{\underline G_\tau} | \underline
    G_0]] = \mathbb E[(\underline 1-\underline
    x)^{\underline G_\tau}].
\end{align*}
Hence, again by the duality relation~\eqref{eq:dual} and by
Remark~\ref{rem:coma}, the right hand side of \eqref{continue} is
equal to 
\begin{align*}\mathbbm E [(1- \underline
  x)^{K^{\underline k}_\tau}]-\mathbb E[(\underline 1-\underline
  x)^{ \underline G_\tau}] 
  & = \mathbb P(\mathcal Z^{\underline  k}_\tau\cap
    \mathcal A_\tau^{(\underline x)}=\varnothing)  -
    \mathbb P({\mathcal A_\tau^{(\underline x)}}=\varnothing)\\
  & = \mathbb P(\mathcal Z^{\underline k}_0
    \cap {\mathscr C_0( \mathcal
    A_\tau^{(\underline x)})}
    =\varnothing)- \mathbb P((\mathcal
    Z^{\underline k}_0\cup \mathcal Y_0)
    \cap
    {\mathscr C_0(  \mathcal A_\tau^{(\underline x)})}=\varnothing)\\
  &= \mathbb P(\{\mathcal Z_0^{\underline
    k}\cap {\mathscr C_0( \mathcal
    A_\tau^{(\underline x)})}
    =\varnothing\} \cap \{\mathcal Y_0\cap
    {\mathscr C_0( \mathcal
    A_\tau^{(\underline x)})}\neq
    \varnothing \}).
\end{align*}
Combining this with \eqref{fracdual}, \eqref{continue} and
\eqref{reprh}, we arrive at the first equality in \eqref{dualfix}.
\end{proof}

\subsection{A paintbox representation
for $\underline  X^\ast(\tau)$}
\label{secpaint} 
We now lift the assertion from
Lemma~\ref{paintboxrep} about the paintbox construction of $\underline
X(\tau)$ to $\underline X^\ast(\tau)$.  For this, let the process
$\mathcal A$ follow the dynamics
$\mathscr D^b$ and have the initial state $\mathcal Y_0 \cup \mathcal Z^{\underline
  \infty}_0$, where
$\mathcal Z^{\underline \infty}_0$ is as in \eqref{Zinfty} and
$\mathcal Y_0$ is an equilibrium state for the
dynamics $\mathscr D^{{b}}$ (as described in
Proposition~\ref{lemeq}) which is independent of $\mathcal
Z^{\underline \infty}_0$. Recall from~\eqref{eq:Fix}. the definition of the asymptotic frequencies
$\underline F^{\underline x, \tau} = (F_i^{\underline x,
  \tau})_{i=1,\dots,d}$ of ${\mathscr C_0(  \mathcal A_\tau^{(\underline x)})}$ within $\mathcal A_0$.

\begin{lemma}[A paintbox for $\underline
  X^\ast(\tau)$] \label{fixfreq} 
  Under $\mathbf P_{\underline x}$ let
  $\mathcal X^\ast = (\underline X^\ast(t))_{t\geq 0}$ be the solution
  of \eqref{eq:SDE1}, started in $\underline X^\ast(0) = \underline
  x$. Under $\mathbb P$, let the process $\mathcal
  A$ and the frequencies $\underline F^{\underline
    x, \tau}$ be as above. Then,
  \begin{equation}\label{deFincond}
    \mathbf P_{\underline x}(\underline X^\ast(\tau) \in (.)) =
    \mathbb P(\underline F^{\underline x,\tau}\in (.) 
    \mid \mathcal Y_\tau
    \cap \mathcal A_\tau^{(\underline x)}\neq \varnothing).
  \end{equation}
\end{lemma}

\begin{proof} 
  { For $\mathcal Z_0^{\underline\infty} = \{J_{ig}:=(i, U_{ig}):
    i=1,...,d, g=1,2,...\}$, w}e observe that the
  sequence~\eqref{seqex} is exchangeable under the measure $\mathbb
  P(\cdot \mid \mathcal Y_\tau \cap \mathcal A_\tau^{(\underline
    x)}\neq \varnothing)$, which guarantees the a.s.\ existence of
  $\underline F^{\underline x,\tau}$. We now parallel the argument in
  the proof of Lemma \ref{paintboxrep}:\\
  For each $\underline k \in \mathbb N_0^d \setminus \{\underline
  0\}$, with $\mathcal Z^{\underline k}_0$ is as in { the right hand
    side of} \eqref{Zk}, we have because of exchangeability
  \begin{equation*}%\label{freqmomentscond}
    \mathbb E[(\underline 1-\underline F^{\underline x, 
      \tau})^{\underline k}\mid \mathcal Y_\tau
    \cap \mathcal A_\tau^{(\underline x)}\neq \varnothing]   =\mathbb P(\mathcal Z_0^{\underline k} 
    \cap {\mathscr C_0(  \mathcal A_\tau^{(\underline x)})}=\varnothing \mid \mathcal Y_\tau
    \cap \mathcal A_\tau^{(\underline x)}\neq \varnothing).  
  \end{equation*}
  Combining this with Lemma \ref{dualcondfix}, and since $\underline
  k$ was arbitrary, we obtain the assertion.
\end{proof}

We are interested in the limit of \eqref{deFincond} as $\underline x =
\underline x(\varepsilon) \sim \varepsilon \underline e_\iota$ and
$\varepsilon \to 0$ for a fixed $\iota \in \{1,\ldots, d\}$. 
%Since,
%$\mathscr C_\tau(\mathcal Z_0^{\underline \Psi}) = \mathcal
%Z_\tau^{\underline\Psi}$, we know from Remark~\ref{rem:coma} that
%$\mathcal Z_0^{\underline{\Psi}} \cap \mathcal M_0^{\underline
%  x,\tau}\neq \varnothing$ appearing in \eqref{deFincond} if and only
%if $\mathcal Z_\tau^{\underline{\Psi}} \cap \mathcal \mathcal
%M^{\underline x, \tau}\neq \varnothing$. 
For brevity we write
\begin{equation}\label{Px}
  \mathbb P^{\underline x,\tau}(\cdot) := \mathbb P(\cdot \mid \mathcal Y_\tau
    \cap \mathcal A_\tau^{(\underline x)}\neq \varnothing).
\end{equation}

\begin{remark}[Limit of small frequencies]\label{Poisrem}
  Let $\mathscr P$ be a Poisson point process on $\{1,\ldots, d\}
  \times [0,1]$ with intensity measure $2\alpha \underline \rho
  \otimes \lambda$. (Compare with Proposition~\ref{lemeq}.) For $\iota
  \in \{1,\ldots, d\}$ and $\underline x = \underline x(\varepsilon) =
  \varepsilon \underline e_\iota$, the conditional distribution of
  $(\mathcal Y_\tau, \mathcal Y_\tau\cap\mathcal A_\tau^{(\underline
    x(\varepsilon))})$ given $\{\mathcal Y_\tau\cap\mathcal
  A_\tau^{(\underline x(\varepsilon))}\neq \varnothing\}$ converges,
  as $\varepsilon \to 0$, to the distribution of $(\mathscr
  P^{(\iota)}, \{(\iota,U)\})$, with $\mathscr P^{(\iota)}:= \mathscr
  P\cup \{(\iota,U)\}$, and $U$ independent of $\mathscr P$
  { and uniformly distributed on $[0,1]$}. In particular,
  under the limit of $\mathbb P^{\varepsilon \underline e_\iota,\tau}$
  as $\varepsilon \to 0$, with probability $1$ there is exactly one
  marked particle in $\mathcal Y_\tau$, with the location of this
  particle being $\iota$. Indeed, (using the same notation as in the
  proof of Corollary~\ref{c.3.1}),
  \begin{equation}
    \label{eq:9810}
    \begin{aligned}
      \lim_{\varepsilon\to 0} \mathbb P^{\underline
        x(\varepsilon),\tau}(\#(\mathcal
      Y_\tau {\cap (\{\iota\} \times[0,1])})=k) & = \lim_{\varepsilon\to
        0}\frac{e^{-2\alpha \rho_\iota} (2\alpha\rho_\iota)^k
        (1-(1-\varepsilon)^k)/ k!}{1 - \sum_{\ell=0}^\infty
        e^{-2\alpha \rho_\iota}(2\alpha \rho_\iota)^\ell
        (1-\varepsilon)^\ell / \ell!} \\ & = \lim_{\varepsilon\to
        0}\frac{e^{-2\alpha \rho_\iota} (2\alpha\rho_\iota)^k
        k\varepsilon/ k!}{1 - e^{- 2\alpha\rho_\iota \varepsilon}} \\ & =
      e^{-2\alpha \rho_\iota}
      \frac{(2\alpha\rho_\iota)^{k-1}}{(k-1)!},
    \end{aligned}
  \end{equation}
  the weight of a Poisson($2\alpha\rho_\iota$)-distribution at
  $k-1${. A similar calculation shows that this also equals
    the limit of $\mathbb P^{\underline
      x(\varepsilon),\tau}(\#(\mathcal Y_\tau {\cap (\{\iota\} \times[0,1])} )=k,
    \#\mathcal Y_\tau \cap \mathcal A_\tau^{(\underline x)} = 1)$ as
    $\varepsilon\to 0$}, explaining the additional particle
  $(\iota,U)$ in $\mathcal Y_\tau$ under $\mathbb P^{\iota,\tau}$.
\end{remark}

\begin{definition}[The process $\mathcal A$ with small marking probability\label{def:small}]
\textcolor{white}{ttt}
\begin{itemize}
\item The weak limit of
  $\mathbb P^{\varepsilon \underline e_\iota, \tau}(\mathcal A \in
  (.))$
  as $\varepsilon\to 0$ will be denoted by
  $$\mathbb P^{\iota, \tau}(\mathcal A \in (.)).$$ From the previous
  remark, under $\mathbb P^{\iota,\tau}$, there is a.s.  exactly one
  marked particle in $\mathcal Y_\tau$, with the location of this
  particle being $\iota$.  This particle will be denoted by~$\bullet$.
\item  For each colony $i$, consider the
  configuration $\mathscr C_0(\{\bullet\}) \cap \mathcal
  Z_0^{\underline \infty} {\cap ( \{i\} \times[0,1])}$, i.e.\ the
  configuration of all particles in $\mathcal Z_0^{\underline \infty}$
   that are located in colony $i$ and are connected with
  $\{\bullet\}$. By exchangeablity, the relative frequency of this
  configuration within $\mathcal Z_0^{\underline \infty}{\cap(
  \{i\} \times[0,1])}$ exists, $i=1,\ldots,d$, cf.\ Remark 3.7.2. As
  before, we denote the vector of these relative
  frequencies by $\underline F^{\iota,\tau} :=
  (F^{\iota,\tau}_1,\dots,F^{\iota,\tau}_d)$.
  \end{itemize}
\end{definition}

\begin{corollary}[Entrance laws for \eqref{eq:SDE1}]
  {\label{corlimit} There exists a weak limit of the
    distribution of $\mathcal X^\ast$ under $\mathbf P_{\varepsilon
      \underline e_\iota}$ as $\varepsilon\to 0$, which we denote by
    $\mathbf P_{\underline 0}^\iota(\mathcal X^\ast \in (.))$.}  In
  particular, $((\underline X^\ast_t)_{t>0}, \mathbf P_{\underline
    0}^\iota)$ defines an entrance law from $\underline 0$ for the
  dynamics~\eqref{eq:SDE1}.
\end{corollary}

\begin{proof}
  As a consequence of \eqref{deFincond} and the reasoning in Remark
  \ref{Poisrem} we have
  \begin{align}\label{entra}
    \mathbf P_{\varepsilon \underline e_\iota}(\underline X^\ast(\tau)
    \in (.)) = \mathbb P^{\varepsilon \underline
      e_\iota,\tau}(\underline F^{\varepsilon \underline e_\iota,
      \tau}\in (.))  \xrightarrow{\varepsilon\to 0} \mathbb
    P^{\iota,\tau}(\underline F^{\iota, \tau}\in (.)).
  \end{align}
  Together with the Markov property, this shows that there exists a
  weak limit of the distribution of $\mathcal X^\ast$ under $\mathbf
  P_{\varepsilon \underline e_\iota}$ as $\varepsilon\to 0$. Hence the
  result follows.
\end{proof}

\begin{remark}[Asymptotic expected frequencies\label{rem:asexfr}]
  For the asymptotic frequencies, we have that $\rho_\iota \mathbb
  E^{\iota,t}[F_j^{\iota,t}]/t \xrightarrow{t\to 0}\delta_{\iota
    j}$. Indeed, $\mathbb E^{\iota,t}[F_j^{\iota,t}]$ is the
  probability that a particle from $\mathcal
  Z^{{\underline \infty}}_0$ located on colony $j$
  belongs to $\mathscr C_0(\{\bullet\})$. In order for the particle to
  be connected to~$\bullet$, a coalescence event within time $t$ must
  occur. For small $t$, and up to linear order in $t$, this can only
  happen if the particle is located on the same colony, i.e.\
  $\iota=j$. In this case, since the coalescence rate on colony
  $\iota$ is $1/\rho_\iota$, the result follows.
\end{remark}

\begin{remark}[A correction of \cite{PfaffelhuberPokalyuk2013}]
  In \cite{PfaffelhuberPokalyuk2013} the case of a single colony
  ($d=1$) is studied. Lemma 2.4 of \cite{PfaffelhuberPokalyuk2013}
  can be seen as an analogue of our Lemma~\ref{fixfreq} (together with
  Remark~\ref{Poisrem}). However, Lemma 2.4 of
  \cite{PfaffelhuberPokalyuk2013} neglects the effect which the
  conditioning on the event $\{\mathcal Y_\tau
    \cap \mathcal A_\tau^{(\underline x)}\neq \varnothing\}$ has on the
  distribution of \, $\underline \Psi$, and works right away with the
  time-reversal of $\mathcal Y$ in equilibrium. Our
  analysis shows that, in spite of this imprecision, the conclusions
  of the main results of \cite{PfaffelhuberPokalyuk2013} remain true.
\end{remark}
As a consequence of \eqref{deFincond} and \eqref{entra} we obtain
\begin{equation} \label{fixdual} \mathbf P^\iota_{\underline 0}(T_{\rm
    fix} \le \tau)= \mathbf P^\iota_{\underline 0}(\underline
  X^\ast(\tau)= 1) = \mathbb P^{\iota,\tau}(\mathcal Z_0^{\underline
    \infty} \subseteq \mathscr C_{0}(\{\bullet\}).
\end{equation}

\subsection{Proof of Proposition \ref{P:main}}\label{proofpropo}
From \eqref{fixdual} we now derive a result on how to approximate
$T_{\rm fix}$ as $\alpha \to \infty$. The idea is that in this limit
the time which it takes for $\mathcal Z^{\underline \infty}$ to
coalesce with $\mathcal Y$ is essentially negligible on the
$\frac{\log \alpha}{\alpha}$-timescale. This is captured by the
following lemma, whose proof we defer to the end of the section.  
{
\begin{lemma}[Approximating $T_{\rm fix}$\label{cdinf}] For
  $\delta, \tau>0$, let
  $\delta_\alpha := \delta \frac{\log \alpha}{\alpha}$ and
  $\tau_\alpha := \tau \frac{\log \alpha}{\alpha}$, and let $\bullet$ be as in Definition \ref{def:small}. Then,
  \begin{align}\label{lower}
    \mathbb P^{\iota,\tau_\alpha}(
    \mathcal Z_0^{\underline\infty} \subseteq \mathscr C_0(\{\bullet\}))
    \leq \mathbb P^{\iota,\tau_\alpha}(\mathcal Y_0
    \subseteq \mathscr C_0(\{\bullet\})) \mbox{ for all } \alpha > 0,
  \end{align}
  \begin{align}\label{upper}
    \liminf_{\alpha\to\infty} \mathbb P^{\iota,\tau_\alpha}(\mathcal Y_0
    \subseteq \mathscr C_0(\{\bullet\}))  \le   \liminf_{\alpha\to\infty} \mathbb P^{\iota,\tau_\alpha +
    \delta_\alpha}(\mathcal Z_0^{\underline\infty} \subseteq \mathscr
    C_0 (\{\bullet\})).
  \end{align}
  % $$ \lim_{\alpha\to\infty}\mathbb P^{\iota,\tau_\alpha}(
  % \mathcal Z_0^{\underline\infty} \subseteq \mathscr
  % C_0(\{\bullet\})) \leq \lim_{\alpha\to\infty} \mathbb
  % P^{\iota,\tau_\alpha}(\mathcal Y_0 \subseteq \mathscr
  % C_0(\{\bullet\})) \leq \lim_{\alpha\to\infty}\mathbb
  % P^{\iota,\tau_\alpha + \delta_\alpha}(\mathcal
  % Z_0^{\underline\infty} \subseteq \mathscr C_0 (\{\bullet\})).$$
\end{lemma}

\noindent
The next corollary follows by combining \eqref{fixdual}
and Lemma \ref{cdinf}.
\begin{corollary}\label{corfix}
  For $\alpha >0$ let $S_\alpha$ be a random variable with
  distribution function
  $\tau \mapsto \mathbb P^{\iota,\tau_\alpha}( \mathcal Y_{0}
  \subseteq \mathscr C_{0}(\{\bullet\}))$,
  where $\tau_\alpha= \tau \frac{\log \alpha}{\alpha}$. (In the
  subsequent proof of Proposition \ref{P:main} we will see that
  $S_\alpha$ has a natural interpretation as the rescaled fixation
  time of $\bullet$ in the time-reversal of $\mathcal Y$.)  If
  $S_\alpha$ converges in distribution as $\alpha \to \infty$ and if
  $\tau$ is a point of continuity of the limiting distribution
  function, we have
  \begin{equation}\label{fixasympt}
    \lim_{\alpha \to \infty} \mathbf P^\iota_{\underline 0}(T_{\rm fix} \le \tau_\alpha) = \lim_{\alpha \to \infty} 
    \mathbb P^{\iota,\tau_\alpha}( \mathcal Y_{0} \subseteq \mathscr C_{0}(\{\bullet\})).
  \end{equation}
\end{corollary}

\begin{proof}
  The limit in the right hand side exists by assumption. If
  $\tau-\delta$ is a continuity point of the limiting distribution
  function $F$, then we have by \eqref{upper} (with $\tau$ replaced
  by $\tau-\delta$) and again abbreviating
  $\delta_\alpha = \delta \frac{\log \alpha}{\alpha}$
  $$\lim_{\alpha \to \infty}\mathbb P^{\iota,\tau_\alpha - \delta_\alpha}( \mathcal Y_{0} \subseteq \mathscr C_{0}(\{\bullet\})) 
  \leq \liminf_{\alpha\to\infty}\mathbb P^{\iota,\tau_\alpha}(
  \mathcal Z_0^{\underline\infty} \subseteq \mathscr
  C_0(\{\bullet\})).$$
  Hence, working along a sequence of continuity points $\tau-\delta$
  of $F$ with $\delta \downarrow 0$, we have 
  \begin{align*}
    \lim_{\alpha \to \infty} \mathbb P^{\iota,\tau_\alpha}( \mathcal Y_{0} \subseteq \mathscr C_{0}(\{\bullet\})) 
    &= \lim_{\delta\to 0} \lim_{\alpha \to \infty}\mathbb P^{\iota,\tau_\alpha - \delta_\alpha}( \mathcal Y_{0} \subseteq \mathscr C_{0}(\{\bullet\})) 
    \\ & \leq  \liminf_{\alpha\to\infty}\mathbb P^{\iota,\tau_\alpha}(
         \mathcal Z_0^{\underline\infty} \subseteq \mathscr C_0(\{\bullet\})) =         \liminf_{\alpha \to \infty} \mathbf P^\iota_{\underline 0}(T_{\rm fix} \le \tau_\alpha) 
    \\ & \leq         \limsup_{\alpha \to \infty} \mathbf P^\iota_{\underline 0}(T_{\rm fix} \le \tau_\alpha) = \limsup_{\alpha\to\infty}\mathbb P^{\iota,\tau_\alpha}(
         \mathcal Z_0^{\underline\infty} \subseteq \mathscr C_0(\{\bullet\})) 
    \\ & \leq 
         \lim_{\alpha \to \infty} \mathbb P^{\iota,\tau_\alpha}( \mathcal Y_{0} \subseteq \mathscr C_{0}(\{\bullet\})).
  \end{align*}
\end{proof}

The preceding corollary shows that, in order to study the asymptotic distribution of $T_{\rm fix}$ on the $\frac{\log \alpha}{\alpha}$-timescale, it suffices to analyse
the asymptotics of the percolation probabilities of the marked particles {\em within the equilibrium ASG} under the (conditional) probability
$\mathbb P^{\iota,\tau_\alpha}$. As already explained in Sec. \ref{S30}, the link to  Proposition \ref{P:main} is now given by a time reversal argument.

\begin{proof}[Proof of Proposition~\ref{P:main}]
  In view of \eqref{fixasympt}, we are done once we show that, for $\tau > 0$,
  \begin{align}\label{eq:P311}
    \mathbb P\big(T\leq \tau\big) = \mathbb P^{\iota,\tau}(\mathscr
    C_{0}(\{\bullet\}) \supseteq  \mathcal Y_{0}),
  \end{align}
  where $T$ is defined in \eqref{defT}. For this, we bring the time reversal
  $\widehat{\mathcal Y}$ of $\mathcal Y
    = ( \mathcal Y_r)_{0\le r\le \tau}$ into play,
  which is defined by
  $$ \widehat {  \mathcal Y}_s:= 
  {\mathcal Y}_{\tau-s}, \quad 0\le s\le \tau.$$ {
    Analogously, we define $\widehat{\mathscr C}_s(\{\bullet\}) :=
    {\mathscr C}_{\tau-s}(\{\bullet\}).$ Then, our assertion
    \eqref{eq:P311} is equivalent to
    \begin{align}\label{eq:P3111}
      \mathbb P\big(T\leq \tau\big) = \mathbb P^{\iota,\tau}(\widehat{\mathscr
        C}_{\tau}(\{\bullet\}) \supseteq  \widehat{\mathcal Y}_{\tau}).      
    \end{align}} We recall that the dynamics of $\widehat { \mathcal
    Y} $ in equilibrium is given by $\mathscr D^a$; see
  Proposition~\ref{lemeq}. While for $ \mathcal Y$ the conditioning
  \eqref{Px} is at the {\em terminal} time $\tau$ (and thus modifies
  the dynamics $\mathscr D$), the same conditioning expressed for
  $\widehat { \mathcal Y}$ happens at the {\em initial} time $0$ and
  thus does effect the initial state but not the dynamics $\mathscr
  D^a$. The distribution of $\hat {\mathcal Y}_0$ which results from
  this conditioning is described in Remark \ref{Poisrem}. Thus we
  observe that under $\mathbb P^{\iota,\tau}$, the time-reversed
  process $\widehat { \mathcal Y}$ follows the dynamics $\mathscr D^a$
  and has initial state $\widehat{ \mathcal Y}_0= \mathscr
  P^{(\iota)}= \mathscr P\cup \{(\iota,U)\}$, with $\mathscr P$
  defined in Remark \ref{Poisrem} and $\bullet:=(\iota,U)$.
  
  We now put for $i=1,\ldots,d$ and $t\ge 0$
  \begin{align}
    \label{eq:prop3.1a}
    \widehat {N}^i_t := \# \left(\widehat { \mathcal Y}_t \cap
      (\{i\}\times[0,1])\right), \quad \widehat H^i_t := \#
    \left(\widehat { \mathscr C}_{t}( \{\bullet \}) \cap
      (\{i\}\times[0,1])\right).
  \end{align}
  Under $\mathbb P^{\iota,\tau}$ the process $(\widehat {\underline
    N}_t, \widehat {\underline H}_t)_{0\le t \le \tau}$ with $\widehat
  {\underline N}_t = ({\widehat N}^1_t,\dots, {\widehat N}^d_t)$ and
  $\widehat {\underline H}_t = (\widehat {H}^1_t,\dots,\widehat
  {H}^d_t)$, then has the same law as the process $({\underline L}_t,
  {\underline M}_t)_{0\le t \le \tau}$ defined in Proposition
  \ref{P:main}.  In particular, {\eqref{eq:P3111} is
    shown.}  
 
\end{proof} {We prepare the proof of Lemma~\ref{cdinf} by
  two estimates and include their (simple) proofs for
  convenience. 
  %These estimates certainly could be sharpened; we need
  %them, however, only in the stated form.

\begin{remark}[Comparing $\underline \Pi$ and
  $\underline\Pi + \underline e_\iota$\label{rem:compPi}]
  Recall that $\underline\Pi = (\Pi_1,...,\Pi_d)$ is
  distributed according to $d$ independent Poisson distributions,
  where $\Pi_i \sim \text{Poi}(2\alpha\rho_i)$. As above,
  $\underline\Psi$ is distributed as $\underline \Pi$,
  conditioned to be positive (compare with \eqref{eq:Pi}) 
  and $\underline \Pi + \underline
  e_\iota$ is as in Proposition~\ref{P:main}. Then, ($d_{TV}$
  denoting the total variation distance)
  \begin{equation}
    \label{eq:TV1}
    \begin{aligned}
      d_{\text{TV}}(\underline\Pi, \underline\Psi) & = o(1),\\
      d_{\text{TV}}(\underline\Pi, \underline\Pi + \underline e_\iota) & = o(1)
    \end{aligned}
  \end{equation}
  as $\alpha\to\infty$.\\
  Indeed: The first result is immediate since $\mathbb
  P[\underline\Pi=0] = e^{-2\alpha}$. For the second result, by
  a second moment calculation, we have that
  $\Pi_\iota/(2\alpha\rho_\iota) \xrightarrow{\alpha\to\infty}1$
  in $L^2$ and therefore, as $\alpha\to\infty$,
  \begin{align*}
    d_{\text{TV}}(\underline\Pi, \underline\Pi + \underline e_\iota) & =
    e^{-2\alpha\rho_\iota} \sum_{k=1}^\infty
    \Big|\frac{(2\alpha\rho_\iota)^k}{k!} -
    \frac{(2\alpha\rho_\iota)^{k-1}}{(k-1)!}\Big| + o(1)\\ & =
    e^{-2\alpha\rho_\iota} \sum_{k=1}^\infty
    \frac{(2\alpha\rho_\iota)^{k}}{k!}\Big| 1 -
    \frac{k}{2\alpha\rho_\iota}\Big| + o(1) \\ & = \mathbb E\Big[\Big|
    1 - \frac{\Pi_\iota}{2\alpha\rho_\iota}\Big|\Big] + o(1) =
    o(1).
  \end{align*}
\end{remark}
}
\noindent
We are now ready for the

\begin{proof}[Proof of Lemma~\ref{cdinf}]
  For proving \eqref{lower} it suffices to show that, for each
  ${\alpha~>0}$,
  $\mathbb P^{\iota,\tau_\alpha}((I,U)\notin \mathscr
  C_0(\{\bullet\}), \mathcal Z_0^{\underline\infty}\subseteq \mathscr
  C_0(\{\bullet\})) = 0$
  for a particle $(I,U)$ taken uniformly from $\mathcal Y_0$.  To show
  this equality, we will prove that for all $i=1,\ldots, d$
  \begin{align}\label{eq:914}
    \mathbb P^{\iota,\tau_\alpha}((I,U)\notin \mathscr
    C_0(\{\bullet\}), \mathcal Z_0^{\underline\infty}\subseteq \mathscr
    C_0(\{\bullet\}), I=i\})=0.
  \end{align}
  We write $\mathbf p:= (I,U)$, and note that
  \begin{align*}
    R_{\mathbf p} := \inf\{r> 0: \mathscr C_r(\{\mathbf p\})
    \not\subseteq \{i\}\times [0,1]\} > 0 \quad \mathbb
    P^{\iota,\tau_\alpha} \mbox{ a.s. }
  \end{align*}
  The idea is now that with probability 1 we will find particles in
  $\mathcal Z$ which coalesce with
  $\bigcup_{r\ge 0} \mathscr C_r(\{\mathbf p\})$, withouth being
  affected by an earlier branching or coalescence with
  $\mathcal Y \setminus \bigcup_{r\ge 0} \mathscr C_r(\{\mathbf p\})$,
  and hence on the event
  $\{\mathbf p \notin \mathscr C_0(\{\bullet\})\}$ never connect to
  the particle~$\bullet$. In order to achieve this, we recall that
  under $ \mathbb P^{\iota,\tau_\alpha}$ the dynamics of $\mathcal Z$
  is given by $\mathscr D^b$, and this also applies conditional under
  $\mathcal Y$ for the particles in $\mathcal Z$ up to the time of
  their possible coalescence with particles in $\mathcal Y$.

  We now consider the subsystem of particles in $\mathcal Z$ which
  initiates from all those $\mathcal Z$ particles that are located in
  colony $i$ at time $0$, and remove from it all those particles that
  undergo a migration or a branching event, or coalesce with some
  particle in $\mathcal Y_r \setminus \mathscr C_r(\mathbf p)$ at some
  time $r \ge 0$. The system of particles of $\mathcal Z$ at time $r$
  which remain after this pruning (and all of which are located in
  colony $i$ by construction) will be denoted by $\mathcal Z^{(i)}_r$.
  
  Given $\mathcal Y$, the process $(\# \mathcal Z^{(i)}_r)$ is up to
  time $R_{\mathbf p}$ stochastically bounded from below by a death
  process $(K_r)_{{r \geq 0}}$ entering from infinity with death rate
  $\frac{1}{\rho_i}\binom{k}2 + (\alpha + \mu\sum_{j\neq i}b(i,j) +
  M)k$,
  where
  $M:= \max \{\#(Y_r{\cap}(\{i\}\times[0,1])) : 0\le r\le R_{\mathbf
    p}\}$.
  Hence, the essentially quadratic death rate guarantees that for any
  $c >0$
  $\lim_{\varepsilon \rightarrow 0 } \int_{\varepsilon}^{c} K_r dr =
  \infty$
  {a.s}. Indeed, $\rho_i rK_r \xrightarrow{r\to 0} 2$ a.s.\ by a
  second moment calculation, and $K_r \approx \frac{2}{\rho_i r}$ is
  not integrable at~$r=0$.  Consequently, also
  $\lim_{\varepsilon\to 0} \int_{\varepsilon}^c\# \mathcal Z^{(i)}_r
  dr = \infty$
  a.s., and thus with probability 1 there will be a coalescence
  between $\mathcal Z^{(i)}_r$ and $\mathscr C_r(\{\mathbf p\})$ for
  some $r < R_{\mathbf p}$. 

  Since on the event
  $\{\mathbf p \notin \mathscr C_0(\{ \bullet \}) \}$ the set
  $\mathscr C_r(\{\mathbf p\})$ is contained in the complement of
  $\mathscr C_r(\{\bullet\})$, we conclude the existence of particles
  in $\mathcal Z^{(i)}_0$ (and hence in
  $\mathcal Z_0^{\underline\infty}$) that belong to the complement of
  $\mathscr C_0(\{\bullet\})$. This shows \eqref {eq:914}.
  \\\\
  To prove \eqref{upper}, we first note that the particle $\bullet$  specified in Definition~\ref{def:small} is (because of the random marking)  a uniform choice
  from the particles in
  $\mathcal Y_{\tau_\alpha} \cap (\{\iota\} \times [0,1])$ under
  $\mathbb P^{\iota,\tau_\alpha}$, and a uniform choice from the
  particles in
  $\mathcal Y_{\tau_\alpha+\delta_\alpha} \cap (\{\iota\} \times
  [0,1])$
  under $\mathbb P^{\iota,\tau_\alpha+\delta_\alpha}$. 

  However, as noted already after formula \eqref{eq:P3111}, the
  conditioning at time $\tau_\alpha$, which is inherent in
  $\mathbb P^{\iota,\tau_\alpha}$, destroys the time-homogeneity of
  the dynamics of $\mathcal Y$ between times $0$ and ${\tau_\alpha}$;
  consequently, under $\mathbb P^{\iota,\tau_\alpha}$ the marking
  probabilities in $\mathcal Z_0^{\underline\infty}$ will be different
  from those in $\mathcal Y_0$. In order to account for this, the
  strategy of our proof will be to define under the {\em unconditioned} probability measure $\mathbb P$  particles
  $\circ$ and
  $\circ'$ whose distributions will turn out to be close in variation distance to that of $\bullet$ under $\mathbb P^{\iota,\tau_\alpha}$ and under $\mathbb P^{\iota,\tau_\alpha+\delta_\alpha}$, respectively, and which lead to the same marking probabilities in $\mathcal Y_0$ and $\mathcal Z_0^{\underline\infty}$.
  
  To be specific, let $\circ$ result from a uniform pick from
    $(\mathcal Y_{\tau_\alpha} \cup \mathcal
    Z^{\underline\infty}_{\tau_\alpha})\cap (\{\iota\}\times [0,1])$
    provided that this set is not empty; otherwise we pick $\circ$
    uniformly from
    $\mathcal Y_{\tau_\alpha} \cup \mathcal
    Z^{\underline\infty}_{\tau_\alpha}$.
    Similarly, we pick $\circ'$ uniformly from
    $(\mathcal Y_{\tau_\alpha+\delta_\alpha} \cup \mathcal
    Z^{\underline\infty}_{\tau_\alpha+\delta_\alpha})\cap
    (\{\iota\}\times [0,1])$
    provided that this set is not empty; otherwise we pick $\circ'$
    uniformly from
    $\mathcal Y_{\tau_\alpha+\delta_\alpha} \cup \mathcal
    Z^{\underline\infty}_{\tau_\alpha+\delta_\alpha}$.
   
    This construction immediately implies that for any fixed
    $i=1,...,d$, the family of events
    $(\{(i, U_{ig}) \in \mathscr C_0(\circ)\})_{(i, U_{ig}) \in
      \mathcal Y_0 \cup \mathcal Z^{\underline\infty}_0}$,
    is exchangeable conditional under
    $\mathcal Y_0 \cup \mathcal Z^{\underline\infty}_0$. We will show
    five properties ((A)-(E)) of the joint distribution of
    $\mathcal A$, $\mathcal Y$, $\mathcal Z^{\underline \infty}$ and
    $\circ$, proceeding in two main steps proving first (A) and then
    (B)-(E).
    \\
    {(A) the total variation distance between the distribution of
      $(\mathcal A_t, \mathcal Y_t, \mathcal Z^{\underline\infty}_t,
      \circ)_{0\leq t \leq \tau_\alpha}$
      under $\mathbb P$ and the distribution of
      $(\mathcal A_t, \mathcal Y_t, \mathcal Z^{\underline\infty}_t,
      \bullet)_{0\leq t \leq \tau_\alpha}$
      under $\mathbb P^{\iota, \tau_\alpha}$ converges to $0$ as
      $\alpha \to \infty$. Likewise, the total variation distance
      between the distribution of
      $(\mathcal A_t, \mathcal Y_t, \mathcal Z^{\underline\infty}_t,
      \circ')_{0\leq t \leq \tau_\alpha+\delta_\alpha}$
      under $\mathbb P$ and the distribution of
      $(\mathcal A_t, \mathcal Y_t, \mathcal Z^{\underline\infty}_t,
      \bullet)_{0\leq t \leq \tau_\alpha+\delta_\alpha}$
      under $\mathbb P^{\iota, \tau_\alpha+\delta_\alpha}$ converges
      to $0$ as $\alpha \to \infty$.}
    
  Having achieved this, we will construct a process
  $({\mathcal Z}'_r)_{0\leq r\leq \delta_\alpha}$ under $\mathbb P$
  with the following properties:
  \\
  (B) $\mathcal Z'_r \subseteq \mathcal Z^{\underline\infty}_r$ for all $r\in [0, \delta_\alpha]$, \\
  (C)
  $ \{\mathcal Z'_{\delta_\alpha} \subseteq \mathscr
  C_{\delta_\alpha}(\circ')\} \subseteq \{\mathcal
  Z^{\underline\infty}_{0} \subseteq \mathscr C_{0}(\circ')\};
  $\\
  (D) for any $i=1,...,d$,
  $\mathcal Z'_{\delta_\alpha}(\{i\}\times[0,1]) = {\mathcal{
      O}(\alpha/\log(\alpha)})$
  with high probability as $\alpha \to \infty$,
  \\
  (E) for any $i=1,...,d$,  the family of events
  $(\{(i,U_{ig}) \in \mathscr C_{\delta_\alpha}(\circ')\})_{(i,U_{ig})
    \in \mathcal Y_{\delta_\alpha} \cup {\mathcal
      Z'_{\delta_\alpha}}}$ is exchangeable conditional under
  $\mathcal Y_{\delta_\alpha} \cup \mathcal
  Z'_{\delta_\alpha}$.
  
The proof of the first assertion of (A) will be achieved in several steps. \\
  (i) We first note that because of Remark \ref{rem:compPi} the total variation distance between the distributions of $\mathcal Y_{\tau_\alpha}$ under $\mathbb P^{\iota,\tau_\alpha}$ and  under $\mathbb P$ converges to~$0$ as $\alpha \to \infty$. \\
  (ii) Now a crucial observation is that
  the time-reversed dynamics of
  $(\mathcal Y_t)_{0\le t \le \tau_\alpha}$ under $\mathbb P$ and
  under $\mathbb P^{\iota,\tau_\alpha}$ both are given by the dual jump kernel ${\mathscr D}^a$. Consequently,  the conditional
  distribution of $(\mathcal Y_t)_{0\le t\le \tau_\alpha}$ given
  $\mathcal Y_{\tau_\alpha}$ under  $\mathbb P^{\iota,\tau_\alpha}$ equals that under $\mathbb P$.  This shows that the variational distance between the distributions of  $(\mathcal Y_t)_{0\leq t \leq \tau_\alpha}$ under $\mathbb P^{\iota,\tau_\alpha}$ and under  $\mathbb P$ equals the variational distance between the distributions of $\mathcal Y_{\tau_\alpha}$ under $\mathbb P^{\iota,\tau_\alpha}$ and under~$\mathbb P$. \\
  (iii) Next note that  the conditional distribution of $(\mathcal A_t, \mathcal Y_t, \mathcal Z^{\underline\infty}_t)_{0\leq t \leq \tau_\alpha}$ given  $(\mathcal Y_t)_{0\leq t \leq \tau_\alpha}$ under  $\mathbb P^{\iota,\tau_\alpha}$ equals that under $\mathbb P$.  Hence the variational distance between the distributions of $(\mathcal A_t, \mathcal Y_t, \mathcal Z^{\underline\infty}_t)_{0\leq t \leq \tau_\alpha}$  under $\mathbb P^{\iota,\tau_\alpha}$ and under  $\mathbb P$ equals the variational distance between the distributions of $(\mathcal Y_t)_{0\leq t \leq \tau_\alpha}$ under $\mathbb P^{\iota,\tau_\alpha}$ and under~$\mathbb P$.\\
  (iv) Combining (i)-(iii) we see that the total variation distance between the distribution of $(\mathcal A_t, \mathcal Y_t, \mathcal Z^{\underline\infty}_t)_{0\leq t \leq \tau_\alpha}$
  under $\mathbb P$ and the distribution of
  $(\mathcal A_t, \mathcal Y_t, \mathcal Z^{\underline\infty}_t)_{0\leq t \leq \tau_\alpha}$
  under $\mathbb P^{\iota, \tau_\alpha}$ converges to $0$ as $\alpha \to \infty$.\\
  (v) According to Definition \ref{def:small}, due to the random marking under $\mathbb P^{\iota,\tau_\alpha}$ the particle $\bullet$ arises by a uniform choice from $\mathcal Y_{\tau_\alpha} \cap (\{\iota\} \times [0,1])$. We now claim that under $\mathbb P$, on an event whose probability converges to 1 as  $\alpha \to \infty$, the particle $\circ$  constitutes a uniform choice from $\mathcal Y_{\tau_\alpha} \cap (\{\iota\} \times [0,1])$. 
We will prove in the next section a key lemma,  Lemma~\ref{l:numberASGnew}, which will tell us that
    under $\mathbb P$ the number of particles in $\mathcal Y_t$ in
    colony $i$, $i=1,\ldots, d$, is with high probability as
    $\alpha \to \infty$ concentrated around $2\rho_i \alpha$,
    uniformly in $t\in[0,\tau_\alpha]$. Hence our claim holds if
    $\#(\mathcal Z^{\underline\infty}_{\tau_\alpha} \setminus \mathcal
    Y_{\tau_\alpha}) = o(\alpha)$  with high probability as $\alpha\to\infty$.  To see this, we note that the probability of the event
  $$\{\# (\mathcal Y_{t} \cap (\{i\}\times [0,1])) \geq
  2\alpha\rho_i(1-\varepsilon) \mbox{ for some } \varepsilon>0 \mbox{
    and for all } i; \, 0\le t \le \tau_\alpha\}$$
tends to 1 as $\alpha \to \infty$ because of
  Lemma~\ref{l:numberASGnew}. On this event, however, the process
  $\#(\mathcal Z^{\underline\infty}_{t} \setminus \mathcal
  Y_{t})_{0\le t\le \tau_\alpha}$ under $\mathbb P$
  is stochastically bounded from above by a birth-death process which in state
  $(k_1, ..., k_d)$ with $k = k_1+ ...+k_d$ has birth rate $\alpha k$
  and death rate at least 
  $\sum_{i=1}^d \frac{1}{\rho_i} \binom {k_i} 2 + \frac{1}{\rho_i} 2
  \alpha \rho_i (1-\varepsilon) k_i \geq \frac{k(k-d)}{2 d} +
  2\alpha(1-\varepsilon)k$, see~\eqref{eq:below}.
  Hence a second moment
  calculation shows that, with high probability as $\alpha\to\infty$, 
  $\#(\mathcal Z^{\underline\infty}_{\tau_\alpha} \setminus \mathcal
  Y_{\tau_\alpha}) = O(\tfrac{\alpha}{\log\alpha}) = o(\alpha)$.
 Together with (iv), this shows the first part of the assertion of (A);
 the arguments for the second part of (A) are the same,
 with $\tau$ being replaced by $\tau+\delta$.

For (B)-(E), we define the particle system
  $(\mathcal Z_t')_{0\leq t\leq \delta_\alpha}$ as a subsystem of
  $({\mathcal Z}^{\underline\infty}_t)_{0\leq t\leq \delta_\alpha}$
  (from which property (B) is automatic). As its initial state we
  take ${\mathcal Z}'_0 := {\mathcal Z}^{\underline\infty}_0$. We then
  impose the rule that the particles in $\mathcal Z'$ perform all coalescence
  and migration events dictated by $\mathcal Z$, but follow only a
  single one of the two particles in $\mathcal Z$ upon a branching
  event. More formally,
  \begin{itemize}
  \item if
    $(i,U_{ig}), (i,U_{ig'})\in \mathcal Z^{\underline\infty}_{r-}$
    coalesce, i.e. are replaced by
    $(i,U_{ig''})\in \mathcal Z'^{\underline\infty}_{r}$, and if
    $(i,U_{ig}),(i,U_{ig'}) \in \mathcal Z'_{r-}$, then the same
    replacement happens in ${\mathcal Z}'_{r}$,
  \item if
    $(i,U_{ig}), (i,U_{ig'})\in \mathcal Z^{\underline\infty}_{r-}
    \cup \mathcal Y_{r-}$
    coalesce, i.e. are replaced by
    $(i,U_{ig''})\in \mathcal Z^{\underline\infty}_{r}$, and if only
    $(i,U_{ig}) \in {\mathcal Z}'_{r-}$ but
    $(i,U_{ig'}) \notin {\mathcal Z}'_{r-}$, then
    $(i,U_{ig}) \in {\mathcal Z}'_{r-}$ is replaced by $(i,U_{ig''})$
    in ${\mathcal Z}'_{r}$,
  \item if $(i,U_{ig})\in \mathcal Z^{\underline\infty}_{r-}$ migrates
    to $j$, i.e. is replaced by $(j, U_{jg'})$ in
    $\mathcal Z^{\underline\infty}_{r}$, and if
    $(i,U_{ig})\in {\mathcal Z}'_{r-}$, the particle also migrates to
    $j$ in ${\mathcal Z}'_{r}$, i.e. $(i,U_{ig})$ is replaced by
    $(j, U_{jg'})$ in ${\mathcal Z}'_{r}$,
  \item if $(i,U_{ig})\in \mathcal Z^{\underline\infty}_{r-}$
    branches, i.e.\ is replaced by
    $(i,U_{ig'}), (i,U_{ig''}) \in \mathcal Z^{\underline\infty}_{r}$,
    and if $(i,U_{ig})\in {\mathcal Z'}_{r-}$, then $(i,U_{ig})$ is
    replaced by $(i,U_{ig'})$ in ${\mathcal Z}'_{r}$.
  \end{itemize}
  Note that
  $\mathcal Z^{\underline\infty}_0 \subseteq \mathscr C_0(\mathcal
  Z_{\delta_\alpha}')$
  by construction, so if
  $\mathcal Z'_{\delta_\alpha} \subseteq \mathscr
  C_{\delta_\alpha}(\circ')$
  then
  $\mathcal Z^{\underline\infty}_0 \subseteq \mathscr C_{0}(\mathscr
  C_{\delta_\alpha}(\circ')) = \mathscr C_{0}(\circ')$,
  i.e.\ we have property (C). Since $\mathcal Z'$ is a coalescing
  random walk, it is a death process which in state $(k_1,...,k_d)$
  with $k=k_1 + \cdots + k_d$ has death rate (using~\eqref{eq:below})
  $\sum_{i=1}^d \frac{1}{\rho_i} \binom {k_i} 2 \geq \frac{k(k-d)}{2
    d}$.
  A second moment calculation then shows (D).  Finally, the
  exchangeability claimed in (E) holds by construction.
    
  Based on properties (A)-(E) we can now prove \eqref{upper}. Indeed,
  because of (A) 
  \begin{align}\label{first}
    \mathbb P^{\iota,\tau_\alpha}({\mathcal Y}_0 \subseteq \mathscr
    C_0(\{\bullet\})) = \mathbb P({\mathcal Y}_0 \subseteq \mathscr
    C_0(\{\circ\}))+ o(1) \mbox{ as } \alpha \to \infty.
  \end{align}
  From the stationarity of $\mathcal Y$ under $\mathbb P$ together
  with property (A), 
  \begin{align}
    \mathbb P({\mathcal Y}_0 \subseteq \mathscr
    C_0(\{\circ\}))= \mathbb P({\mathcal Y}_{\delta_\alpha}
    \subseteq \mathscr C_{\delta_\alpha}(\{\circ'\}))+ o(1) \mbox{ as } \alpha \to \infty.
  \end{align}
  For all fixed $i \in \{1,\ldots, d\}$, consider the event
  $$E_{i,\alpha} := \{({\mathcal Y}_{\delta_\alpha}(\{i\}\times[0,1]) \ge \rho_i \alpha,  {\mathcal Z}'_{\delta_\alpha}(\{i\}\times[0,1]) \le \rho_i \alpha\}.$$
  Then because of the exchangeability property (E) we have
 $$\mathbb P({\mathcal Y}_{\delta_\alpha}\cap(\{i\}\times[0,1]) \subseteq  
 \mathscr C_{\delta_\alpha}(\{\circ'\}) \mid E_{i,\alpha} ) \le
 \mathbb P(\mathcal
   Z'_{\delta_\alpha}\cap(\{i\}\times[0,1])\subseteq \mathscr
 C_{\delta_\alpha}(\{\circ'\}) \mid E_{i,\alpha} ) .$$
 Because of property (D) we have $\mathbb P(E_{i,\alpha}) \to 1$ as
 $\alpha \to \infty$, and consequently
   \begin{align}
     \liminf_{\alpha \to \infty}  \mathbb P({\mathcal Y}_{\delta_\alpha}
     \subseteq \mathscr C_{\delta_\alpha}(\{\circ'\}))   \le  \liminf_{\alpha \to \infty}\mathbb P(\mathcal Z'_{\delta_\alpha}
     \subseteq \mathscr C_{\delta_\alpha}(\{\circ'\})).                               \end{align}
   Property (C) yields
            \begin{align}
\mathbb P(\mathcal Z'_{\delta_\alpha}
           \subseteq \mathscr C_{\delta_\alpha}(\{\circ'\}))    \le  \mathbb P(\mathcal Z^{\underline\infty}_{0}
           \subseteq \mathscr C_{0}(\{\circ'\}))                     \end{align}
           and property (A) implies 
           \begin{align}\label{last}
 \mathbb P(\mathcal Z^{\underline\infty}_{0}
           \subseteq \mathscr C_{0}(\{\circ'\})) =  \mathbb P^{\iota, \tau_\alpha + \delta_\alpha}(\mathcal Z^{\underline\infty}_{0}
                     \subseteq \mathscr C_{0}(\{\bullet\}))+ o(1) \mbox{ as } \alpha \to \infty.
                          \end{align}
                          Combining \eqref{first}-\eqref{last} we
                          arrive at \eqref{upper}.
\end{proof}

\subsection{Proof of Theorem \ref{t.1}}\label{proofthm1}
Let $\underline x \neq \underline 0$. Then equation \eqref{dualfix}
shows that the one-dimensional distributions of $\mathcal X^\ast$ are
determined. This shows the uniqueness (see Theorem~4.4.2 of
\cite{EthierKurtz86}).

Now let $(\mathcal X^\ast, \mathbf P)$ with
$\mathcal X^\ast = (\underline X^\ast(t))_{t\geq 0}$ be an entrance
law from $\underline 0$ for the dynamics \eqref{eq:SDE1}. For fixed
$t > 0$ and $0<\delta < t$ we can represent
$\mathbf P(\underline X^{\ast}(t) \in (\cdot))$ by means of
\eqref{deFincond}, putting $\tau:= t-\delta$ and using the ``random
paintbox'' $\underline X^\ast_\delta$ instead of the deterministic
$\underline x$ figuring in \eqref{deFincond}. More specifically, we
have by the Markov property of $\mathcal X^\ast$
\begin{align}
  \notag \mathbf P(\underline X^\ast(t) \in (.)) 
  &= \mathbf E[\mathbf P_{\underline X^\ast(\delta)}(\underline X^\ast(t-\delta) \in (.)) ]
  \\ & = \mathbf E[\mathbb P(\underline F^{\underline X^\ast(\delta),t-\delta}\in (.) \mid \mathcal Y_{t-\delta}
       \cap \mathcal A_{t-\delta}^{(\underline X^\ast(\delta))}\neq \varnothing)\mid  \underline X^\ast(\delta)].
       \label{repent}
\end{align}
Now consider the random vector
$\mathcal N_\delta:= (\mathcal Y_{t-\delta}(\{i\}\times
[0,1]))_{i=1,\ldots,d}$,
and write $\nu_\delta^{\underline X^\ast(\delta)}$ for the
distribution of $\mathcal N_\delta$ conditioned under the event
$\{\mathcal Y_{t-\delta}\cap\mathcal A_{t-\delta}^{(\underline
  X^\ast(\delta))}\neq \varnothing\}$
for given $\underline X^\ast(\delta)$. We recall that the
unconditional distribution of $\mathcal Y_{t-\delta}$ is the
distribution $\pi$ described in Proposition \ref{lemeq}. Thus we are
faced with a Poisson coloring, where the coloring is rare (due to the
assumption that $\underline X^\ast(\delta) \to \underline 0$ in
probability as $\delta \to 0$) but conditioned to produce at least one
colored particle.  Using the notation ${\underline \Pi}$ for a Poisson
vector as in Proposition \ref{lemeq}, we infer that there exist
$\{1,\ldots,d\}$-valued random variables $J_\delta$ independent of
$\underline \Pi$ such that the total variation distance between
$\nu_\delta^{\underline X^\ast(\delta)}$ and the distribution of
${\underline \Pi} + \underline e_{J_\delta}$ converges to $0$ as
$\delta \to 0$. We thus obtain from \eqref{repent} for all $t >0$
\begin{equation} \label{entrep} \mathbf P(\underline X^\ast(t) \in
  (.))  = \mathbf E [\mathbb P^{J_\delta,t-\delta}(\underline
  F^{J_\delta,t-\delta}\in (.))] + o(1) \qquad \mbox{as } \delta \to
  0.
\end{equation}}
Because of compactness, there is a sequence $\delta_n\to 0$, and an 
$\{1,\ldots, d\}$-valued random variable $J$ such that $J_{\delta_n}
\xRightarrow{n\to\infty} J$. By continuity, we thus obtain from \eqref{entrep} the representation
\begin{equation} \label{entrep1} \mathbf P(\underline X^\ast(t) \in (.))  =
  \mathbf E [\mathbb P^{J,t}(\underline F^{J,t}\in (.))], \quad t > 0.
\end{equation}
We claim that this representation is unique. Indeed, let $J'$ be a 
$\{1,\ldots, d\}$-valued random variable whose distribution is different from that of $J$, and which obeys
\begin{equation} \label{JJ} 
  \mathbf E [\mathbb P^{J,t}(\underline F^{J,t}\in (.))] = \mathbf E
  [\mathbb P^{J',t}(\underline F^{J',t}\in (.))], \quad t > 0.
\end{equation}
Then there must exist an $i \in \{1,\ldots, d\}$ such
that $\mathbf P(J=i) < \mathbf P(J'=i)$. On the other hand, from
Remark~\ref{rem:asexfr},
\begin{equation} \label{JJJ} \limsup_{t\to 0} \frac{\mathbf E[\mathbb
    E^{J,t}[F_i^{J,t}]]}{\mathbf E[ \mathbb E^{J',t}[F_i^{J',t}]]}=
  \limsup_{t\to 0} \frac{\sum_{j=1}^d  \mathbf P(J=j) \mathbb
    E^{j,t}[F_i^{j,t}]}{\sum_{j=1}^d  \mathbf P(J'=j) \mathbb
    E^{j,t}[F_i^{j,t}]} = \frac{\mathbf P(J=i)}{\mathbf P(J'=i)}< 1,
\end{equation}
which contradicts \eqref{JJ}.

From \eqref{entrep1} and  \eqref{entra} we obtain the
representation
$$\mathbf P(\underline X^\ast(t) \in (.))  = \mathbf 
E [\mathbb P^{J,t}(\underline F^{J,t}\in (.))] = \mathbf E [\mathbf
P^J_{\underline 0}(\underline X^\ast(t) \in (.))], \quad t>0,$$
which shows that every entrance law from $\underline 0$ is a convex
combination of the entrance laws
$\mathbf P^i_{\underline 0}(\underline X^\ast \in (.))$,
$i=1,\ldots, d$. To see the extremality of the latter, note that by
the same reasoning which led to the contradiction of \eqref{JJ} and
\eqref{JJJ}, the equality
$$\mathbf P_{\underline 0}^i(\underline X^\ast(t) \in (.))  =  
\mathbf E [\mathbf P^J_{\underline 0}(\underline X^\ast(t) \in
(.))],\quad t>0$$ is impossible unless $\mathbf P(J=i)=1$. This
completes the proof of Theorem \ref{t.1}.

\section[Proofs]{Proof of Theorem \ref{T2}}
\label{S:proofs2}
\subsection{Heuristics}\label{sec4.1}
Before we come to the formal proofs, we give a summary of all three
cases. Some basic ideas will be formalised in a few lemmas that are
collected in Section~\ref{S:lemmas}. The basis of our proof is the
ancestral selection graph and the approximate representation of the
fixation time in Proposition~\ref{P:main}. Moreover, by our
interpretation of the $d$ extremal entrance laws (see
Remark~\ref{rem:inter}) and symmetry, we can consider the situation
when the ASG has a single marked particle in colony~1.  Recall from
Definition \ref{def:small} that this marked particle $\bullet$ is of
the form $(1, U)$ for a [0,1]-uniformly distributed $U$.

It is important to note that at all times during the sweep, $L_t^i$
from Proposition~\ref{P:main} (which is the same as the number of
particles in $\mathcal Y$ with jump kernel $\mathscr D^a$ from
Section~\ref{ASGdef}, started in $\mathscr P \cup \{\bullet\}$) in colony $i$ is about $2\alpha\rho_i$ with high probability, see
Lemma~\ref{l:numberASGnew}. Within $\mathcal Y$, we distinguish
between marked particles (comprising
${\underline M_t} = (M_t^1,\dots,M_t^d)$ with
$M_{t}^i := \# \big(\mathscr C_t(\{\bullet\})
\cap(\{i\}\times[0,1])\big)$)
and wildtype particles; see also \eqref{eq:prop3.1a}.

Let us turn to {\em case 1}. Here, migration happens at rate of order
$\alpha$. Since splitting events of marked {particles} in
$(\underline M_t)_{t\geq 0}$ happen at rate $\alpha$ as well, marked
particles are present quickly (i.e.\ after time of order $1/\alpha$)
in all colonies. More precisely, the number of {particles} of the
$\mathpzc B$ allele $(M_1(t) + \cdots + M_d(t))_{t\geq 0}$ is close to
a pure branching process with branching rate $\alpha$ in this starting
phase. Then, when the number of {particles} exceeds
$\alpha\varepsilon$ (for some small $\varepsilon$), the {particles}
start to coalesce and the process is not pure branching any more. The
time when this happens is roughly
$(\log(\varepsilon\alpha))/\alpha \approx \log(\alpha)/\alpha$;
compare with Lemma~\ref{l:1}. Rescaling time by a factor of $\alpha$,
we can see -- using an ordinary differential equation -- that the time
the system needs to reach at least $2\alpha \rho_i(1-\varepsilon)$
{particles} in colony $i$, $i=1, {\dots,d}$, is of order $1/\alpha$
and hence is negligible for our claim. When there are
$2\alpha \rho_i(1-\varepsilon)$ marked {particles} in colony $i$,
there are about $\varepsilon 2\alpha$ wildtype {particles} in
total. Any wildtype line performs a subcritical branching process with
splitting rate $\alpha$ (which is the splitting rate within the ASG)
and death rate at least
$\frac{1}{\rho_i}2\alpha \rho_i(1-\varepsilon) = 2\alpha
(1-\varepsilon)$
(which is the coalescence rate with one of the
$2\alpha \rho_i(1-\varepsilon)$ marked {particles} within the same
colony). The extinction time of such a subcritical branching process
can be computed to be about $\log(\alpha)/\alpha$; see
Lemma~\ref{l:2}. In total, this gives a fixation time
$2\log(\alpha)/\alpha$.

~

Now we come to {\em case 2}, where migration happens at rate of order
$\alpha^\gamma$. For simplicity let us consider the case of two colonies
first. The number of marked {particles}
increases exponentially at rate $\alpha$ in colony~1, so the number of
 {particles} at time $(1-\gamma)\log(\alpha)/\alpha$
is $\exp((1-\gamma)(\log\alpha)) = \alpha^{1-\gamma}$. Since the migration rate
is of the order $\alpha^\gamma$, the first migrant to colony~2 arises
exactly by that time. Indeed, the total rate of migration is of order
$\alpha^{1-\gamma}\alpha^\gamma=\alpha$, but at time
$(1-\gamma-\varepsilon)\log(\alpha)/\alpha$ the total migration rate was
only $\alpha^{1-\gamma-\varepsilon}\alpha^\gamma =
\alpha^{1-\varepsilon}$. Moreover, we note that at time
$(1-\gamma+\varepsilon)\log(\alpha)/\alpha$ there are already
$\alpha^{1+\varepsilon}$ migrants, such that the first migrant occurs
around time $(1-\gamma)\log(\alpha)/\alpha$. After the first
migrant arises, its offspring starts to expand exponentially at rate
$\alpha$ in colony 2. After another time $x \log(\alpha)/\alpha$, it
increased in frequency to $\alpha^x$ 
{particles}. Moreover, the number of migrants from colony 1
(in the case $x<\gamma$, i.e.\ during the exponential growth phase in
colony 1) is $\int_0^{x\log(\alpha)/\alpha} \alpha^{1-\gamma}e^{\alpha
  t}\alpha^\gamma dt \approx \alpha^x$ which indicates that the number of
 marked {particles} in colony
2 is of order $\alpha^x$ by time $(1-\gamma+x)\log(\alpha)/\alpha$ for
$x<\gamma$; see also 2.\ in Lemma~\ref{l:1}. After time
$\log(\alpha)/\alpha$, the exponential growth phase in colony 1 is
over and the marked {particles} in colony
2 still increase exponentially due to splitting events in colony 2. At
time $(2-\gamma)\log(\alpha)/\alpha$, the exponential growth phase in both
colonies is over and -- as in case 1 -- it takes time of order
$1/\alpha$ until there are at least $2\alpha \rho_i(1-\varepsilon)$
 {particles} in colony $i$, $i=1,2$. Again, we
can consider the total number of wildtype  
{particles} and approximate it by a subcritical branching
process which dies after time about $\log(\alpha)/\alpha$; see again
Lemma~\ref{l:2}. Hence, the fixation time is about
$(3-\gamma)\log(\alpha)/\alpha$.
\\
For more than two colonies, it is clear that infection of a new colony
happens if and only if a neighbouring colony has about $\alpha^{1-\gamma}$
marked particles, which happens some time $(1-\gamma)\log(\alpha)/\alpha$
after this colony was infected. This leads to the first epidemic
model.

~

For {\em case 3}, where migration happens at rate of order $1/(\log\alpha)$,
observe that the total number of migration events between colonies in
a time of order $\log(\alpha)/\alpha$ is of order~1 (since there are
of order $\alpha$ {particles} per colony, each of which has a
migration rate of order $1/\log\alpha$). Again, we start by
considering two colonies{, $\mu = c/(\log\alpha)$,} and consider the
process on the new time-scale $d\tau = \frac{\alpha}{\log\alpha}
dt$. If the number of marked particles in colony 1 is smaller than $\alpha$, migration of a
marked particle is unlikely. At
time $\tau=1$, however, there are about $2\rho_1 \alpha$ marked particles in colony~1,
each of which migrates at rate $c/\alpha$ (on time-scale $d\tau$),
leading to an effective rate $2c\rho_1$ of migration. This means we
have to wait an exponential waiting time with rate $2c\rho_1$ for the
first migrant. After that time, the marked particles have already fixed in colony 1, but colony 2
needs another 2
time-units (on the time-scale $d\tau$) before fixation. \\
For $d$ colonies, note that a new colony {$k$} gets infected, if a
migrant from another infected colony is successful. After time
$\tau=1$, enough {particles} have accumulated on this colony such that
it can send migrants to its neighbouring colonies, hence becomes
infectious. If it is infectious, it sends migrants at rate { $2\rho_k
  a(k,j)$ to colony $j$}, which is exactly the second epidemic model.

\subsection{Some lemmas}
\label{S:lemmas}
We now state some general lemmas, which are used in the proof of
Theorem~\ref{T2}. {Recall that
  $\underline\rho=(\rho_1,\dots,\rho_d)$ constitutes the equilibrium
  distribution for the migration dynamics.}

% \begin{lemma}[$\mathcal K^\circ$ concentrated around $2\alpha
%   \underline\rho$]
%   \label{l:numberASG}
%   Let $\varepsilon>0$ and $\mathcal K^\circ = (\underline
%   K_\beta)_{\beta\geq 0}$ with $\underline K_\beta =
%   (K_\beta^1,\dots,K_\beta^d)$ be the line-counting process within the
%   spatial ASG as defined in~\ref{def:ASG} (or from
%   Proposition~\ref{P:main}), started in $\underline K(0) =
%   (K_i(0))_{i=1,\dots,d}$ with $K_i(0) \in [2\alpha \rho_i
%   (1-\varepsilon^{3/2}), 2\alpha \rho_i (1+\varepsilon^{3/2})]^d$
%   with $\mu \leq c\alpha^2$ for some $c>0$. Then,
%   \begin{align*}%\label{eq:numberASG1}
%     \limsup_{\alpha\to\infty}\mathbf P\big( \sup_{0\leq \beta\leq
%     \beta_\alpha} \max_{i=1,\dots,d} |K_\beta^i - 2\alpha \rho_i| >
%     2\varepsilon \alpha\big) \leq \varepsilon
%   \end{align*}
%   for $\beta_\alpha = o(1)$.
% \end{lemma}

\begin{lemma}[$\underline L$ concentrated around $2\alpha
  \underline\rho$]
  \label{l:numberASGnew}
  Assume $t_\alpha \downarrow 0$ and let
  {$\underline L = (\underline L_t)_{t\geq 0}$ with
    $\underline L_t = (L_t^1,\dots,L_t^d)$ follow the same dynamics as
    in Proposition~\ref{P:main}}. (Recall that { this process} depends
  on
  the parameter{s} $\alpha$ {and $\mu$}.)  \\
  Let $\mu = O(\alpha)$, $\varepsilon_\alpha\downarrow 0$ be any
  sequence such that $t_\alpha/\varepsilon_\alpha \to 0$ and
  $\mathbb P\Big(\Big|\frac{\underline L_0}{\alpha} - 2\underline
  \rho\Big|>\varepsilon_\alpha^2\Big) \to 0$. Then,
  $$\lim_{\alpha\to\infty}\mathbb P\Big(\sup_{0\leq r\leq t_\alpha} \Big|\frac{\underline L_r}{\alpha} - 2\underline \rho\Big|
  >\varepsilon_\alpha\Big) = 0.$$
\end{lemma}
Before turning to the proof of this lemma, let us observe that a sequence $\varepsilon_\alpha \downarrow 0$, which fulfills the
  requirements of Lemma~\ref{l:numberASGnew}, exists iff
  $\underline L_0/\alpha \xRightarrow{\alpha\to\infty} 2\underline\rho$.
\begin{remark}[A Lyapunov function for the limiting system]
  In the proof of the lemma, a function $h$ arises;
  see~\eqref{eq:ell}. In order to understand the form of this
  function, consider a chemical reaction network for chemical species
  $A_1,\dots,A_d$, governed by
  \begin{align}\label{eq:chem}
    A_i & \xrightarrow{\alpha} 2A_i, \qquad 2A_i
    \xrightarrow{1/\rho_i} A_i,\qquad A_i \xrightarrow{\mu b(i,j)}
    A_j.
  \end{align}
  for $i,j=1,\dots,d$. Here, the chemical species $A_i$ refers to the
   {particles} in colony $i$. (We refer the
  reader to \cite{Feinberg1979} for general notions of chemical
  reaction network theory.) For \emph{mass action kinetics}, properly
  rescaled, the vector of concentrations $\underline c =
  (c_1,\dots,c_d)$ with $c_i$ being the concentration of species $A_i$
  satisfies the dynamical system
  \begin{align}\label{eq:equi}
  \dot c_i = \alpha c_i - \frac{1}{2\rho_i} c_i^2 + 
  \mu \sum_{j\neq i} c_j b(j,i) - c_i b(i,j), \qquad i=1,\dots,d.
  \end{align}
  Since the system~\eqref{eq:chem} is weakly reversible and complex
  balanced, local asymptotic stability has been shown via the Lyapunov
  function $h(\underline c)= \sum_{i=1}^d
  ((\log(c_i/c_i^*)-1)c_i+c_i^*)$, see Proposition 5.3 in
  \cite{Feinberg1979}, where $(c_1^*, \dots, c_d^*)$
    denotes the equilibrium value of \eqref{eq:equi}. In fact, with
  $\kappa_i = c_i$ and $ 2\rho_i = c_i^\ast$, this is the function $h$
  appearing in \eqref{eq:ell} below.
\end{remark}

\begin{proof}[Proof of Lemma \ref{l:numberASGnew}]
  The generator of $\underline L^\alpha := \underline L/\alpha$ is
  \begin{align*}
    G_{\underline L^\alpha} f(\underline \kappa) & = \alpha^2
    \sum_{i=1}^d \Big(\kappa_i \big(f(\underline \kappa + \underline
    e_i/\alpha) - f(\underline \kappa)\big) \\ & \qquad \qquad \qquad
    + \frac{\kappa_i(\kappa_i-1/\alpha)}{2\rho_i} \big(f(\underline
    \kappa - \underline e_i/\alpha) - f(\underline \kappa)\big)\Big)
    \\ & \qquad \qquad \qquad + \mu\alpha \sum_{i,j=1}^d b(i,j)
    \kappa_i\big( f(\underline \kappa + \underline e_j/\alpha -
    \underline e_i/\alpha) - f(\underline \kappa)\big)
  \end{align*}
  for functions $f: \mathbb R_+^d \to\mathbb R$. Now, define
  \begin{align}\label{eq:ell}
    h(\underline \kappa) & = \sum_{i=1}^d \Big(\Big( \log\Big(
                           \frac{\kappa_i}{2\rho_i}\Big) - 1\Big) \kappa_i + 2\rho_i\Big) = 2 +
                           \sum_{i=1}^d \Big( \log \Big(\frac{\kappa_i}{2\rho_i}\Big) -
                           1\Big)\kappa_i.
  \end{align}
  This function is strictly convex and vanishes if and
    only if $\underline \kappa = 2\underline \rho$. Hence we are done
    once we show that $\sup_{0\leq r \leq t_\alpha} h(\underline
    L^\alpha_r) \xrightarrow{\alpha\to \infty} 0$ in probability. {For
      this, we will make use of Doob's maximal inequality for
      sub-martingales and some calculations using the generator of
      $\underline L^\alpha$. Since $\log(x+\delta) \leq (\log x) +
      \frac\delta x$, for $i,j=1,\dots,d$ and $i\neq j$,}
    \begin{align*}
    h(\underline \kappa \pm \underline e_i/\alpha) - h(\underline
    \kappa) & = \Big( \log \Big( \frac{\kappa_i \pm
      1/\alpha}{2\rho_i}\Big) -
    \log\Big(\frac{\kappa_i}{2\rho_i}\Big)\Big)(\kappa_i \pm \tfrac
    1\alpha) \\ & \qquad \qquad \qquad \qquad \qquad \quad\pm \frac
    1\alpha\Big(\log\Big(\frac{\kappa_i}{2\rho_i}\Big) - 1\Big) \\ & =
    \pm \frac 1\alpha \Big( \log \Big(\frac{\kappa_i \pm
      1/\alpha}{2\rho_i}\Big) - 1\Big) + \kappa_i \log\Big(1 \pm
    \frac{1}{\alpha \kappa_i}\Big) \\ & \leq \pm \frac 1\alpha
    \log\Big(
    \frac{\kappa_i \pm 1/\alpha}{2\rho_i}\Big),\\
    h(\underline \kappa + \underline e_j/\alpha - \underline
    e_i/\alpha) - h(\underline \kappa) & \leq \frac 1\alpha\Big( \log
    \Big(\frac{\kappa_j+1/\alpha}{2\rho_j}\Big) -
    \log\Big(\frac{\kappa_i - 1/\alpha}{2\rho_i}\Big)\Big).
  \end{align*}
  Moreover, 
  \begin{align*}
    \sum_{i,j=1}^d b(i,j) \Big(\kappa_j \frac{\rho_i}{\rho_j} -
    \kappa_i\Big) & = \sum_{j=1}^d
    \frac{\kappa_j}{\rho_j}\sum_{i=1}^d \rho_i b(i,j) -
    \sum_{i,j=1}^d \kappa_i b(i,j) \\ & = \sum_{j=1}^d
    \frac{\kappa_j}{\rho_j}\sum_{i=1}^d \rho_j b(j,i) -
    \sum_{i,j=1}^d \kappa_i b(i,j) \\ & = \sum_{i,j=1}^d \kappa_j
    b(j,i) - \kappa_i b(i,j) = 0,
  \end{align*}
  Hence, using that {$\log(x) \leq x-1$ and
    $(1-x)\log(x)\leq 0$ for all $x\geq 0$}, we obtain for sufficiently large $\alpha$ and for 
    $\underline
  \kappa \in A:=(\rho_1, 4\rho_1) \times\cdots\times (\rho_d,
  4\rho_d)$ 
  \begin{equation}
    \label{eq:761}
    \begin{aligned}
      G_{\underline L^\alpha}h(\underline \kappa) & \leq
      \alpha\sum_{i=1}^d \kappa_i \log\Big( \frac{\kappa_i +
        1/\alpha}{2\rho_i}\Big) - \frac{\kappa_i(\kappa_i-1/ \alpha )
      }{2\rho_i} \log\Big( \frac{\kappa_i - 1/\alpha}{2\rho_i}\Big) \\
      & \qquad + \mu \sum_{i,j=1}^d b(i,j) \kappa_i \Big(
      \underbrace{\log \Big( \frac{\kappa_j + 1/\alpha}{2\rho_j}\Big)
        - \log\Big(\frac{\kappa_i -
          1/\alpha}{2\rho_i}\Big)}_{{\leq
          \frac{(\kappa_j+1/\alpha)\rho_i}{(\kappa_i-1/\alpha)\rho_j}
          - 1}}\Big) \\ & \leq \sum_{i=1}^d \alpha \kappa_i \Big(
      \underbrace{\log\Big( \frac{\kappa_i - 1/\alpha}{2\rho_i}\Big) -
        \frac{\kappa_i - 1/\alpha}{2\rho_i} \log\Big(\frac{\kappa_i -
          1/\alpha}{2\rho_i}\Big)}_{\leq 0}\Big) \\ &
      \qquad \qquad \qquad \qquad \qquad \qquad \qquad \qquad \qquad
      \qquad + \frac{2 \alpha \kappa_i}{\alpha 
        (\kappa_i - 1/\alpha)}
      \\
      & \qquad + \mu \sum_{i,j=1}^d b(i,j) \Big(
      \kappa_j\frac{\rho_i}{\rho_j} - \kappa_i\Big) +
      C\frac{\mu}{\alpha} \sum_{i,j=1}^d b(i,j)
      {\frac{(\kappa_i + \kappa_j)\rho_i}{\kappa_i\rho_j}} \\
      & {\leq C'}
    \end{aligned}
  \end{equation}
  for some $C, C'>0$ which are independent of all parameters; recall that
 $\mu = \mathcal O(\alpha)$ by
  assumption.  Note that~\eqref{eq:761} shows that $(G_{\underline
    K_\alpha} h)^+$ is bounded uniformly { by $C'$} for all $\alpha$
  on the set $A$. Now, consider the martingale (recall that $g = g^+ -
  g^-$ with $g^+ = g\vee 0$ and $g^- = (-g)^+\geq 0$)
  \begin{align*}
    \Big(h(\underline L^\alpha(r\wedge T_A)) & -
    \int_0^{r\wedge T_A} (G_{\underline L^\alpha} h(\underline
    L^\alpha(s)) ds\Big)_{r\geq 0}\qquad \\ 
 & = \Big(h(\underline
   L^\alpha(r\wedge T_A)) + \int_0^{r\wedge T_A} (G_{\underline L^\alpha} h(\underline L^\alpha(s)))^- - (G_{\underline L^\alpha}
   h(\underline L^\alpha(s)))^+ ds\Big)_{r\geq 0},
  \end{align*}
  which is stopped when $\underline L^\alpha$ leaves the set $A$ at the
  stopping time $T_A$. Clearly, since $h\geq 0$,
  \begin{align*}
    \Big(h(\underline
    L^\alpha(r\wedge T_A)) + \int_0^{r\wedge T_A} (G_{\underline
    L^\alpha} h(\underline L^\alpha(s)))^-\Big)_{r\geq 0}
  \end{align*}
  is a positive submartingale. We restrict the initial state
  $\underline L^\alpha(0)$ to be in the set $A$ (this event has
  probability converging to~1 as $\alpha\to\infty$). Note that, by
  assumption, we find some $C''>0$ such that
  $\mathbb E[h(\underline L^\alpha(0))] \leq C'' \varepsilon_\alpha^2$
  and
  $\frac{t_\alpha}{\varepsilon_\alpha} \xrightarrow{\alpha\to\infty}
  0$.
  By Doob's martingale inequality, for $t_\alpha\downarrow 0$ and if
  $\varepsilon$ is small enough, for $\underline L^\alpha(0)\in A$,
  \begin{align*}
    \mathbb P( \sup_{0\leq r\leq t_\alpha} h(\underline
      L^\alpha(r)) > \varepsilon_\alpha) & = \mathbb P( \sup_{0\leq r\leq
        t_\alpha} h(\underline
      L^\alpha(r\wedge T_A)) > \varepsilon_\alpha)\qquad \\
    & \leq \mathbb P\Big(\sup_{0\leq r\leq t_\alpha} h(\underline
    L^\alpha(r\wedge T_A)) + \int_0^{r\wedge T_A} (G_{\underline
      L^\alpha} h(\underline L^\alpha(s)))^-ds > \varepsilon_\alpha\Big) \\ &
    \leq \frac {1 }{\varepsilon_\alpha}\mathbb E\Big[h(\underline
    L^\alpha(t_\alpha \wedge T_A)) + \int_0^{t_\alpha \wedge T_A}
    (G_{\underline L^\alpha} h(\underline L^\alpha(s)))^-ds\Big] \\ &
    = \frac{1}{\varepsilon_\alpha}\mathbb E\Big[ h(\underline L^\alpha(0)) +
    \int_0^{t_\alpha \wedge T_A} (G_{\underline L^\alpha} h(\underline
    L^\alpha(s)))^+ ds\Big] \\ & \leq \frac{C''\varepsilon_\alpha^2 +
      C't_\alpha}{\varepsilon_\alpha} \xrightarrow{\alpha\to\infty}
    0
  \end{align*}
  and the result follows. 
\end{proof}
\noindent
We also need a little refinement of the last lemma. Here, only bounds
on the birth and death rates are assumed.

\begin{corollary}[Particle-counting in a single colony concentrated
  around $2\alpha \rho$]
  \label{cor:numberASG}
  Let ${\mathcal V = (V_r)_{r\geq 0}}$ be a birth-death process
  with birth- and death rates ${b_k}$ and ${d_k}$ satisfying
  $$ 
\alpha k\leq {b_k} \leq \alpha k + c \alpha^{1+\gamma}, \qquad
  \frac{1}{\rho} \binom k 2 \leq {d_k} \leq \frac{1}{\rho} \binom k 2
  + c \alpha^\gamma k
$$ 
for some $\gamma\in [0,1)$ and $c\geq 0, \rho>0$. If
  $V_0/\alpha \xrightarrow {\alpha\to\infty}_p 2\rho$,
  then
  \begin{align*}%\label{eq:numberASG2}
    \sup_{0\leq r\leq t_\alpha} \Big|\frac{V_r}{\alpha} -
    2\rho\Big| \xRightarrow{\alpha\to\infty} 0
  \end{align*}
  for $t_\alpha \downarrow 0$.
\end{corollary}

\begin{proof}
  For $c=0$, the assertion would just be a special case of
  Lemma~\ref{l:numberASGnew} for a single colony. For $c>0$, we fix
  $\varepsilon>0$ and take $\alpha$ large enough such that
  $$ \alpha k \leq {b_k} \leq (\alpha + c' \alpha^\gamma) k, \qquad
  \frac 1 \rho \binom k 2 \leq {d_k} \leq \frac 1 \rho (1+\varepsilon)
  \binom k 2$$ for some $c'>0$ whenever $k\in[\alpha \rho, 4\alpha
  \rho]$. Now consider the process $\mathcal V' = (V'_r)_{r\geq 0}$ {
    ($\mathcal V'' = (V''_r)_{r\geq 0}$) } with the lower { (upper)}
  bound of ${b_k}$ and the upper { (lower)} bound of ${d_k}$ as birth-
  and death rates. Clearly, the processes $\mathcal V$, $\mathcal V'$,
  $\mathcal V''$ can be coupled such that $V_r' \leq V_r \leq V_r''$
  for all $r$ as long as $V_r, V_r', V_r'' \in [\alpha\rho,
  4\alpha\rho]$ and conclude from Lemma~\ref{l:numberASGnew} { (by
     {suitably
      modifying the proof and the value of $\alpha$} used there)} that
  \begin{align*}
    \sup_{0\leq r\leq t_\alpha} \Big| \frac{V_r'}{\alpha} -
    \frac{2 \rho}{1+\varepsilon}\Big| &\xrightarrow{\alpha\to\infty}_p
    0,
\\
    \sup_{0\leq r\leq t_\alpha} \Big| \frac{V_r''}{\alpha} - 2
    \rho \frac{\alpha + c'\alpha^\gamma}{\alpha}\Big|
    & \xrightarrow{\alpha\to\infty} 0.
  \end{align*}
  Combining the last two {limits} gives the
  result since $\varepsilon>0$ was arbitrary.
\end{proof}

\noindent
Since the processes $M_1, ...,M_d$, which count the marked particles,
are in their initial phases close to a supercritical branching process,
we need  bounds for this kind of process{es}. In the proof of
Theorem \ref{t.2} we will use the next lemma to control (i) the time
until the number of marked particles in the first colony reaches the
order $\alpha^p$, (ii) the time until another colony is infected from
the first colony (i.e.\ the occurrence of the first marked particle on
this second colony), and (iii) the time until $\alpha^\gamma$
particles are marked in the infected colony, when the migration rate
$\mu= c\alpha^\gamma$. These three asymptotics correspond to
\eqref{eq:l21a}, \eqref{eq:l21}and \eqref{eq:l22} below.  In Lemma
\ref{l:1} we will deliberately suppress the effects of back-migration.
These effects are controlled in the course of the proof of Theorem \ref{t.2} by comparison arguments.
\begin{lemma}[Asymptotic hitting times of a bivariate birth-death process]
  Let \label{l:1} $c,c',c''>0$, $\gamma, p\in (0,1]$;
  $\varepsilon_\alpha \downarrow 0, \varepsilon'_\alpha \downarrow 0$
  with $\varepsilon_\alpha, \varepsilon_\alpha' > 1/(\log\alpha)$.
  Let ${\mathcal V} = (V_t)_{t\geq 0}$ be a birth-death process with
  birth rate ${b_k} = \alpha k$ and death rate
  ${d_k} \leq c \varepsilon_\alpha \alpha k$ for
  $k\leq \varepsilon_\alpha\alpha$, started in $V_0 = 1$.  Moreover,
  conditional under ${\mathcal V}$ let $\mathcal W = (W_t)_{t\geq 0}$
  be a birth-death process with time-inhomogeneous birth rate
  $\mu V_t + \alpha W_t$ and death rate
  ${d_k} \leq c'{\varepsilon'_\alpha} \alpha^\gamma k$ for
  $k \leq \varepsilon'_\alpha\alpha^\gamma$, starting in $W_0=0$.
  Then we can conclude
  \begin{enumerate}
  \item For $n\in \mathbb N$ let $T_{n}$ be the first time when
    $V_t=n$. Then,
    $\mathbb P(T_{\varepsilon_\alpha\alpha} = \infty) \leq
    c\varepsilon_\alpha$ and for all $\varepsilon>0$
   \begin{align} 
     \label{eq:l21a}
     \mathbb P\Big(\Big|\frac{\alpha}{\log\alpha}T_{\varepsilon_\alpha{\alpha^p}}
     - p\Big| > \varepsilon\Big)
     \xrightarrow{\alpha\to\infty} 0.
    \end{align}    
  \item For $n\in \mathbb N$ let $S_n$ be the first time when
    $W_t=n$. Then, for $\mu = c''\alpha^\gamma$, and any
    $\varepsilon>0$
    \begin{align}
      \label{eq:l21}
      & \mathbb P\Big(\Big|\frac{\alpha}{\log\alpha}S_1 -
        (1-\gamma)\Big|> \varepsilon\Big)
        \xrightarrow{\alpha\to\infty} 0
        \intertext{and}
        \label{eq:l22} 
      &\mathbb
        P\Big(\Big|\frac{\alpha}{\log\alpha}S_{{\varepsilon'_\alpha}\alpha^\gamma} 
        - 1\Big| > \varepsilon\Big) \xrightarrow{\alpha\to\infty} 0.
    \end{align}
  \end{enumerate}
\end{lemma}

\begin{proof}
  1. We start with proving \eqref{eq:l21a}.  First, let $\mathcal V'$
  be a pure branching process with branching rate $\alpha$ (i.e.\
  ${b_k}' = \alpha k$ and ${d_k}'=0$){, started with $V_0'=1$} and
  $T_n'$ its hitting time of $V_t'=n$. Then we observe that, as
  $\alpha \to \infty$,
  \begin{equation}
    \label{eq:432}
    \mathbb E[T_{\varepsilon_\alpha\alpha^p}']  =
    \sum_{i=1}^{\varepsilon_\alpha\alpha^p-1} \frac{1}{\alpha i} =
    \frac{\log\alpha^p}{\alpha} + \mathcal{O}\left(\frac{\log(\varepsilon_\alpha)}{\alpha} \right),\qquad 
    \mathbb {V}[T_{{\varepsilon_\alpha}{\alpha^p}}']  =
    \sum_{i=1}^{{ \varepsilon_\alpha}{\alpha^p}- 1} \frac{1}{\alpha^2 i^2} =
    \mathcal O\Big(\frac 1{\alpha^2}\Big).
  \end{equation}
  Hence by Chebyshev's inequality
  \begin{align*}
    \mathbb P\Big(
    \Big| \frac{\alpha}{\log\alpha}T_{\varepsilon\alpha^p}' - {p} \Big| > \varepsilon\Big) 
    \leq \frac{\alpha^2\mathbb
    V[T_{{ \varepsilon_\alpha}{\alpha^p}}']}{(\log\alpha)^2 {\varepsilon^2}}
    \xrightarrow{\alpha\to\infty} 0.
  \end{align*}
  Since $T_n' \leq T_n$ stochastically for all $n$, this implies  
   \begin{align*}
     \mathbb P \Big( \frac{\alpha}{\log\alpha}T_{{\varepsilon_\alpha}\alpha^p} -
     {p} < - \varepsilon \Big)
     \xrightarrow{\alpha\to\infty} 0.
  \end{align*}
    
  For the second bound in \eqref{eq:l21a} we consider a process
  $\mathcal V'' = (V_t'')_{t\geq 0}$ with ${b_k}'' = \alpha k$ and
  ${d_k}'' = c\varepsilon_\alpha \alpha k$ { with $V_0''=1$}, and its
  hitting time $T_n''$ of $n$. Within the branching process
  $\mathcal V''$ we consider the immortal {lines}, i.e. the process of
  those particles which have descendants at any later time. By
  classical theory \cite[Chapter I.5]{AthreyaNey1972}, the probability
  that a single line will not be immortal equals the solution of
  $\frac{c\alpha\varepsilon_\alpha}{\alpha(1+c\varepsilon_\alpha)} +
  \frac{\alpha}{\alpha(1+c\varepsilon_\alpha)} x^2 = x$,
  which is smaller than~1, and hence equals $c\varepsilon_\alpha$. So,
  $\mathbb P(T_{{\varepsilon_\alpha}\alpha^p} {<}\infty)
  \geq 1-c\varepsilon_\alpha$
  follows and assuming
  $T_{{\varepsilon_\alpha}\alpha}<\infty$ we can restrict
  ourselves in the sequel to the event that the (single) initial
  particle of $\mathcal V''$ is immortal. Moreover, when an immortal {
    particle} splits in $\mathcal V''$, the new {particle} has the
  chance $1-c\varepsilon_\alpha$ to be immortal. So, every splitting
  event leads to a new immortal {particle} with probability
  $1-c\varepsilon_\alpha$, so $\mathcal V''$ (given it starts with a
  {single} immortal {particle}) is bounded from below by a binary pure
  branching process $\mathcal V'''$ with individual branching rate
  $\alpha(1-c\varepsilon_\alpha)$. For $n\in \mathbb N$, let $T_n'''$
  be the time it takes $\mathcal V'''$ to reach $n$.  Then
  $T_n'''\geq T_{n}''$ {stochastically} for all $n$, on the event that
  {$\mathcal V''$} starts with an immortal particle { at time 0}. On
  the other hand it is clear that, for all $n \in \mathbb N$,
  $T_n'' \geq T_n$ {stochastically}. Hence we obtain by the same
  calculations as in~\eqref{eq:432}, now applied to the process
  $\mathcal V'''$, the estimate
 \begin{align*}
   \lim_{\alpha\to\infty}
   \mathbb P \Big(
   \frac{\alpha}{\log\alpha} T_{{\varepsilon_\alpha}\alpha^{p}} -
   {p}  > \varepsilon \Big) 
   &
     \leq    \lim_{\alpha\to\infty}
     \mathbb P\Big(
     \frac{\alpha}{\log\alpha}T_{{\varepsilon_\alpha} \alpha^{p}}''' - {{p}} > \varepsilon\Big)
   \\
   & =    \lim_{\alpha\to\infty}
     \mathbb P\Big(\frac{\alpha(1-\frac{ c\varepsilon_\alpha}
     {2})}{\log(\alpha(1-\frac{ c\varepsilon_\alpha}{2}))} T_{{\varepsilon_\alpha}\alpha^{p}}'''
     - {p} > \varepsilon
     \Big) \\ & \leq \lim_{\alpha\to\infty} 
                \frac{\alpha^2\mathbb
                V[T_{{\varepsilon_\alpha}\alpha^{{p}}}''']}{\log(\alpha)^2
                \varepsilon^2} = 0.
  \end{align*}
  This completes the proof of \eqref{eq:l21a}.

  2. For the proof of \eqref{eq:l21}, we again use comparison
  arguments based on the processes $\mathcal V{'}$ and $\mathcal V'''$
  defined in the first part of the proof.  Having in mind that
  $ V_t''' \leq V_t \leq V_t'$stochastically as long as
  $V_t'\leq {\varepsilon_\alpha}\alpha$, we introduce the birth
  processes $\mathcal W' = (W'_t)_{t\geq 0}$ and
  $\mathcal W''' = (W_t''')_{t\geq 0}$, whose birth rates, conditional
  on $\mathcal V{'}$ resp. $\mathcal V'''$ are
  $\mu V_t' + \alpha W_t' $ and $\mu (V_t'''- W_t''')+ \alpha W_t'''$,
  respectively.  Also, we assume $W'_0=W'''_0 =0$.  Let $S_1'$ and
  $S_1'''$ be the first jump times of $\mathcal W'$ and of
  $\mathcal W'''$ (from $0$ to $1$).  From this construction, it is
  clear that $S_1' \leq S_1 \leq S_1'''$ stochastically. We claim
  that, on the event $\{T_{\varepsilon_\alpha \alpha}<\infty\}$, for any
  $\varepsilon>0$,
  \begin{align}\label{eq:SP}
    \mathbb P\Big(\frac{\alpha}{\log\alpha}S_1' - (1-\gamma) < -
    2\varepsilon\Big) \xrightarrow{\alpha\to\infty} 0 \intertext{as
    well as} \label{eq:SPPP} \mathbb
    P\Big(\frac{\alpha}{\log\alpha}S_1''' - (1-\gamma) > 2\varepsilon\Big)
    \xrightarrow{\alpha\to\infty} 0
  \end{align}
  which together imply the assertion
  \eqref{eq:l21}. For~\eqref{eq:SP}, let $L'$ be the number of
  {particle}s in $\mathcal V'$ {at the time when {$\mathcal W'$}
    reaches $1$ for the first time.}  Then, $L'$ is geometrically
  distributed with success parameter
  $\frac{c''\alpha^\gamma}{\alpha + c''\alpha^\gamma} =
  \frac{c''}{\alpha^{1-\gamma} + c''}$
  and thus
  $\mathbb P(L'<\alpha^{1-\gamma-\varepsilon})
  \xrightarrow{\alpha\to\infty} 0$.
  Recalling that $T_n'$ is the first time when $V_t'=n$, we conclude
  by
  \begin{align*}
    \lim_{\alpha\to\infty}\mathbb P\Big(\frac{\alpha}{\log\alpha}S_1'
    - (1-\gamma) < - 2\varepsilon  \Big)  
    & =  \lim_{\alpha\to\infty}\mathbb P\Big(\frac{\alpha}{\log\alpha}S_1'
      -(1-\gamma) < - 2\varepsilon, L' \geq \alpha^{1-\gamma-\varepsilon}\Big) \\ 
    &
      \leq \lim_{\alpha\to\infty}\mathbb P\Big(\frac{\alpha}{\log\alpha}
      T'_{\alpha^{1-\gamma-\varepsilon}} -(1-\gamma) < -2\varepsilon \Big) = 0,
  \end{align*}
  where the last equality follows by a similar calculation as in
  1. For~\eqref{eq:SPPP}, let $L'''$ be the number of {particle}s in
  $\mathcal V'''$ {at the time when {$\mathcal W'''$} reaches $1$ for
    the first time.} Then, $L'''$ is geometrically distributed with
  success parameter
  $\frac{c''\alpha^\gamma}{\alpha(1-c\varepsilon_\alpha) +
    c''\alpha^\gamma} =
  \frac{c''}{\alpha^{1-\gamma}(1-c\varepsilon_\alpha) + c''}$
  and thus
  $\mathbb P({L'''}>\alpha^{1-\gamma+\varepsilon})
  \xrightarrow{\alpha\to\infty} 0$. Similarly as above we observe that
  \begin{align*}
    \lim_{\alpha\to\infty}\mathbb
    P\Big(\frac{\alpha}{\log\alpha}S_1''' - (1-\gamma) > 2\varepsilon\Big) 
    & = \lim_{\alpha\to\infty}\mathbb
      P\Big(\frac{\alpha}{\log\alpha}S_1''' -( 1 - \gamma) < 2\varepsilon, L'''
      \leq \alpha^{1-\gamma+\varepsilon}\Big) \\ 
    & \leq
      \lim_{\alpha\to\infty}\mathbb P\Big(\frac{\alpha}{\log\alpha}
      T'''_{\alpha^{1-\gamma+\varepsilon}} - ( 1 - \gamma) < 2\varepsilon\Big) = 0.
  \end{align*}
  This concludes the proof of \eqref{eq:l21}. 
  
  Let us now turn to the proof of~\eqref{eq:l22}. Using \eqref{eq:l21}
  we can work on the event
  $$\Big\{\Big|\frac{\alpha}{\log\alpha} S_1 -( 1 - \gamma)\Big| <
  \varepsilon\Big\}\cap \{T_{\varepsilon_\alpha\alpha} < \infty\}.$$
  Then the time it takes to have
  $W_t = \varepsilon'_\alpha\alpha^\gamma$ is {stochastically} smaller
  than the waiting time until one {particle} { starting at time
    $(1 - \gamma + 2\varepsilon)\frac{\log\alpha}{\alpha}$} has
  $\varepsilon'_\alpha\alpha^\gamma$ offspring if we take the birth
  rate to be $\alpha k$ and the death rate to be
  ${c' \varepsilon_\alpha'}\alpha^\gamma k$. This time, in turn, is
  smaller than the time until { the number of immortal lines
    $\mathcal U$ in the latter process reaches}
  $ {c' \varepsilon'_\alpha} \alpha^\gamma$.  { ({In fact,}
    $\mathcal U$ is a pure branching process with individual branching
    rate $(1-c'\varepsilon_\alpha'\alpha^{\gamma-1})\alpha$.)}  Hence,
  by the same calculation as in the proof of part 1., now denoting  by $T'_n$ the first time when $U_t=n$
  \begin{align*}
    \lim_{\alpha\to \infty} \mathbb P\Big( \frac{\alpha}{\log\alpha} 
    &
      S_{\varepsilon'_\alpha\alpha^\gamma} -1 > 3\varepsilon\Big)  
    \\ & =
         \lim_{\alpha\to \infty} \mathbb P\Big(
         \frac{\alpha}{\log\alpha}S_{\varepsilon'_\alpha\alpha^\gamma} -1 
         >3\varepsilon, 
         \frac{\alpha}{\log\alpha} S_1 {< 1 - \gamma +
         2\varepsilon}\Big) \\
    & \leq \lim_{\alpha\to\infty} \mathbb
      P\Big({ \frac{\alpha}{\log\alpha}
      T'_{\varepsilon'_\alpha \alpha} - 1  > \varepsilon \Big| \frac{\alpha}{\log\alpha} T_1' 
      = (1-\gamma+2\varepsilon)} \Big) = 0.
  \end{align*}
  This proves one of the bounds in \eqref{eq:l22}. For the other bound
  we work again with $\mathcal V'$, the pure branching process with
  individual branching rate $\alpha$ started in $V'_0=1$, and note
  that $\mathbb E[V_t] \leq \mathbb E[V_t'] = e^{\alpha t}$. Again,
  conditional on $\mathcal V'$, let $\mathcal W'$ be a birth-death
  process with time-inhomogeneous birth rate $\mu V_t' + \alpha W_t'$
  and death rate~0, now starting at time
  $s = (1-\gamma-2\varepsilon)\frac{\log \alpha}{\alpha}$ with
  $W_s'=1$, and recall
  $\mathbb E[V_s']=e^{\alpha s} =
  \alpha^{1-\gamma-2\varepsilon}$.
  Then, the time it takes to have
  $W_t={\varepsilon_\alpha'}\alpha^\gamma$ is stochastically larger
  than the hitting time of ${\varepsilon_\alpha'} \alpha^\gamma$ of
  the process $\mathcal W'$. We have that
  $\frac d{dt} \mathbb E[W_t'] = \mu \mathbb E[V_t'] + \alpha \mathbb
  E[W'_t]$, $t\ge s$, $W_s' = 1$, which is solved by
    $$ \mathbb E[W_t'] 
    = \frac{e^{\alpha t}}{\alpha} \big( \alpha^{\gamma+2\varepsilon} +
    \alpha \mu t - \mu(1-\gamma-2\varepsilon)\log\alpha\big), \quad t \ge s.$$ Therefore,
    with $\mu=c''\alpha^\gamma$ and $t =
    (1-3\varepsilon)\frac{\log\alpha}{\alpha}$, using Markov's
    inequality,
    \begin{align*}
      \lim_{\alpha\to\infty} \mathbb P\Big( \frac{\alpha}{\log\alpha}
      S_{{\varepsilon_\alpha'}\alpha^\gamma} < 1 - 3\varepsilon\Big) & \leq
      \lim_{\alpha\to\infty} \mathbb P\Big(
      W'_{(1-3\varepsilon)\frac{\log\alpha}{\alpha}} > {\varepsilon_\alpha'}
      \alpha^\gamma\Big) \\ & \leq \lim_{\alpha\to\infty}
      \frac{\alpha^{1-3\varepsilon}}{{\varepsilon_\alpha'}
        \alpha^{1+\gamma}}\big(\alpha^{\gamma+2\varepsilon} +
      c''\alpha^\gamma (\gamma-\varepsilon) \log\alpha\big) = 0,
  \end{align*}
  which completes the proof of \eqref{eq:l22}.
 \end{proof}

 \noindent The following is a direct consequence of Lemma \ref{l:1} in
 the case of $d$ colonies.

 \begin{corollary}\label{r:1} Assume the birth-death process
   $\mathcal V$ with the same rates as in Lemma \ref{l:1} starts in
   $V_0=k$ for $k\in \mathbbm{N}$, and consider not a single
   birth-death process $\mathcal{W}$, but $\ell$ birth-death processes
   $\mathcal{W}^1$, ..., $\mathcal{W}^\ell$, which, conditional under
   $\mathcal V$, have birth rate $\mu V_t + \alpha W^{i}$ for
   $i \in \{0, ..., \ell\}$ and death rate
   $d_k \leq c\varepsilon'_\alpha \alpha^\gamma k$ for
   $k\leq \varepsilon'_\alpha \alpha^{\gamma}$ (again with the
   notation and assumptions from Lemma \ref{l:1}). Let
   $n\in \mathbb N$ and $S^{i}_n$ be the first time when
   $W^{i}_t=n$. Then, for $\mu = c''\alpha^\gamma$ and any
   $\varepsilon>0$,
   \begin{align*}& \mathbb P\Big(\Big|\frac{\alpha}{\log\alpha}S^{i}_1
     - (1-\gamma)\Big| > \varepsilon, i \in \{0, ..., \ell\}\Big)
     \xrightarrow{\alpha\to\infty} 0 \intertext{ and } &\mathbb
     P\Big(\Big|\frac{\alpha}{\log\alpha}S^{i}_{\varepsilon'_\alpha
       \alpha^{\gamma}} - 1\Big| > \varepsilon, i \in \{0, ...,
     \ell\}\Big) \xrightarrow{\alpha\to\infty} 0.
   \end{align*}
 \end{corollary}

\noindent
We now complement Lemma \ref{l:1} to cover also the case in which the
process {${\mathcal V}$} starts in ${c'} \alpha^\gamma$ { for some
  $c'>1$} instead of~1. {This lemma will be used later to control the
  time until of order $\alpha$ particles are marked when one starts
  with $c' \alpha^\gamma$ marked particles.}

\begin{lemma}[Exponential growth of a near-exponential process]
  Let \label{l:4} $0\leq \gamma< p \leq 1$, $c', c''>0$ and
  $c_\alpha, \varepsilon_\alpha\downarrow 0$ be sequences with
  $\varepsilon_\alpha>1/(\log\alpha)$ and
  $\log c_\alpha \in o(\log\alpha)$. Let
  ${\mathcal V} = (V_t)_{t\geq 0}$ be a birth-death process with birth
  rate ${b_k}$ with
  $\alpha k \leq {b_k} \leq \alpha k + c'\alpha^{1+\gamma}$ and death
  rate ${d_k} \leq c''\varepsilon_\alpha\alpha k$ for
  $k\leq \varepsilon_\alpha\alpha$, started in
  $V_0 = c_\alpha \alpha^\gamma$. Let $T_{n}$ be the first time when
  $V_t=n$. 
  
  Then, for all $\varepsilon>0$,
  \begin{align}\label{eql:4}
    \mathbb P\Big(\Big|\frac{\alpha}{\log\alpha}T_{\varepsilon_\alpha{\alpha^p}}
    - ({p}-\gamma)\Big|  >
    \varepsilon \Big) \xrightarrow{\alpha\to\infty} 0.
  \end{align}
\end{lemma}

\begin{proof}
  We need to take two bounds for the process $\mathcal V$. Let
  $\mathcal V' = (V_t')_{t\geq 0}$ be a birth-death process with birth
  rate ${b_k'} = \alpha k + c'\alpha^{1+\gamma}$, death rate ${d_k}=0$
  and $V_0'= c_\alpha \alpha^\gamma$. If
  $T_{\varepsilon_\alpha{\alpha^p}}'$ is the first time when
  $V_t'=\varepsilon_\alpha{\alpha^p}$, it is clear that
  $T_{\varepsilon_\alpha {\alpha^p}}'\leq T_{\varepsilon_\alpha
    {\alpha^p}}$ stochastically.

  We define ${\mathcal W'} = (W_t')_{t\geq 0}$ with
  $W'_t := \frac {\log V'_{t \log(\alpha)/\alpha}}{\log \alpha}$,
  i.e.\ $V'_{t \log(\alpha)/\alpha} = \alpha^{{W'_t}}$ and
  $W_0'=\gamma+ \frac{\log c_\alpha}{\log \alpha}\in \gamma + o(1) $.
  Note that
  $\frac{\alpha}{\log\alpha}T_{\varepsilon_\alpha {\alpha^p}}'$ is the
  time when $\mathcal W'$
  hits~$p + \frac{\log \varepsilon_\alpha}{\log \alpha} \in p + o(1)$.
  Let $G'$ be the generator of $\mathcal W'$. Then, for $x>\gamma$
  \begin{align*}
    G'f(x) & = (\log\alpha) (\alpha^x +
    c\alpha^\gamma)(f(\underbrace{\tfrac 1{\log \alpha}\log (\alpha^{x} +
      1)}_{\approx x+ \tfrac 1{\log\alpha}\alpha^{-x}}) - f(x))
    \xrightarrow{\alpha\to\infty}f'(x).
  \end{align*}
  {Consequently, and since $W_t'$ quickly leaves its initial state
    $W_0'\in\gamma + o(1)$, by Theorem 4.2.11 in \cite{EthierKurtz86}
    the process $\mathcal W'$ converges as $\alpha\to\infty$ on the
    subsets
    $E_{\alpha}:= \{ \frac { \log k}{\log \alpha} : \log k \in
    \mathbb{N}, k \geq \gamma \log \alpha + \log c_\alpha \}$
    to the (right continuous) process with semigroup
    $T(t)f(x)= f(x+t)$ for $x\geq \gamma$}, growing linearly and
  deterministically at speed~1. Since $W_0'\in \gamma + o(1)$, it
  hits~${p} + o(1)$ asymptotically as $\alpha \to \infty$ at time ${p} -\gamma$ and
  \begin{align*}
    \mathbb P\Big(\frac{\alpha}{\log\alpha}T_{\varepsilon_\alpha {\alpha^p}} - ({p}-\gamma) < -
    \varepsilon\Big) & \leq \mathbb
                       P\Big(\frac{\alpha}{\log\alpha}T_{\varepsilon_\alpha {\alpha^p}}' -({p}
                       - \gamma) < -
                       \varepsilon\Big) \xrightarrow{\alpha\to\infty} 0. 
  \end{align*} 
  On the other hand, consider the process
  $\mathcal V'' = (V''_t)_{t\geq 0}$ with birth rate
  ${b_k''} = \alpha k$, death rate
  ${d_k} = c''\varepsilon_\alpha \alpha k$ and
  $V_0''= c_\alpha \alpha^\gamma$, as well as the time
  $T_{\varepsilon_\alpha {\alpha^p}}''$ when this process hits
  $\varepsilon_\alpha {\alpha^p}$. Again, consider
  $\mathcal W'' = (W''_t)_{t\geq 0}$ with
  $W''_t := \frac {\log V''_{t \log(\alpha)/\alpha}}{\log \alpha}$ and
  note that
  $\frac{\alpha}{\log\alpha}T''_{\varepsilon_\alpha {\alpha^p}}$ is
  the time when $\mathcal W''$
  hits~${p} + \frac{\log \varepsilon_\alpha}{\log \alpha} \in p +
  o(1)$.
  Then, as above, if $G''$ is the generator of $\mathcal W''$, for
  smooth $f$,
  \begin{align*}
    G''f(x) & = (\log\alpha)\alpha^x
              (f(\tfrac 1{\log \alpha}\log (\alpha^{x} +
              1)) - f(x)) + c'' (\log\alpha) \varepsilon_\alpha \alpha^x (f(\tfrac 1{\log \alpha}\log (\alpha^{x} -
              1)) - f(x))
    \\ & \xrightarrow{\alpha\to\infty} f'(x)
  \end{align*}
  and, since
  $W_0''=\gamma + \frac{\log c_\alpha}{\log \alpha} \in \gamma +
  o(1)$,
  the process $\mathcal W''$ hits 
    $ {p}+ \frac{\log \varepsilon_\alpha}{\log \alpha}\in p + o(1)$
    asymptotically at time ${p}-\gamma$ and
    \begin{align*}
      \mathbb P\Big(\frac{\alpha}{\log\alpha}T_{\varepsilon_\alpha {\alpha^p}} - ({p}-\gamma) >
      \varepsilon\Big) & \leq \mathbb
                         P\Big(\frac{\alpha}{\log\alpha}T_{\varepsilon_\alpha {\alpha^p}}'' -({p}
                         - \gamma) >
                         \varepsilon\Big) \xrightarrow{\alpha\to\infty} 0. 
    \end{align*} 
\end{proof}

\medskip

\noindent
While the last two lemmata were about supercritical branching
processes, we also need the following result about the extinction time
of a process which is close to a subcritical branching process.

\begin{lemma}[Extinction time of a birth-death process]\label{l:2}
  Let $c>0$ and $\varepsilon_\alpha\downarrow 0$. Let
  ${\mathcal V} = (V_t)_{t\geq 0}$ be a birth-death process with birth
  rate ${b_k} = \alpha k$ and death rate ${d_k}$ such that
  $ \alpha (2-\varepsilon_\alpha) k \leq {d_k} \leq
  \alpha(2+\varepsilon_\alpha) k$,
  started in $V_0 = z_\alpha \alpha$ with $z_\alpha \rightarrow z$ for
  some $z >0$. Let $T_{z_\alpha \alpha}$ be the extinction time of
  $\mathcal V$, i.e.\ the first time when $V_t=0$. 
  
  Then,
  for all $\varepsilon>0$,
  $$ \mathbb P\Big(  \Big|\frac{\alpha}{\log\alpha}T_{z_\alpha \alpha} - 1\Big| > \varepsilon\Big)
  \xrightarrow{\alpha\to\infty} 0.$$
\end{lemma}

\begin{proof}
  As a first step, consider a sub-critical branching process
  ${\mathcal W} = (W_t)_{t\geq 0}$ with birth rate $\alpha$ and death
  rate $ \alpha(1+x_\alpha)$, where $x_\alpha\downarrow x$ with
  $x>0$. Let $S^{x_\alpha}_1$ be the extinction time,
  when the process is started in a single {particle}, $W_0=1$. Then,
  from classical theory (see e.g.\ \cite[Chapter~V~(3.4)]{Harris1963})
  it follows, that
  $$
  f(t):=\mathbb P(S^{{x_\alpha}}_1>t\mid W_0=1) =\frac{{{x_\alpha}}}{(1+{{x_\alpha}})
    e^{t\alpha {{x_\alpha}}} - 1}.
  $$ 
  Now, consider the same branching process, but started in
  $W_0 = z_\alpha\alpha$ and denote its extinction time by
  $S^{{x_\alpha}}_{z_\alpha \alpha}$. Then,
  $g(t) := \mathbb P(S^{x_\alpha}_{z_\alpha \alpha}>t)$ satisfies
  \begin{align*}
    g(t) & = 1 - (1-f(t))^{z_\alpha \alpha}.
  \end{align*}
  Hence, for $\varepsilon>0$,
  \begin{equation}
    \label{eq:2198}
    \begin{aligned}
      \mathbb{P} \Bigg( \frac{\alpha} {\log \alpha}
      S^{{x_\alpha}}_{{z_\alpha}\alpha} -
      \frac{1}{{x_\alpha}} > \varepsilon \Bigg) & = g
      \Bigg( \frac{\log \alpha}{ \alpha} \Bigg(
      \frac{1}{{x_\alpha}} + \varepsilon \Bigg) \Bigg)
      { = 1- \left(1-
          \frac{x_\alpha}{(1+x_\alpha)\alpha^{1+x_\alpha\varepsilon}-1}\right)^{z_\alpha
          \alpha}}
      \\ & \xrightarrow{\alpha\to\infty} 0,\\
      \mathbb{P} \Bigg( \frac{\alpha} {\log \alpha}
      S^{{x_\alpha}}_{{z_\alpha}\alpha} -
      \frac{1}{{x_\alpha}} < - \varepsilon \Bigg) &
      = 1- g \Bigg( \frac{ \log \alpha}{ \alpha} \Bigg(
      \frac{1}{{x_\alpha}} - \varepsilon\Bigg)
      \Bigg) = 1- \left(1-
        \frac{x_\alpha}{(1+x_\alpha)\alpha^{1-x_\alpha
            \varepsilon}}\right)^{z_\alpha \alpha} \\ &
      \xrightarrow{\alpha\to\infty} 0.
    \end{aligned}
  \end{equation}
  Stochastically,
  $S_{z_\alpha \alpha}^{1+\varepsilon_\alpha} \leq T_{z_\alpha \alpha}
  \leq S_{z_\alpha \alpha}^{1-\varepsilon_\alpha}$
  and hence,
  \begin{align*}
    \mathbb P\Big( \frac{\alpha}{\log\alpha}T_{{z_\alpha}\alpha} - 1 < -
    2 \varepsilon\Big) & \leq \mathbb P\Big(
                         \frac{\alpha}{\log\alpha}S^{1+\varepsilon_\alpha}_{{z_\alpha} \alpha} -
                         {\frac{1}{1+\varepsilon_\alpha}} < -
                         2\varepsilon +
                         \frac{\varepsilon_\alpha}{1+\varepsilon_\alpha}\Big)
                         \xrightarrow{\alpha\to\infty} 0 
                         \intertext{as well as, by~\eqref{eq:2198}, } 
                         \mathbb     P\Big( \frac{\alpha}{\log\alpha}T_{{z_\alpha}\alpha} - 1 >
                         2\varepsilon\Big) 
    & \leq \mathbb P\Big(
      \frac{\alpha}{\log\alpha}S^{1-\varepsilon_\alpha}_{{z_\alpha}\alpha} -
      {\frac{1}{1-\varepsilon_\alpha}} >
      2\varepsilon -
      \frac{\varepsilon_\alpha}{1-\varepsilon_\alpha}\Big)
      \xrightarrow{\alpha\to\infty} 0,
  \end{align*}
   and we are done.
\end{proof}

\noindent
While Lemma~\ref{l:1} dealt with the initial phase in which allele
$\mathpzc B$ is established in a colony, and Lemmata~\ref{l:4} and~\ref{l:2} {are}
good for the final phase of fixation, the following lemma links up these
two phases.

\begin{lemma}[Fast middle phase of local sweep] Let \label{l:3}
  ${\mathcal V} = (V_t)_{t\geq 0}$ be a birth-death process with birth
  rate ${b_k} \geq \alpha k$ and death rate
  ${d_k} \leq \frac 1\rho \binom k 2 + c\alpha^\gamma k$ for some
  $\gamma\in (0,1)$ and $c\geq 0, \rho>0$. Moreover, let $T_n$ be the
  first time when $V_t=n$. {Then there exists a sequence
    $\varepsilon_\alpha\downarrow 0$ with
    $\varepsilon_\alpha > 1/(\log\alpha)$ such that for all
    $\varepsilon'_\alpha\downarrow 0$ with
    $\varepsilon_\alpha'\geq \varepsilon_\alpha$ and for all
    $\varepsilon>0$
  \begin{align}\label{middlephase} 
    \mathbb P\Big( \frac{\alpha}{\log\alpha} 
    T_{(1-\varepsilon_\alpha')2\alpha\rho} > \varepsilon \Big| V_0 = \varepsilon'_\alpha 
    \alpha\Big) \xrightarrow{\alpha\to\infty} 0.
  \end{align}}
\end{lemma}

\begin{proof}
  {We only need to consider the case ${b_k} = \alpha k$ and
    ${d_k} = \frac 1\rho \binom k 2 + c\alpha^\gamma k$, since
    $T_{(1-\varepsilon_\alpha')2\alpha\rho}$ is maximal in this
    case. It suffices to show that for all $\delta>0$ small enough and
    for all $\varepsilon>0$
    \begin{align}\label{middlephase2}
      \mathbb P\Big( \frac{\alpha}{\log\alpha} 
      T_{(1-\delta)2\alpha\rho} > \varepsilon \Big| V_0 = \delta
      \alpha\Big) \xrightarrow{\alpha\to\infty} 0.
    \end{align}}
  We consider the generator of the process
  $(V_{t/\alpha}/\alpha)_{t\geq 0}$, which is given by
  \begin{align*}
    G_\alpha f(y) & = \alpha y(f(y+\tfrac 1\alpha) - f(y)) + 
    \Big(\frac 1\rho \frac{\alpha y(y-\tfrac 1\alpha)}{2} + c\alpha^\gamma
    y\Big)(f(y-\tfrac 1\alpha) - f(y)) \\ &
    \xrightarrow{\alpha\to\infty} y\Big(1 - \frac{y}{2\rho}\Big)
    f'(y).
  \end{align*}
  Using standard arguments, $(V_{t/\alpha}/\alpha)_{t\geq 0}$
  converges weakly as $\alpha \to \infty$ to the solution of the ODE $y' = y(1-y/(2\rho))$,
  and if $V_0 = \delta\alpha$, the limiting process starts in
  $y(0)=\delta$. This solution converges to~2$\rho$ as $t \to \infty$ (from below since
  $\delta<2\rho(1-\delta)$) {and its hitting time of $2\rho(1-\delta)$
    is finite. Consequently,  $T_{(1-\delta)2\alpha\rho} \in O(1/\alpha)$ with high probability as $\alpha \to \infty$, and
    \eqref{middlephase2} follows.}
\end{proof}

\subsection{Proof of Theorem~\ref{T2}}
We are now in the position to prove our main result, Theorem~\ref{T2}.
The proof will be structured in three main parts, corresponding to the
three cases $\mu \in \Theta(\alpha)$, $\mu \in \Theta(\alpha^\gamma)$,
and $\mu = \frac{1}{\log\alpha}$ in Theorem \ref{T2}. Parts 2 and 3
will each be divided into subparts A and B, where A deals with the
special case $d=2$ and B with the general case $d\ge 2$. We feel that
this is instructive, because most of the ideas and tools prepared in
Sections~\ref{S:ASG} and \ref{S:lemmas} come into play already in the
case $d=2$. We will give the arguments in parts 1, 2.A and 3.A in
detail, whereas we restrict to an outline of the main ideas in parts
2.B and 3.B. Parts 2.A and 2.B will additionally be structured into
the cases (i) $\gamma \in (0,1)$ and (ii) $\gamma = 0$.

The proof of all cases is based on an application of
Proposition~\ref{P:main}.  In view of this result, it suffices to
check that the {\em fixation time} $T$ defined in \eqref{defT}
satisfies the properties claimed for $T_{\rm fix}$ in
Theorem~\ref{T2}.  In the sequel, $T_k$ or $T^V_k$ will always denote
the hitting time of $k$ (or of $\lfloor k \rfloor$ if $k$ is not an
integer) of a birth-death process $V$. 
\\
{\bf Convention.} We will use the term {\em with
  high probability} or {\em whp} as a synonym for {\em with
  probability~1 as $\alpha\to\infty$}.  
  
  {Note that in cases~1
  and~2 of Theorem \ref{T2} the right hand sides are deterministic,
  so that we have to show that for all $\varepsilon>0$
  $$\mathbbm{P}\left( \Big| \frac{\alpha}{\log(\alpha)} T - 2 - S_{\mathcal{I}^{\iota,\gamma}}\Big| > \varepsilon \right) 
  \xrightarrow{\alpha\to\infty} 0.$$
} {As a prelude, we state two results which hold in all cases. Recall
  from Proposition \ref{P:main} that the process
  $(\underline L, \underline M)$ starts
  in $(\underline \Pi + \underline e_1, \underline e_1)$. \\
  (a) { Note that
    $\frac{\underline L_0}{\alpha}\Rightarrow 2\underline\rho$. Hence,
    by} Lemma~\ref{l:numberASGnew}, for some large $c$,}
{there exists a sequence $\varepsilon_\alpha \downarrow 0$
  with}
\begin{align}
  \label{eq:671} 
  \mathbb P\Big(L^i_r \in [2\alpha
  \rho_i(1-{\varepsilon^2_\alpha}), 2\alpha \rho_i(1+{\varepsilon^2_\alpha})] \text{ for
  all $i$, for all }0\leq r\leq cd \tfrac{\log\alpha}{\alpha}\Big)
  \xrightarrow{\alpha\to\infty} 1.
\end{align}
(b) Let $\varepsilon_\alpha>0$ be as in (a). For some $\tau > 0$, and
$\tau_\alpha= \tau\frac{\log \alpha}{\alpha}$, consider the event
\begin{align*}
  E_{\tau,i_0}:=  \{L_{\tau_\alpha}^i \in [ 2\alpha\rho_i(1-\varepsilon^2_\alpha),
  2\alpha\rho_i(1+\varepsilon^2_\alpha)], M^i_{\tau_\alpha} \in [2\alpha\rho_i
  (1-4\varepsilon^2_\alpha), 2\alpha \rho_i(1+\varepsilon^2_\alpha)], i=1,...,d,\\ 
  M_{\tau_\alpha}^{i_0} = 2\alpha\rho_i(1-4\varepsilon^2_\alpha) \mbox{ for some } i_0\}.
\end{align*}
Now, consider $L^1 + \cdots + L^d - M^1 - \cdots - M^d$, which is
a birth-death process with birth rate $b_n=\alpha n$ if
$L^i=\ell_i, M^i=m_i, i=1,...,d$ and
$\ell_1 + \cdots + \ell_d - m_1 - \cdots - m_d=n$ and death rate
\begin{align*}
  d_{\underline \ell, \underline m} :=\sum_{i=1}^d
  \frac{1}{\rho_i}\Big( \binom{\ell_i - m_i}{2} + m_i(\ell_i -
  m_i)\Big) = \sum_{i=1}^d \frac{\ell_i + m_i - 1}{2} \frac{\ell_i
  - m_i}{\rho_i}.
\end{align*} 
(Note that the birth and death rates are independent of $\mu$.)  By
Lemma~\ref{l:numberASGnew}, for all $i=1,...,d$, { and
  since the dynamics of $\underline M$ and of $\underline L$
  coincide,} whp on the event $E_{\tau,i_0}$, $M^i$ stays
in
$[2\alpha \rho_i(1 - 2 \varepsilon_\alpha ), 2\alpha \rho_i(1+2
\varepsilon_\alpha )]$
between the times $\tau_\alpha$ and
$\tau_\alpha+cd \tfrac{\log \alpha}{\alpha}$. Moreover, $L^1,...,L^d$
are bounded as stated in \eqref{eq:671}. Hence, we find the bounds
\begin{align*}
  2\alpha (1-\mathcal{O}(\varepsilon_\alpha)) n \leq d_{\underline \ell, \underline
  m} \leq 2\alpha (1+\mathcal{O}(\varepsilon_\alpha)) n.
\end{align*}
By stopping at time $\tau_\alpha+cd \tfrac{\log \alpha}{\alpha}$
with $cd \ge 1$, we can apply Lemma~\ref{l:2} to conclude that,
whp,
\begin{align}\label{eq:672}
  L^1 +  \cdots + L^d  - M^1 - \cdots & - M^d \text{ hits~0 at time
                                        in} 
  \\ \notag & \Big[\tau_\alpha +\frac{\log\alpha}{\alpha}(1-\mathcal{O}(\varepsilon_\alpha)),
              \tau_\alpha+\frac{\log\alpha}{\alpha}(1+\mathcal{O}(\varepsilon_\alpha))\Big].
\end{align}

\noindent {\bf 1.  Case $\dickm{\mu \in \Theta(\alpha)}$.}
Set $\mu = \alpha$ for simplicity. If
$(M^i)_{i=1,...,d}=\underline k = (k_i)_{i=1,...,d}$ and
$\ell := k_1 + \cdots k_d$, the process $M^1 + \cdots + M^d$ has birth
rate $b_{\ell} = \alpha \ell$ and death rate
$\frac{1}{\rho_1} \binom{k_1}{2} + \cdots \frac{1}{\rho_d}
\binom{k_d}{2} \leq \frac{1}{\min \rho_i} \binom{\ell}{2} =: d_\ell$.
{For any $\varepsilon_\alpha\downarrow 0$ with
  $\varepsilon_\alpha > 1/(\log\alpha)$,} we can choose { $c>0$} {
  such that $d_\ell \leq c \varepsilon_\alpha \alpha \ell$ for
  $\ell\leq {\varepsilon_\alpha}\alpha$.} Then, Lemma~\ref{l:1},
Assertion~1, (used for $M^1 + \cdots + M^d$ in place of~$\mathcal V$,
{and with $p=1$}) shows that $M^1 + \cdots + M^d$ hits
${\varepsilon_\alpha}\alpha$ at some time
{$T_{{\varepsilon_\alpha}\alpha}$} in
{$\tfrac{\log(\alpha)}{\alpha}(1 + o(1))$} whp. Arguing as in the
proof of Lemma~\ref{l:3}, we see that for any $\varepsilon >0$ small
enough any potential limit $\underline x = (x_i)_{i=1,...,d}$ of the
processes
$\tfrac{1}{\alpha}({M}^1_{T_{\varepsilon\alpha}+t/\alpha},\dots,
{M}_{T_{\varepsilon\alpha}+t/\alpha}^d)$
as $\alpha\to\infty$ solves for $t>0$ the system of ODEs
\begin{align*}
  \dot x_i & = x_i - \frac{1}{2\rho_i} x_i^2 + \sum_{j=1}^d
             a(j,i)x_j - a(i,j)x_i,
\end{align*}
starting at $t=0$ in some state with
$x_1 {+ \cdots + x_d}=\varepsilon$. These ODEs have equilibrium
$2\underline\rho$ and a state $\underline{x}$ with
$x_i \in [2\rho_i(1-\varepsilon), 2\rho_i(1+\varepsilon)]$,
$i=1,...,d$ and $x_{i_0} = 2\rho_{i_0}(1-\varepsilon)$ for some $i_0$
is reached after time of order $o(\log(\alpha))$. Now we can -- as in
the proof of Lemma~\ref{l:3} -- pass to a sequence
$\varepsilon_\alpha$, such that the conditions from above are
fulfilled and so that at some time
$t \in {\frac{\log\alpha}{\alpha}(1+o(1))}$ a state
$({M}_t^1,\dots,{M}_t^d)$ with
$M^{i}_{t} \in [2\alpha\rho_i (1-\varepsilon_\alpha), 2\alpha
\rho_i(1+\varepsilon_\alpha)]$,
$i=1,...,d$ and $M_{t}^{i_0} = 2\alpha\rho_i(1-\varepsilon_\alpha)$
for some $i_0$ is reached whp. In summary, fixation in the sense of
\eqref{defT} occurs at time
$t \in {\frac{\log\alpha}{\alpha}(2+o(1))}.$

\noindent {\bf 2.A.(i) Case $\dickm{\mu \in \Theta(\alpha^\gamma)}$
  for $\dickm{\gamma\in (0,1)}$, $\dickm{d=2}$:}

In the first steps we will apply Lemma~\ref{l:1} a couple of times,
with suitable choices of the process $\mathcal V$ and $\mathcal W$ in
order to control the ``initial phase'' of the pair of processes
$(M^1, M^2)$. Note that when $(M^1, M^2)$ is in state $(k,\ell)$, then
the process $M^1$ has birth rate $b_k^1=\alpha k + \mu a(2,1) \ell$
and death rate $d_k^1 = \frac 1{\rho_1} \binom k 2 + \mu a(1,2) k$,
whereas the process $M^2$ has birth rate $\alpha \ell + \mu a(1,2) k$
and death rate
$d_\ell^2 = \frac{1}{\rho_2} \binom \ell 2 + \mu a(2,1) \ell$.
Moreover, $M^1 + M^2$ is a birth-death process with birth rate
$\alpha(k+\ell)$ and death rate
$\frac{1}{\rho_1}\binom k 2 + \frac{1}{\rho_2}\binom \ell 2$.
{Let
  $\varepsilon_\alpha, \varepsilon'_\alpha \downarrow 0$ be sequences
  with $\varepsilon_\alpha, {\varepsilon'_\alpha} > 1/(\log\alpha)$.}

First, we are going to establish that $M^1$ hits
${\varepsilon_\alpha}\alpha^p$ by time
\begin{align*}
  T^{M^1}_{{\varepsilon_\alpha}\alpha^p} \in 
  { \frac{\log\alpha}{\alpha}(p+o(1))}
  \mbox{ whp. }
\end{align*}
% with high probability.
On the one hand, this hitting time
$T^{M^1}_{{\varepsilon_\alpha}\alpha^p}$ is stochastically
larger than
$T^{M^1+M^2}_{{\varepsilon_\alpha}\alpha^p}$. For the
latter, Assertion 1 of Lemma~\ref{l:1} (applied to with
$V_t= M_t^1+M^2_t$) ensures that
$T^{M^1+M^2}_{{\varepsilon_\alpha}\alpha^p} \in
{ \frac{\log\alpha}{\alpha}(p+o(1))}$ whp.
% with high probability. 
On the other hand,
$T^{M^1}_{{\varepsilon_\alpha}\alpha^p}$ is smaller than
the hitting time of $\varepsilon_\alpha \alpha^p$ when only
non-(im)migrated lines in $M^1$ are counted. This process of
non-immigrated lines is a birth-death process $\tilde M^1$ with birth
rate $\alpha k$ and death rate
$\frac{1}{\rho_1}\binom k 2 + \mu a(1,2)k$, and therefore fulfills the
conditions of the process $\mathcal V$ of Lemma~\ref{l:1} (with
$\varepsilon_\alpha$ as above). Consequently, also
$T^{\tilde M^1}_{{\varepsilon_\alpha}\alpha^p} \in
{ \frac{\log\alpha}{\alpha} (p+o(1))}$
whp.  Taking these two comparisons together, we find that
$T^{M^1}_{{\varepsilon_\alpha}\alpha^p} \in {
  \frac{\log\alpha}{\alpha}(p+o(1))}$ whp as well.

Second, we will show that the process $M^2$ hits 1 by time
\begin{align}\label{T_1^M_2}
  T_1^{M^2} \in
  { \frac{\log\alpha}{\alpha}(1-\gamma + o(1)) \mbox{ whp. }}
\end{align}
This hitting time is actually the same if we change the birth rate of
$M^1$ (from $\alpha k + \mu\ell$) to $\alpha k$, since $M^2=0$ before
$T_1^{M^2}$. Hence, up to time $T_1^{M^2}$, the process $(M^1, M^2)$
in place of ($\mathcal V, \mathcal W)$ fulfills the conditions of
Lemma~\ref{l:1}, with the $\mu$ appearing there replaced by
$\mu a(1,2)$. This lemma can now be directly applied to obtain \eqref
{T_1^M_2}.
 
Third, we will argue that { there exists a sequence
  $\varepsilon'_\alpha \downarrow 0$ with
  $\varepsilon'_\alpha > 1/(\log\alpha)$ such that} $M^2$ hits
$\varepsilon'_\alpha \alpha^\gamma$ by time
\begin{align}\label{eq:91204}T^{M^2}_{\varepsilon'_\alpha\alpha^\gamma} \in
  { \frac{\log\alpha}{\alpha}(1+o(1))} \, \mbox{ whp. }
\end{align}
% with high probability.
On the one hand, this hitting time is stochastically larger than the
hitting time if migration from {colony} 2 to {colony} 1 is
suppressed. For the thus modified process $(\hat M^1, \hat M^2)$,
$\hat M^1$ has birth rate $\alpha k$ and therefore
$(\hat M^1, \hat M^2)$ fulfills the requirements of Lemma~\ref{l:1}
(for the same combination of $\varepsilon_\alpha, \varepsilon_\alpha'$
as described above) { and
  $T^{\hat M^2}_{\varepsilon_\alpha'\alpha^\gamma} \in
  \frac{\log\alpha}{\alpha}(1+o(1))$ whp}.
On the other hand, this hitting time is stochastically smaller than
the hitting time of ${\varepsilon'_\alpha}\alpha^\gamma$
if only a single migration event from colony 1 to colony 2 happens,
i.e.\ the hitting time
$T^V_{{\varepsilon'_\alpha}\alpha^\gamma}$ of a process
$V$ which starts at time $\frac{\log\alpha}{\alpha}(1-\gamma+o(1))$
with $V=1$, and has birth rate $\alpha \ell$ and death rate
$\frac{1}{\rho_2}\binom \ell 2 + \mu a(2,1)\ell$. By Lemma~\ref{l:1},
Assertion~1, this time is 
$T^V_{{\varepsilon'_\alpha}\alpha^\gamma} \in
\frac{\log\alpha}{\alpha}(1-\gamma+o(1)) +
\frac{\log\alpha}{\alpha}(\gamma+o(1)) =
\frac{\log\alpha}{\alpha}(1+o(1))$
whp and \eqref{eq:91204} follows.  {Moreover,} we have
shown that the pair $(M^1, M^2)$ inherits the properties
\eqref{eq:l21a}, \eqref{eq:l21}, \eqref{eq:l22} proved in
Lemma~\ref{l:1} for the pair $(\mathcal V, \mathcal W)$.

~

In order to go further, we next observe that (as a consequence of the
statement in the first step of this proof, with $p=1$) we have that
${T^{M^1}_{\varepsilon_\alpha\alpha} \in
  \frac{\log\alpha}{\alpha}(1+o(1))}$
{ and}
${T^{M^2}_{\varepsilon_\alpha\alpha^\gamma} \in
  \frac{\log\alpha}{\alpha}(1+o(1))}$
whp.  By Lemma~\ref{l:3} (applied to the process $M^1$) we find a
sequence $\varepsilon_\alpha$ decreasing sufficiently slow such that
it takes time of at most order $o(\log(\alpha)/\alpha)$ until $M^1$
hits $2\alpha\rho_1(1-\varepsilon_\alpha)$. Note, that the sequences
$\varepsilon_\alpha, \varepsilon_\alpha'$ were arbitrary and only had
to fulfill
{$\varepsilon_\alpha, \varepsilon_\alpha' >
  1/(\log\alpha)$},
hence there exist sequences $\varepsilon_\alpha, \varepsilon_\alpha'$
for which all assertions claimed so far are fulfilled. Also in the
following we will if neccessary replace the sequences by slower
converging ones. We note that, due to Corollary~\ref{cor:numberASG},
whp the process $M^1$ will not drop below
$2\alpha\rho_1(1-2\varepsilon_\alpha)$ for the entire period remaining
to fixation with $\varepsilon_\alpha$ again suitably adapted. Now, if
$M^2=\ell$, it has birth rate $\alpha\ell + \mu a(1,2) M^1$, and since
$M^1\leq L^1 \leq 2\alpha \rho_1(1+\varepsilon_\alpha)$, this is
bounded above by $\alpha\ell + c\alpha^{1+\gamma}$ for some constant
$c$. The death rate of $M^2$ is (for the same
$\varepsilon_\alpha, {\varepsilon'_\alpha }$ as above)
$\frac{1}{\rho_2}\binom{\ell}{2} + \mu a(2,1) \ell \leq
c\varepsilon'_\alpha\alpha\ell/2$
for $\ell\leq \varepsilon'_\alpha \alpha$ for some $c>0$. {In
  addition, the sequence $\varepsilon_\alpha$ fulfills the conditions
  on the sequence $c_\alpha$ in Lemma \ref{l:4}.} Hence,
Lemma~\ref{l:4} implies that $M^2$ hits $\varepsilon'_\alpha \alpha$
by time
$T^{M^2}_{\varepsilon'_\alpha \alpha} \in T^{M^2}_{\varepsilon'_\alpha
  \alpha^\gamma} + {{\tfrac{\log(\alpha)}{\alpha}(1-\gamma
    + o(1))}} = {{\tfrac{\log(\alpha)}{\alpha}(2-\gamma +
    o(1))}}$
{when $\varepsilon_\alpha, \varepsilon'_\alpha$ are
  suitably adapted}. Again, $M^2$ rises to
$2\alpha \rho_2(1-2\varepsilon_\alpha)$ by some time of order
{$o(\log(\alpha)/\alpha)$} by Lemma~\ref{l:3} (applied to
the process $M^2$), so by some time in
{$ \tfrac{\log(\alpha)}{\alpha}(2-\gamma + o(1))$}, we
find that $M^1 \geq 2\alpha\rho_1(1-2\varepsilon_\alpha)$ and
$M^2 = 2\alpha\rho_2(1-2\varepsilon_\alpha)$. Now, fixation occurs
after time in
{$T = T^{M^2}_{\varepsilon'_\alpha\alpha} +
  \tfrac{\log(\alpha)}{\alpha}(1+o(1)) =
  \tfrac{\log(\alpha)}{\alpha}(3-\gamma+o(1)) $} by \eqref{eq:672}.

\noindent {\bf 2.A.(ii) Case $\dickm{\mu \in \Theta(1)}$,
  $\dickm{d=2}$:}
Arguing exactly as in Case 2.A.(i), but now with $p=1$, we obtain for
any $\varepsilon_\alpha \downarrow 0$ with
$\varepsilon_\alpha > 1/(\log\alpha)$ that $M^1$ hits
$\varepsilon_\alpha \alpha$ by time
$T^{M^1}_{\varepsilon_\alpha \alpha} \in
\frac{\log\alpha}{\alpha}(1+o(1))$
whp. In addition, by Lemma \ref{l:3}, $M^1$ has increased to
$(1-\varepsilon_\alpha)2\alpha\rho_1$ (maybe after modifying
$\varepsilon_\alpha$) by time
$T^{M^1}_{(1-\varepsilon_\alpha)2\alpha\rho_1} \in
\frac{\log\alpha}{\alpha}(1+o(1))$.

{ For bounding the time $T^{M^2}_{1}$ stochastically from
  below, fix $\varepsilon>0$ and let $\hat M^2$ be as $M^2$ but with
  $\gamma = \varepsilon/2$. Since $T^{\hat M^2}_{1}\leq T^{M^2}_{1}$,
  we find that by Lemma~\ref{l:1}
  \begin{align}\label{T_1}
    \mathbb P\Big(\frac{\alpha}{\log\alpha}T_1^{M^2} - 1 < - \varepsilon\Big) \leq 
    \mathbb P\Big(\frac{\alpha}{\log\alpha}T_1^{\hat M^2} - (1-\varepsilon/2) 
    < - \varepsilon/2\Big) \xrightarrow{\alpha\to\infty}0.
  \end{align}
  For bounding $T^{M^2}_1$ from above, consider migrants only after
  time
  $T^{M^1}_{(1-\varepsilon_\alpha)2\alpha\rho_1} \in
  \frac{\log\alpha}{\alpha}(1+o(1))$.}
Due to Corollary~\ref{cor:numberASG}, whp the process $M^1$ will not
drop below $2\alpha\rho_1(1 - 2\varepsilon_\alpha)$ for the entire
period remaining to fixation. { The expected number of
  migrants between times
  $T^{M^1}_{(1-\varepsilon_\alpha)2\alpha\rho_1}$ and
  $T^{M^1}_{(1-\varepsilon_\alpha)2\alpha\rho_1} +
  \frac{1}{\log\log\alpha}\frac{\log\alpha}{\alpha}$
  is at least
  $ \mu a(1,2) 2\alpha\rho_1(1-2\varepsilon_\alpha)
  \frac{1}{\log\log\alpha}\frac{\log\alpha}{\alpha}
  \xrightarrow{\alpha\to\infty}\infty$
  and hence we have $M^2_t\geq 1$ for some
  $t\in\frac{\log\alpha}{\alpha}(1+o(1))$ whp.} Together with
\eqref{T_1} this says that
$T^{M^2}_1 \in \frac{\log\alpha}{\alpha}(1+o(1))$ whp. We can now
apply Lemma~\ref{l:4} (with $\gamma:=0$ and $p=1$) to infer that the
process ${M}^2$ reaches $\varepsilon_\alpha \alpha$ in
$\frac{\log\alpha}{\alpha}(2+o(1))$ whp.  From Lemma~\ref{l:3}, we
hence find some $t\in 2\frac{\log\alpha}{\alpha}(1+o(1))$ for which
${M}_t^1 \geq 2\alpha \rho_1(1-\varepsilon_\alpha), {M}_t^2 = 2\alpha
\rho_2(1-2\varepsilon_\alpha)$.
Then by \eqref{eq:672}, fixation occurs at time in 
$\frac{\log\alpha}{\alpha}(3+o(1))$.

\begin{table}
  \centering
  {
    \begin{tabular}{cccccc}
      $t$ & $M_{t\log\alpha/\alpha}^\iota$ &   $M_{t\log\alpha/\alpha}^{i}, i\in 
                                             {D}_1$ 
      &  $M_{t\log\alpha/\alpha}^{i}, i\in 
        {D}_2$ &  $M_{t\log\alpha/\alpha}^{i}, i\in
                 {D}_3$
      & $\cdots $\\
      $0$ & $1$ & $0$ & $0$ & $0$ & $\cdots$ \\
      $1-\gamma$ & $\Theta(\alpha^{1-\gamma})$ & $1$ & $0$ & $0$ &$\cdots$ \\
      $2(1-\gamma)$ & $\Theta(\alpha^{1\wedge (2(1-\gamma))})$ & $\Theta(\alpha^{1-\gamma})$ & $1$ & $0$ & $\cdots$ \\
      $3(1-\gamma)$ & $\Theta(\alpha^{1\wedge (3(1-\gamma))})$ & $\Theta(\alpha^{1\wedge(2(1-\gamma))})$ & $\Theta(\alpha^{1-\gamma})$ & $1$ & $\cdots$ \\
    \end{tabular}
    \caption{\label{tab1} The table gives (approximate) times and orders
      of magnitude in the case 2.B.(i) ($\gamma \in (0,1)$, $d\geq 2$); 
      see text for more explanation.}}
\end{table}

\noindent {\bf 2.B.(i) Case $\dickm{\mu \in \Theta(\alpha^\gamma)}$
  for $\dickm{\gamma\in (0,1)}$, $\dickm{d \ge 2}$:}
Set $[d]:=\{1,...,d\}$ and, for $s=0,1,\ldots, \Delta_\iota$, let
$B_s(\iota)$ be the set of vertices in $[d]$ which can be reached from
$\iota$ by at most $s$ steps (cf. Definition \ref{def:auxProc}). We
partition $[d] = \bigcup_{s=0}^{\Delta_\iota} {D}_s$ into
$ {D}_0 := \{\iota\}$ and
$ {D}_{s} := B_{s}(\iota) \setminus B_{s-1}(\iota)$,
$s=1,\ldots, \Delta_\iota.$ Arguing similarly as in part 2.A (i), now
based on Corollary \ref{r:1}, we obtain the analogue of \eqref
{T_1^M_2}, simultaneoulsly for all $i\in D_1$. In the language of the
epidemic {process} $\mathcal I^{\iota, \gamma}$ this means that all
colonies $i \in {D}_1$ are infected at times
$$T_1^{M^{i}} \in \frac{\log(\alpha)}{\alpha}(1+o(1))\, \mbox{ whp; }$$ 
see also Table \ref{tab1} for orientation.

Let us concentrate now on a colony $m \in {D}_2$ and set
$ {D}_1^{(m)} := \{j\in D_1: a(j,m)>0\}$.  From the second assertion
of Corollary \ref{r:1} we obtain that { there exists a
  sequence $\varepsilon_\alpha \downarrow 0$}, such that for all
colonies $i \in {D}_1^{(m)}$,
$$T_{\varepsilon_\alpha \alpha^{1-\gamma}}^{M^{i}}  \in 
\frac{\log(\alpha)}{\alpha}(2(1-\gamma) + o(1))\, \mbox{ whp. }$$
Hence, all $i\in {D}_1^{(m)}$ will infect $m$ by this time whp.
Equation \eqref{eq:l22}, translated to the pairs $(M^{\iota}, M^{i})$
for $i\in {D}_1$ in a similar way as done in part 2.A.(i) for the pair
$(M^{1}, M^{2})$, implies that migration from the founder colony
$\iota$ does not speed up (on the
$\frac{\log(\alpha)}{\alpha}$-timescale) the processes $M^{i}$ till
they reach $\varepsilon_\alpha\alpha^{\gamma}$ for an appropriate
sequence $\varepsilon_\alpha$; in fact, during this period the rate of
growth of $M^{i}$ is that of a branching process with Malthusian
parameter $\alpha$. Lemma \ref{l:4} carries this assertion further:
Since $M^{\iota} \leq 2 \alpha \rho_\iota (1+ 2 \varepsilon_\alpha)$,
migration from colony $\iota$ to colonies in $D_1$ is bounded by
$c \alpha^{1+\gamma}$ for an appropriate constant $c$. In addition,
the sequence $\varepsilon_\alpha$ fulfills the conditions of the
sequence $c_\alpha$ in Lemma \ref{l:4}. Consequently, the process
$M^{i}$ continues to grow like a branching process with rate $\alpha$
by Lemma \ref{l:4} and for $m\in D_2$ the assertions of
Lemma \ref{l:1} are fulfilled with $\sum_{j \in {D}_1^{(m)}} M^{j}$
playing the role of $\mathcal V$ and $M^m$ playing the role of
$\mathcal W$, see also Table \ref{tab1}.  It follows that
$T_1^{M^m} \in \frac{\log(\alpha)}{\alpha}(2(1-\gamma) + o(1))$ whp.

Repeating these arguments one finds that all colonies are, whp,
infected by a time in
$\frac{\log(\alpha)}{\alpha}(\Delta_{\iota}(1-\gamma) + o(1))$, with
$\Delta_{\iota}$ as in Definition \ref{def:auxProc}. Finally, arguing
as in part 2.A.(i), it takes an additional time in
$ \frac{\log(\alpha)}{\alpha}(2+o(1))$ until fixation occurs. This
sums up to a total time in
$\frac{\log(\alpha)}{\alpha}((2 + S_{\mathcal{I}^{\iota, \gamma}}) +
o(1))$
whp, with $S_{\mathcal{I}^{\iota, \gamma}}= (1-\gamma)\Delta_{\iota}$
according to Definition \ref{def:auxProc}.
\\

\noindent {\bf 2.B.(ii) Case $\dickm{\mu \in \Theta(1)}$, $\dickm{d}\ge 2$
:} 
We will use the same notation as in Case 2.B.(i). Let
$\varepsilon_\alpha>0$.  The argument from Case 2.A.(ii) works for all
colonies $i\in {D}_1$ which are distance 1 apart from colony
$\iota$. Hence, whp, at some time in
$\frac{\log\alpha}{\alpha}(2+o(1))$, there is
$M^i \in [2\alpha \rho_i(1-2\varepsilon_\alpha),2\alpha
\rho_i(1+2\varepsilon_\alpha)]$
{for $i=\iota$ and $i\in {D}_1$}. Similarly, each colony $m \in {D}_2$
has $M^m=1$ (and in this sense is \emph{infected}) within an
additional time interval of length $o(\tfrac{\log(\alpha)}{\alpha})$,
and then $M^m$ increases to $2\alpha \rho_m(1 - 2\varepsilon_\alpha)$
after a duration in $\frac{\log\alpha}{\alpha}(1+o(1))$. This
procedure is iterated, and all colonies are infected by a time in
$\frac{\log(\alpha)}{\alpha}(\Delta_{\iota}+1 + o(1))$ whp. Then, from
\eqref{eq:672}, fixation occurs at time in
$\frac{\log(\alpha)}{\alpha}(\Delta_{\iota}+2 + o(1))$ whp, giving the
result.

~

\noindent {\bf 3.A. Case $\dickm{\mu = \frac{1}{\log\alpha}}$,
  $\dickm{d = 2}$:}
The main step in this case is to show that
\begin{align} \label{eq:1X}
  \frac{\alpha}{\log\alpha} T_1^{M^2} \xRightarrow{\alpha\to\infty}
  1+X, \text{ where }X \sim \exp(2\rho_1 a(1,2)).
\end{align}
By the same arguments as in Case 2.A.(ii), { for any}
sequence $\varepsilon_\alpha \downarrow 0$ with
$\varepsilon_\alpha > 1/(\log\alpha)$, we have
$T^{M^1}_{(1-\varepsilon_\alpha\alpha)2\alpha\rho_1} =
\frac{\log\alpha}{\alpha}(1+o(1))$
whp. In addition, $M^2=0$ before $T^{M^1}_{\varepsilon_\alpha\alpha}$
whp, as we can estimate the number of migrants from colony 1 to colony
2 by
$\varepsilon_\alpha\alpha \mu a(1,2)
\frac{\log\alpha}{\alpha}(1+o(1))\xrightarrow{\alpha\to \infty} 0 $.
Here, the expected number of migrants from colony~1 to colony~2 during
$[T^{M^1}_{\varepsilon_\alpha\alpha},T^{M^1}_{(1-\varepsilon_\alpha\alpha)2\alpha\rho_1}]$
is bounded from above by
$2\alpha\rho_1(1-\varepsilon_\alpha) \mu a(1,2) o\big(
\frac{\log\alpha}{\alpha}\big) \xrightarrow{\alpha\to\infty} 0$
since
$T^{M^1}_{(1-\varepsilon_\alpha\alpha)2\alpha\rho_1} -
T^{M^1}_{\varepsilon_\alpha\alpha} = o\big(
\frac{\log\alpha}{\alpha}\big)$
by Lemma~\ref{l:3} with a possibly slower decreasing sequence
$\varepsilon_\alpha$. Hence, we have $M^2=0$ before
$T^{M^1}_{(1-\varepsilon_\alpha\alpha)2\alpha\rho_1}$ whp as well. By
Corollary~\ref{cor:numberASG}, we have that
${M}_t^1 \in [2\alpha \rho_1 (1-2\varepsilon_\alpha), 2\alpha
\rho_1(1+2\varepsilon_\alpha)]$
after $T^{M^1}_{(1-\varepsilon_\alpha\alpha)2\alpha\rho_1}$ until
fixation. Hence, for all $x>0$, {
  \begin{align*}
    \lim_{\alpha\to\infty}\mathbb P\Big( \frac{\alpha}{\log\alpha}
    T^{M^2}_1 - 1 > x \Big) 
    & =
      \lim_{\alpha\to\infty}\mathbb E\Big[\exp\Big( - 
      \int_{T^{M^1}_{(1-\varepsilon_\alpha\alpha)2\alpha\rho_1}}^{\frac{\log\alpha}{\alpha}(1 +
      x)} \mu a(1,2){M}_{t}^1 dt\Big) \Big]\\
    & = \lim_{\alpha\to\infty} \exp\Big( -
      \int_{\frac{\log\alpha}\alpha}^{\frac{\log\alpha}{\alpha}(1+x)}
      \frac{2\alpha\rho_1 a(1,2)}{\log\alpha} dt\Big) \\
    & = e^{-2\rho_1 a(1,2) x},
\end{align*}
which  gives
\eqref{eq:1X}. }
Analogously to the other cases we find $c>0$ and a sequence
$\varepsilon_\alpha \downarrow 0$ with
$\varepsilon_\alpha > 1/(\log\alpha)$, such that if
$(M^1,M^2)=(k,\ell)$, $M^2$ is a birth-death process with birth rate
$b_\ell = \alpha \ell + \mu a(1,2) k \leq \alpha\ell + c
\alpha/\log\alpha$
and death rate
$d_\ell = \frac{1}{\rho_2} \binom \ell 2 + \mu a(2,1) \ell \leq
\varepsilon_\alpha \alpha\ell$
for $\ell\leq \varepsilon_\alpha \alpha$. So, we can apply
Lemma~\ref{l:4} (for $\gamma=0$) in order to see that
$T^{M^2}_{\varepsilon_\alpha \alpha}$ occurs after duration in
$\frac{\log\alpha}{\alpha}(1+o(1))$ Then, using Lemma~\ref{l:3}, we
see that
$T^{M^2}_{2\alpha\rho_2(1-2\varepsilon_\alpha)} \in
\frac{\log\alpha}{\alpha}(2+X+o(1))$
for some $X$ distributed as above.  Then, using~\eqref{eq:672},
fixation occurs at time in $\frac{\log\alpha}{\alpha}(3+X+o(1))$, as
claimed.

~

\noindent {\bf 3.B. Case $\dickm{\mu = \frac{1}{\log\alpha}}$,
  $\dickm{d}\ge 2$ :} By the same arguments as in Case 3.A at a time
$t \in \frac{\log\alpha}{\alpha}(1+o(1))$, colony~$\iota$ in the
process ${\mathcal J^{\iota}}$ from Definition~\ref{def:auxProc}
switches from being infected to being infectious. From here on, each
colony $i \in {D}_1$ can be \emph{infected} by a migrant from
colony~$\iota$ at rate
$2 \rho_1 a({\iota},i) \alpha/(\log\alpha)$,
{i.e. at rate $2 \rho_1 a({\iota},i)$ on the
  $\frac{\log\alpha}{\alpha}$-timescale}.  After $i$ is infected,
$M^i$ increases until there are of the order $\alpha$ particles, which
happens after time of duration $\frac{\log\alpha}{\alpha} (1+o(1))$.
Then, the colony becomes \emph{infectious}, meaning that other
colonies can be infected from that colony. More precisely, if colony
$i$ is infectious and colony $j$ satisfies $a(i,j) >0$, then, as long as
$M^j = 0$, a migrant from $M^i$ arrives in colony $j$ after an
exponential time with rate $2 \rho_i a(i,j)$ { on the
  $\frac{\log\alpha}{\alpha}$-timescale.} Continuing in this way, the
waiting time until all colonies are infectious is
$\frac{\log(\alpha)}{\alpha}(S_{\mathcal J^{\iota}} + o(1))$ in the
approximating process ${\mathcal J^{\iota}}$. At this time, each
colony~$i$ has $M^i\geq 2\alpha \rho_i(1-\varepsilon_\alpha)$,
$i=1,\dots,d$. As in the other cases we conclude from \eqref{eq:672}
that after an additional time of duration in
$\frac{\log\alpha}{\alpha}(1+o(1))$, fixation has occurred.

\subsubsection*{Acknowledgments}
We thank Wolfgang Stephan for posing the question initiating this
work, Tom Kurtz for help and stimulating discussions related to the
proof of Lemma~\ref{l:numberASGnew}, Jeff Jensen for valuable comments
on the manuscript and two anonymous referees for their careful and
critical reading which helped us to improve the presentation and to
correct errors from previous versions of the manuscript. This research
was supported by the DFG through the research unit 1078 and the
priority program 1590, and in particular through grant GR 876/16-1 to
AG, Pf-672/3-1 and Pf-672/6-1 to PP, and Wa-967/4-1 to AW.

\end{document}